\documentclass{amsart}

\usepackage{enumerate}
\usepackage{amsmath}
\usepackage{amssymb}
\usepackage{amsthm}

\usepackage{math}
\usepackage{thm}
\usepackage{ref}

\title
	[Energy Preserving Boundary Conditions]
	{On a Class of \\ Energy Preserving Boundary Conditions for \\ Incompressible Newtonian Flows}

\author
	[Dieter Bothe]
	{Dieter Bothe}
	
\address
	{Center of Smart Interfaces and\newline\indent
	 Department of Mathematics\newline\indent
	 Technische Universit{\"a}t Darmstadt, \newline\indent
	 Petersenstr.~32, D-64287 Darmstadt, Germany}
	
\email
	{bothe@csi.tu-darmstadt.de}

\author
	[Matthias K{\"o}hne]
	{Matthias K{\"o}hne}

\address
	{Center of Smart Interfaces\newline\indent
	 Technische Universit{\"a}t Darmstadt, \newline\indent
	 Schlossgartenstr.~7, D-64289 Darmstadt, Germany}
	
\email
	{koehne@csi.tu-darmstadt.de}

\author
	[Jan Pr{\"u}ss]
	{Jan Pr{\"u}ss}

\address
	{Institut f{\"u}r Mathematik \newline\indent
	 Martin-Luther-Universit{\"a}t Halle-Wittenberg, \newline\indent
	 Theodor-Lieser-Str.~5, D-60120 Halle, Germany}

\email
	{jan.pruess@mathematik.uni-halle.de}

\keywords
	{energy preserving boundary condition,
	 incompressible Newtonian fluid,
	 Navier-Stokes equations,
	 Stokes equations,
	 local-in-time well-posedness,
	 maximal regularity,
	 initial boundary value problem}

\subjclass
	[2010]
	{Primary: 35Q30; Secondary: 35S30, 76D03, 76D07}

\date
	{\today}

\begin{document}
\begin{abstract}
	We derive a class of energy preserving boundary conditions for incompressible Newtonian flows
	and prove local-in-time well-posedness of the resulting initial boundary value problems,
	i.\,e.\ the Navier-Stokes equations complemented by one of the derived boundary conditions,
	in an $L_p$-setting in domains $\Omega \subseteq \bR^n$,
	which are either bounded or unbounded with almost flat boundary of class $C^{3-}$.
	The results are based on maximal regularity properties of the underlying linearisations,
	which are also established in the above setting.
\end{abstract}
\renewcommand{\baselinestretch}{1.125}
\normalsize
\maketitle

\section*{Introduction}
We consider the flow of an incompressible Newtonian fluid
with constant density $\rho > 0$ and constant viscosity $\mu > 0$
in a domain $\Omega \subseteq \bR^n$,
which is governed by the incompressible Navier-Stokes equations
\begin{equation*}
	\label{eqn:n}\tag*{${(\mbox{N})}^{a, \Omega, \cB}_{f, u_0}$}
	\begin{array}{rclll}
		\rho \partial_t u + \mbox{div}(\rho u \otimes u - S) & = & \rho f & \quad \mbox{in} & (0,\,a) \times \Omega,          \\[0.5em]
		                                       \mbox{div}\,u & = & 0      & \quad \mbox{in} & (0,\,a) \times \Omega,          \\[0.5em]
		                                          \cB(u,\,p) & = & 0      & \quad \mbox{on} & (0,\,a) \times \partial \Omega, \\[0.5em]
		                                                u(0) & = & u_0    & \quad \mbox{in} &                \Omega,          \\[0.5em]
	\end{array}
\end{equation*}
where $u$ denotes the velocity field,
$S = 2\mu D - p I$ is the (total) stress tensor with pressure $p$ and
$D = \frac{1}{2}(\nabla u + \nabla u^{\sf{T}})$  the rate of deformation tensor,
$f$ denotes the mass specific density of the driving force, e.\,g.\ gravity,
and $u_0$ is the initial velocity.
The domain $\Omega$ is assumed to be a half space, a bent half space,
or a bounded domain with boundary $\Gamma := \partial \Omega$ of class $C^{3-}$.
The flow is subject to a local boundary condition prescribed by the linear operator $\cB$,
which may depend on the velocity field $u$ and the pressure $p$.
If, for a specific flow, the ratio between the fluid viscosity $\mu$ and its density $\rho$ is sufficiently large, then
the inertia term $\mbox{div}(\rho u \otimes u)$ may be neglected and the flow is essentially governed by the incompressible Stokes equations
\begin{equation*}
	\label{eqn:s}\tag*{${(\mbox{S})}^{a, \Omega, \cB}_{f, u_0}$}
	\begin{array}{rclll}
		\rho \partial_t u - \mbox{div}\,S & = & \rho f & \quad \mbox{in} & (0,\,a) \times \Omega,          \\[0.5em]
		                    \mbox{div}\,u & = & 0      & \quad \mbox{in} & (0,\,a) \times \Omega,          \\[0.5em]
		                       \cB(u,\,p) & = & 0      & \quad \mbox{on} & (0,\,a) \times \partial \Omega, \\[0.5em]
		                             u(0) & = & u_0    & \quad \mbox{in} &                \Omega,          \\[0.5em]
	\end{array}
\end{equation*}
which form a linear system of equations.

We investigate the solvability of \ref{eqn:n} (locally in time) and of \ref{eqn:s} under diverse boundary conditions and
in the $L_p$-setting with $p>1$ sufficiently large in the non-linear Navier-Stokes case.
For standard boundary conditions at impermeable walls like the no-slip, the free-slip or the more general Navier boundary condition,
but also for the Neumann boundary condition, the (local-in-time) well-posedness of these systems is well-known.
Let us note that even for generalized Newtonian flows, the corresponding Stokes system has the property of $L_p$-maximal regularity
as was recently shown in \cite{Bothe-Pruess:Non-Newtonian}. In particular, the Navier-Stokes system for a generalized Newtonian fluid with one
of the standard boundary conditions above is locally-in-time well-posed.

On the other hand, in computational fluid dynamics the problem of formulation of appropriate boundary conditions especially at outflow boundaries appears.
For this purpose, there is a multitude of so-called artificial boundary conditions (ABCs) in use, cf.\ \cite{Gresho:Fluid-Mechanics}, but a rational derivation of those
is often missing. Even more importantly, not much is rigorously known about strong solvability of the corresponding Stokes of Navier-Stokes system.
In fact, some of these ABCs are actually known to lead to ill-posed problems.

The first aim of the present paper therefore is to give a rational derivation of a large class of non-standard boundary conditions,
containing several different ABCs, by introducing the concept of energy preserving boundary conditions.
Secondly, we will establish an $L_p$-theory of the Stokes system under (standard and) non-standard boundary conditions and apply this to
the corresponding Navier-Stokes systems, obtaining the local-in-time well-posedness of the underlying incompressible Newtonian flows.
Let us note that $L_p$-maximal regularity of the Stokes system under the diverse boundary conditions does not immediately follow from known results as,
e.g., provided in \cite{Denk-Hieber-Pruess:Maximal-Regularity, Denk-Hieber-Pruess:Maximal-Regularity-Inhomogeneous}.
The reasons for this are that (i) the known relevant maximal regularity theory only applies to parabolic systems, but not
to the Stokes system and, (ii) the known version of the Lopatinskii-Shapiro condition requires all components
of the system of boundary conditions to be of the same order.
Note that, while mixed order parabolic systems are a field of active research, cf.\ \cite{Denk-Seiler:Maximal-Regularity},
no complete generalization of the Lopatinskii-Shapiro condition for mixed order initial boundary value problems seems yet to be available.
Therefore, a main part of the present paper is to develop the $L_p$-maximal regularity
theory of the Stokes problem with various boundary conditions.
This will be done in such generality that the above mentioned standard boundary conditions are also covered,
so that the present paper provides a rather complete picture.

To establish $L_p$-maximal regularity we employ a localization procedure.
As a starting point, we consider the prototype geometry for initial boundary value problems, a halfspace.
Using $L_p$-maximal regularity for parabolic systems with divergence type boundary conditions,
which are rigorously derived in appendices,
we are able to derive the boundary symbols of the halfspace problems and employ state of the art methods such as the $\cH^\infty$-calculus to obtain $L_p$-maximal regularity.
The case of a bounded domain with sufficiently smooth boundary is then reduced to finitely many (bent) halfspace problems,
where the overall approach generalizes the well-known techniques for parabolic problems to the Stokes equations.
This way, we obtain a new generic localization procedure, which is applicable for all boundary conditions under consideration.
However, the above procedure relies on the ability to reduce a fully inhomogeneous Stokes system to the case of a Stokes flow driven
only by the boundary conditions.
Therefore, we also introduce a new splitting scheme, which is applicable for all boundary conditions and all geometries under consideration.
The overall approach may also be applied to study incompressible Newtonian flows under 
dynamic energy preserving boundary conditions of relaxation type or under
further non energy preserving boundary conditions; both cases are not contained in the present paper.

The paper is organized as follows.
We start with a rational derivation of a large class of non-standard boundary conditions in \Secref{bcs}.
The main results are stated in \Secref{main}, which also includes an overview on known results from the literature.
The above announced splitting scheme is presented in \Secref{splitting},
which also contains a brief overview of the (weak) Dirichlet and Neumann problems for the Laplacian in several geometries,
since these are the main tools to achieve the desired reduction of the Stokes system.
The proof of our main theorem concerning the $L_p$-maximal regularity of the Stokes equations
is carried out in Sections \ref{sec:halfspace}, \ref{sec:bent-halfspace} and \ref{sec:domain} for the case of a halfspace, a bent halfspace and a bounded smooth domain, respectively.
The non-linear problem is treated in \Secref{n}.
The paper closes with two appendices dealing with parabolic systems with divergence type boundary conditions,
which play a key role for the splitting scheme and the treatment of the halfspace problems.

\section{A Class of Energy Preserving Boundary Conditions}\seclabel{bcs}
The derivation of suitable boundary
conditions for flow problems is not obvious and depends on the physics which is to be modeled there.
However, one set of equations strictly applies at the boundary $\Gamma=\partial \Omega$.
These are the balance equations at $\Gamma$, also termed \emph{transmission conditions}.
For the relevant quantities mass and momentum, these transmission conditions read
\begin{equation}
\label{trans-mass}
[\![ \rho (u -u_\Gamma ) ]\!] \cdot \nu = r_\Gamma \quad \mbox{ on } (0,\,a) \times \Gamma,
\end{equation}
\begin{equation}
\label{trans-mom}
[\![ \rho u \otimes (u -u_\Gamma ) - S ]\!] \cdot \nu = f_\Gamma \quad \mbox{ on } (0,\,a) \times \Gamma,
\end{equation}
where $[\![ \phi ]\!]$ denotes the jump of a quantity $\phi$ across the interface if $\Gamma$ is passed in the direction
opposite to the interface normal $\nu$. In these equations $u$ denotes the velocity and $u_\Gamma$ is the interface velocity,
which is zero here, since we consider fixed domains for which $\Gamma$ is independent of time. As introduced above,
$S$ denotes the total stress tensor and the right-hand sides in \eqref{trans-mass} and \eqref{trans-mom}
model sources and sinks due to processes on the boundary, e.g.\ due to transfer of mass from the bulk to
the boundary (so-called adsorption). In the situations we have in mind, $r_\Gamma =0$ but $f_\Gamma$ may be non-zero,
motivated by free liquid surfaces with surface tension.

Now, if $\Gamma$ is a physical boundary, by which we mean that it separates the domain $\Omega$ from a different external
bulk phase $\Omega_{\rm ext}$, then physically sound boundary conditions are sometimes available from knowledge about
the velocity field in $\Omega_{\rm ext}$. Usually, this will provide one condition and the remaining ones need to be modeled
based on constitutive assumptions. For instance, if the exterior phase is a rigid solid, then $u_{\rm ext}=0$ and
\eqref{trans-mass} implies
\begin{equation}
\label{impermeable}
u  \cdot \nu = 0 \quad \mbox{ on } (0,\,a) \times \Gamma,
\end{equation}
expressing the fact that the ``wall'' $\Gamma = \partial \Omega$ is impermeable.

To obtain the required additional boundary conditions, consideration of the kinetic energy is helpful.
The incompressible flows under consideration are governed by
\begin{equation}
		\rho \partial_t u + \mbox{div}(\rho u \otimes u - S) = \rho f,
			\quad \quad \mbox{div}\,u = 0
			\quad \quad \mbox{in}\ (0,\,a) \times \Omega.
\end{equation}
Hence the total kinetic energy satisfies
\begin{equation}
  	\frac{\mbox{d}}{\mbox{d}t} \int_\Omega {\textstyle{\frac{1}{2}}} \rho {|u|}^2\,\mbox{d}x
	+ \int_{\Gamma} {\textstyle{\frac{1}{2}}} \rho {|u|}^2\, (u \cdot \nu) \,\mbox{d}\sigma
	+ 2\mu\int_\Omega {|D|}^2\,\mbox{d}x
	- \int_{\Gamma} u \cdot S \nu\,\mbox{d}\sigma
	= 0,
\end{equation}
where $\nu: \Gamma \longrightarrow \bR^n$ denotes the outer unit normal field of $\Omega$.
According to this energy balance, the rate of change of total kinetic energy plus the loss of kinetic energy due to internal friction
equals the power
\begin{equation}
\label{power-NS}
\pi_{\rm NS} =
	\int_{\Gamma} \left( u \cdot S \nu - {\textstyle{\frac{1}{2}}} \rho {|u|}^2\, (u \cdot \nu) \right)\,\mbox{d}\sigma,
\end{equation}
which changes the total amount of kinetic energy of the system via the boundary.
A boundary condition which causes this contribution via the boundary to vanish may therefore be considered
as an \emph{energy preserving boundary condition} for incompressible Newtonian flows.

Consider first the case when $\Gamma$ is impermeable, i.e.\ \eqref{impermeable} holds. Then
\[
\pi_{\rm NS} = \int_{\Gamma} u \cdot S \nu\,\mbox{d}\sigma
\]
and \eqref{impermeable} implies
\[
u \cdot S \nu = P_\Gamma u \cdot P_\Gamma S \nu,
\]
where $P_\Gamma = I-\nu \otimes \nu$.
Note that at fixed $x\in \Gamma$, the map $P_\Gamma (x)$ is the orthogonal projection onto the plane tangent to $\Gamma$ at $x$.
Evidently, we obtain $\pi_{\rm NS} =0$ if the \emph{no-slip boundary condition}
\begin{equation}
\label{no-slip}
P_\Gamma u  = 0 \quad \mbox{ on } (0,\,a) \times \Gamma
\end{equation}
holds. Together with \eqref{impermeable}, this yields the homogeneous \emph{Dirichlet boundary condition}
\begin{equation}
\label{eqn:nbc-uu}
 u  = 0 \quad \mbox{ on } (0,\,a) \times \Gamma
\end{equation}
which has already been introduced by Stokes; see \cite{Stokes:Navier-Stokes}.
Let us note in passing that the no-slip condition at fixed walls is much under debate recently,
especially in the context of flow through micro-channels;
cf.\ \cite{Chen-Doolen:Lattice-Boltzmann}. Its inhomogeneous version, i.e.\
\begin{equation}
\label{Dirichlet-inhom}
u  = u_{\rm D} \quad \mbox{ on } (0,\,a) \times \Gamma,
\end{equation}
is used in numerical simulations to model inflow boundaries, where the velocity profile needs to be known within
sufficient accuracy.

Another obvious way to have $\pi_{\rm NS} =0$, assuming \eqref{impermeable} to hold, is
the \emph{free-slip (or, perfect-slip) boundary condition}
\begin{equation}
\label{slip}
P_\Gamma S \nu  = 0 \quad \mbox{ on } (0,\,a) \times \Gamma.
\end{equation}
The complete boundary condition then reads
\begin{equation}
\label{eqn:nbc-du}
u \cdot \nu = 0 \quad \mbox{ and } \quad
2 \mu P_\Gamma D \nu  = 0 \quad \mbox{ on } (0,\,a) \times \Gamma.
\end{equation}
Note that at points where $\Gamma$ is locally planar, a simplification is possible. Indeed, in this case
the tangential derivatives of the velocity have no normal component due to \eqref{impermeable}, hence
$2 \mu P_\Gamma D \nu$ reduces to $\mu P_\Gamma \partial_{\nu} u$ and, hence, the complete boundary condition becomes
\begin{equation}
\label{slip-planar}
u \cdot \nu = 0 \quad \mbox{ and } \quad \mu P_\Gamma \partial_{\nu} u =0.
\end{equation}
In the applied literature this is often written as
\[
u \cdot \nu = 0 \quad \mbox{ and } \quad \partial_{\nu} u_{||} =0,
\]
which this has to be understood in the right sense: at a fixed point $x \in \Gamma$ it means $\partial_{\nu} (u \cdot \tau) = 0$
at $x$ for every fixed tangential vector $\tau \perp \nu (x)$.
Since $\partial_\nu (u \cdot \tau) = \partial_\nu u \cdot \tau$ for fixed $\tau$,
this is indeed equivalent.
Note that for non-planar $\Gamma$, the condition \eqref{slip-planar} differs from \eqnref{nbc-du} by additional curvature-related terms.

The free-slip and the no-slip conditions are the two extreme cases $\alpha =0$
and $\alpha \to \infty$, respectively, of the more general \emph{Navier boundary condition}
\begin{equation}
\label{Navier-BC}
P_\Gamma u + \alpha \, P_\Gamma S \nu = 0 \quad \mbox{ with } \alpha >0 \quad \mbox{ on } (0,\,a) \times \Gamma.
\end{equation}
This condition, which is due to Navier \cite{Navier:Navier-Stokes}, implies $\pi_{\rm NS} \leq 0$, but $\pi_{\rm NS}$ does not need to vanish.
Note that, mathematically, the Navier condition is a lower order perturbation of the free-slip condition.
Therefore, it will not play a role later on.
There is a more general \emph{Navier-type partial-slip condition}, which has recently been shown in \cite{Bucur-Feireisl-Necasova:Friction-Driven-BC}
to be in some sense the most general boundary condition which is possible for flows of incompressible Newtonian fluids at impermeable walls.

Another standard boundary condition is motivated by \eqref{trans-mom} as follows.
If $\Omega_{\rm ext}$ is a gas phase with negligible gas viscosity, then a reasonable simplification of
the momentum transmission condition is
\[
- S \nu = p_{\rm ext}\, \nu + f_\Gamma \quad \mbox{ on } (0,\,a) \times \Gamma
\]
with $p_{\rm ext}$ the external pressure. The associated homogeneous boundary condition is the homo\-geneous
\emph{Neumann boundary condition}
\begin{equation}
\label{eqn:nbc-ds}
S  \nu = 0 \quad \mbox{ on } (0,\,a) \times \Gamma,
\end{equation}
which models a free surface with external pressure set to zero and vanishing surface tension ($f_\Gamma =0$).
Under this condition, energy need not be preserved for the Navier-Stokes system, but for the Stokes equations.
Indeed, we have $\pi_S =0$, where
\begin{equation}
\label{power-S}
\pi_{\rm S} =
	\int_{\Gamma} u \cdot S \nu\,\mbox{d}\sigma.
\end{equation}

To obtain further boundary conditions, which are  energy preserving for the Stokes system,
we split both $u$ and $S \nu$ into normal and tangential parts and obtain
\begin{equation}
\pi_{\rm S} =
	\int_{\Gamma} \Big( (u \cdot \nu)\,(S\nu \cdot \nu) + P_\Gamma u \cdot P_\Gamma S \nu \Big) \,\mbox{d}\sigma.
\end{equation}
Hence another admissible combination is \eqref{no-slip} together with $S \nu \cdot \nu = 0$, i.e.\
\begin{equation}
\label{eqn:nbc-us}
P_\Gamma u = 0 \quad \mbox{ and } \quad 2 \mu\,\partial_\nu u \cdot \nu - p = 0 \quad \mbox{ on } (0,\,a) \times \Gamma.
\end{equation}
In the applied literature it is usually written in the form
\begin{equation}
\label{outflow1}
P_\Gamma u = 0 \quad \mbox{ and } \quad 2 \mu \partial_\nu (u \cdot \nu) - p = 0 \quad \mbox{ on } (0,\,a) \times \Gamma
\end{equation}
and is employed as an \emph{outflow boundary condition}; cf.\ the remark behind \eqref{slip-planar}.
Often, the factor 2 in front of the viscous term is omitted; cf.\ \cite{Gresho:Fluid-Mechanics} and note that the factor 2 does not appear
if one derives a kinetic energy balance backwards, starting with the Navier-Stokes system in which ${\rm div}\,S$ has been replaced by
$\mu \Delta u - \nabla p$, employing already ${\rm div}\,u = 0$.

For incompressible flow and planar outflow boundary, the condition $P_\Gamma u=0$ together with
$ \mbox{div}\, u =0$ implies $\partial_{\nu} (u \cdot \nu) = 0$ for sufficiently regular solutions.
This way one obtains another outflow boundary condition, namely the \emph{pressure condition}
\begin{equation}
\label{outflow2}
P_\Gamma u=0 \quad \mbox{ and } \quad p = p_0 \quad \mbox{ on } (0,\,a) \times \Gamma.
\end{equation}
The latter two boundary conditions are examples of so-called \emph{artificial boundary conditions (ABCs)}, which are imposed
at artificial domain boundaries, being inside the flow domain.

To motivate further boundary conditions, which are employed as ABCs in the numerical literature, we first need the following
simple observation. Due to the incompressibility condition, it follows that
\begin{equation*}
\mbox{div}\,D = \mbox{div}\,R = \textstyle{\frac{1}{2}}\Delta u, \quad \quad
\mbox{div}\,S = \mbox{div}\,T = \mu\Delta u - \nabla p,
\end{equation*}
where $R = \frac{1}{2}(\nabla u - \nabla u^{\sf{T}})$ denotes the  rate of rotation tensor (or, spin tensor)
and $T = 2 \mu R - pI$ is the antisymmetric counterpart of the stress tensor $S$.
Hence, the {\itshape alternative form}
\begin{equation*}
  	\frac{\mbox{d}}{\mbox{d}t} \int_\Omega {\textstyle{\frac{1}{2}}} \rho {|u|}^2\,\mbox{d}x
	+ \int_{\Gamma} {\textstyle{\frac{1}{2}}} \rho {|u|}^2\, (u \cdot \nu)\,\mbox{d}\sigma
	+ 2\mu\int_\Omega {|R|}^2\,\mbox{d}x
	- \int_{\Gamma} u \cdot T \nu\,\mbox{d}\sigma
	= 0,
\end{equation*}
{\itshape of the kinetic energy balance} is also available.
Consequently, the power $\pi_{\rm NS}$ can also be expressed as
\begin{equation}
\label{power-NS-T}
\pi_{\rm NS} =
	\int_{\Gamma} \left( u \cdot T \nu - {\textstyle{\frac{1}{2}}} \rho {|u|}^2\,(u \cdot \nu) \right)\,\mbox{d}\sigma.
\end{equation}
If this contribution vanishes, the total kinetic energy will also be monotonically decreasing
and we hence also consider boundary conditions which imply zero power due to expression \eqref{power-NS-T}.

Starting again with the case of an impermeable wall $\Gamma$, i.e.\ assuming \eqref{impermeable} to hold, a further energy preserving
boundary condition evidently is
\begin{equation}
\label{eqn:nbc-ru}
u \cdot \nu = 0 \quad \mbox{ and } \quad  2 \mu R \nu = 0 \quad \mbox{ on } (0,\,a) \times \Gamma.
\end{equation}
In $\bR^3$, a direct computation shows that this is in turn equivalent to ${\rm rot}\,u \cdot \tau = 0$ for any $\tau \perp \nu$.
This leads to the \emph{vorticity boundary condition}
\[
u \cdot \nu = 0 \quad \mbox{ and } \quad  P_\Gamma {\rm rot}\,u = 0 \quad \mbox{ on } (0,\,a) \times \Gamma.
\]

For the special case of the Stokes system, additional energy preserving boundary conditions can be read off.
The power input via the boundary then is
\begin{equation}
\label{power-S-T}
\pi_{\rm S} =
	\int_{\Gamma} u \cdot T \nu\,\mbox{d}\sigma,
\end{equation}
and $u \cdot T \nu$ can be decomposed according to
\begin{equation}
\pi_{\rm S} =
	\int_{\Gamma} \Big( P_\Gamma u \cdot P_\Gamma T \nu - (u \cdot \nu)\,p \Big) \,\mbox{d}\sigma;
\end{equation}
note that $T \nu \cdot \nu = - p$.
Hence, replacing \eqref{impermeable} by the complementary equation $T \nu \cdot \nu = 0$ leads to the Neumann-type
boundary condition $T\nu = 0$. In the split form it reads as
\begin{equation}
\label{eqn:nbc-rp}
-p = 0 \quad \mbox{ and } \quad  2 \mu R \nu = 0 \quad \mbox{ on } (0,\,a) \times \Gamma,
\end{equation}
and it may be considered as the homogeneous version of a \emph{vorticity-pressure boundary condition}, related to a free boundary.

There is one more admissible combination, namely
\[
P_\Gamma u = 0 \quad \mbox{ and } \quad T \nu \cdot \nu = 0 \quad \mbox{ on } (0,\,a) \times \Gamma.
\]
This leads to the homogeneous version of the outflow boundary condition \eqref{outflow2}, i.e.\
\begin{equation}
\label{eqn:nbc-up}
P_\Gamma u = 0 \quad \mbox{ and } \quad -p = 0 \quad \mbox{ on } (0,\,a) \times \Gamma;
\end{equation}
but note that this time it also appears in this form for general curved boundaries.

Finally, we will consider the well-posedness under two further boundary conditions which are not energy preserving,
even for Stokes flow, but appear
naturally as combinations of partial boundary conditions given above. These conditions are also employed as ABCs in the
numerical literature and they read as
\begin{equation}
	\label{eqn:nbc-dp}
	2 \mu P_\Gamma D \nu = 0 \quad \mbox{and} \quad -p = 0 \quad \mbox{ on } (0,\,a) \times \Gamma,
\end{equation}
which corresponds to the prescription of the tangential part of the normal deformation rate and the external pressure, and
\begin{equation}
	\label{eqn:nbc-rs}
	2 \mu R \nu = 0 \quad \mbox{and} \quad 2 \mu\,\partial_\nu u \cdot \nu - p = 0 \quad \mbox{on}\ (0,\,a) \times \Gamma,
\end{equation}
which corresponds to the prescription of the tangential part of the vorticity and the normal part of the normal stress.

For a better overview of the boundary conditions which are rigorously analyzed in this paper, we summarize them below.
As energy preserving boundary conditions for the Navier-Stokes system we have the conditions \eqnref{nbc-uu}, \eqnref{nbc-du} and \eqnref{nbc-ru}, which read
\begin{equation*}\tag{B1a}
	P_\Gamma u = 0 \quad \mbox{and} \quad u \cdot \nu = 0 \quad \mbox{on}\ (0,\,a) \times \Gamma,
\end{equation*}
which equals the {\itshape no-slip condition} $[u] = 0$,
\begin{equation*}\tag{B1b}
	2 \mu P_\Gamma D \nu = 0 \quad \mbox{and} \quad u \cdot \nu = 0 \quad \mbox{on}\ (0,\,a) \times \Gamma,
\end{equation*}
which is known as the {\itshape perfect slip condition}, and, finally,
\begin{equation*}\tag{B1c}
	2 \mu R \nu = 0 \quad \mbox{and} \quad u \cdot \nu = 0 \quad \mbox{on}\ (0,\,a) \times \Gamma,
\end{equation*}
Additional energy preserving boundary conditions for the Stokes system are the conditions \eqnref{nbc-ds}, \eqnref{nbc-us}, \eqnref{nbc-rp} and \eqnref{nbc-up}, which read
\begin{equation*}\tag{B2a}
	P_\Gamma u = 0 \quad \mbox{and} \quad 2 \mu\,\partial_\nu u \cdot \nu - p = 0 \quad \mbox{on}\ (0,\,a) \times \Gamma,
\end{equation*}
which corresponds to the prescription of tangential velocities and the normal component of normal stress,
\begin{equation*}\tag{B2b}
	2 \mu P_\Gamma D \nu = 0 \quad \mbox{and} \quad 2 \mu\,\partial_\nu u \cdot \nu - p = 0 \quad \mbox{on}\ (0,\,a) \times \Gamma,
\end{equation*}
which equals the {\itshape Neumann condition} $S \nu = 0$,
\begin{equation*}\tag{B2c}
	P_\Gamma u = 0 \quad \mbox{and} \quad - p = 0 \quad \mbox{on}\ (0,\,a) \times \Gamma,
\end{equation*}
which corresponds to the prescription of tangential velocities and the external pressure, and,
\begin{equation*}\tag{B2d}
	2 \mu R \nu = 0 \quad \mbox{and} \quad - p = 0 \quad \mbox{on}\ (0,\,a) \times \Gamma.
\end{equation*}
Finally, we also consider the two boundary conditions \eqnref{nbc-dp} and \eqnref{nbc-rs}, which read
\begin{equation*}\tag{B3a}
	2 \mu P_\Gamma D \nu = 0 \quad \mbox{and} \quad - p = 0 \quad \mbox{on}\ (0,\,a) \times \Gamma,
\end{equation*}
which corresponds to the prescription of the tangential part of the normal deformation rate and the external pressure, and
\begin{equation*}\tag{B3b}
	2 \mu R \nu = 0 \quad \mbox{and} \quad 2 \mu\,\partial_\nu u \cdot \nu - p = 0 \quad \mbox{on}\ (0,\,a) \times \Gamma,
\end{equation*}
In the remainder of this paper we show that all of these conditions lead to locally-in-time well-posed Navier-Stokes systems
in the appropriate $L_p$-setting.

\section{Main Results}\label{sec:main}
Our analysis of incompressible Newtonian flows subject to one of the boundary conditions (B)
is based on $L_p$-maximal regularity of the underlying linear system, i.\,e. the incompressible Stokes equations
\begin{equation*}
	\label{eqn:s-data}\tag*{${(\mbox{S})}^{a, \Omega, \cB}_{f, g, h, u_0}$}
	\begin{array}{c}
		\rho\partial_t u - \mu \Delta u + \nabla p = \rho f,
		\quad \quad \mbox{div}\,u = g
			\quad \quad \mbox{in}\ (0,\,a) \times \Omega, \\[0.5em]
		\cB(u,\,p) = h
			\quad \quad \mbox{on}\ (0,\,a) \times \partial\Omega, \\[0.5em]
		u(0) = u_0
			\quad \quad \mbox{in}\ \Omega
	\end{array}
\end{equation*}
with fully inhomogeneous data $f$, $g$, $h$ and $u_0$.
Here $\cB$ denotes the linear operator, which realizes one of the discussed boundary conditions (B).

We will focus on the cases where $\Omega$ is the half-space
\begin{equation*}
	\bR^n_+ := \left\{\,(x,\,y) \in \bR^{n - 1} \times \bR\,:\,y > 0\,\right\},
\end{equation*}
a bent half-space
\begin{equation*}
	\bR^n_\omega := \left\{\,(x,\,y) \in \bR^{n - 1} \times \bR\,:\,y > \omega(x)\,\right\},
\end{equation*}
with a sufficiently smooth and flat function $\omega: \bR^{n - 1} \longrightarrow \bR$,
or a bounded domain with sufficiently smooth boundary.
In all cases we require the boundary $\Gamma = \partial \Omega$ of $\Omega$ to be of class $C^{3-}$
and we assume $1 < p < \infty$.

To establish maximal regularity in an $L_p$-setting, we employ the natural solution spaces
\begin{equation*}
	u \in \bX_u(a) := H^1_p((0,\,a),\,L_p(\Omega,\,\bR^n)) \cap L_p((0,\,a),\,H^2_p(\Omega,\,\bR^n)),
\end{equation*}
where $H^s_p$ denotes the Bessel potential space of order $s \geq 0$, and
\begin{equation*}
	p \in \bX_p(a) := L_p((0,\,a),\,\dot{H}^1_p(\Omega)),
\end{equation*}
where
\begin{equation*}
	\dot{H}^1_p(\Omega) := \left\{\,\phi \in \cD^\prime(\Omega)\,:\,\nabla \phi \in L_p(\Omega,\,\bR^n)\,\right\}
\end{equation*}
denotes the homogeneous Bessel potential space of order one,
which becomes a semi-normed space with
\begin{equation*}
{|\phi|}_{\dot{H}^1_p(\Omega)} := {\|\nabla \phi\|}_{L_p(\Omega, \bR^n)}, \quad \phi \in \dot{H}^1_p(\Omega).
\end{equation*}
Note, that $H^1_p(\Omega)$ is a dense subspace of $\dot{H}^1_p(\Omega)$ for all domains under consideration.
As usual, $\cD^\prime(\Omega)$ denotes the space of distributions on $\Omega$.
Note, that the regularity assumptions on the domain $\Omega$ imply the embedding
\begin{equation*}
	\dot{H}^1_p(\Omega) \hookrightarrow \left\{\,\phi \in L_{p, loc}(\Omega)\,:\,\phi \in H^1_p(\Omega^\prime),\ \Omega^\prime \subseteq \Omega\ \mbox{open and bounded}\,\right\}
\end{equation*}
to be valid for all $1 < p < \infty$, cf.~\cite[Chapitre 2, Th{\'e}or{\`e}me 7.6]{Necas:Elliptic-Equations}.
The corresponding data spaces are therefore determined as
\begin{equation*}
	\begin{array}{rclcl}
		f   & \in & \bY_f(a) & := & L_p((0,\,a),\,L_p(\Omega,\,\bR^n)), \\[0.5em]
		g   & \in & \bY_g(a) & := & W^{1/2}_p((0,\,a),\,L_p(\Omega)) \cap L_p((0,\,a),\,H^1_p(\Omega)), \\[0.5em]
		u_0 & \in & \bY_u    & := & W^{2 - 2/p}_p(\Omega,\,\bR^n)
	\end{array}
\end{equation*}
and the regularity class for $h$ depends on the boundary condition.

Note, that in the $L_p$-setting it is necessary to carefully distinguish between a function defined on $\Omega$ and its trace on $\Gamma$.
To account for this, we denote by $[ \cdot ]$ the trace operator.

To simplify our notation, we denote the boundary operator as $\cB = \cB^{\alpha, \beta}$,
where the first index $\alpha \in \{\,-1,\,0,\,1\,\}$ is used to describe the tangential part of the boundary condition
and the second index $\beta \in \{\,-1,\,0,\,1\,\}$ is used to describe the normal part of the boundary condition, i.\,e.
$\alpha = 0$ indicates the prescription of the tangential velocity $P_\Gamma[u]$,
$\alpha = 1$ indicates the prescription of the tangential part of the normal stress $2 \mu P_\Gamma[D] \nu$,
$\alpha = -1$ indicates the prescription of the vorticity $2 \mu [R] \nu$,
$\beta = 0$ indicates the prescription of the normal velocity $[u] \cdot \nu$,
$\beta = 1$ indicates the prescription of the normal part of the normal stress $2 \mu\,\partial_\nu u \cdot \nu - [p]$, and, finally,
$\beta = -1$ indicates the prescription of the pressure $-[p]$.
Thus, the parameters $\alpha$ and $\beta$ will be used to describe the order of the corresponding part of the boundary condition as $|\alpha|$ resp.~$|\beta|$
and to simultaneously fix the particular convex combination between $\nabla u$ and $\nabla u^{\sf{T}}$,
which would have to be used to obtain a boundary condition of order one based on the kinetic energy balance or its alternative form.
Hence, we consider the linear operators
\begin{equation*}
	\cB^{\alpha, \beta}(u,\,p) = P_\Gamma \cB^{\alpha, \beta}(u,\,p) + Q_\Gamma \cB^{\alpha, \beta}(u,\,p),
\end{equation*}
where $Q_\Gamma = I - P_\Gamma$ denotes the projection onto the normal bundle of $\Gamma$, given as
\begin{equation*}\tag{$\cB$1}
	P_\Gamma \cB^{0, \beta}(u,\,p) := P_\Gamma [u], \qquad P_\Gamma \cB^{\pm 1, \beta} := \mu P_\Gamma [\nabla u \pm \nabla u^{\sf{T}}] \nu
\end{equation*}
for $\beta \in \{\,-1,\,0,\,+1\,\}$ and
\begin{equation*}\tag{$\cB$2}
	\begin{array}{c}
		Q_\Gamma \cB^{\alpha, 0}(u,\,p) \cdot \nu := [u] \cdot \nu \\[0.5em]
		Q_\Gamma \cB^{\alpha, +1}(u,\,p) \cdot \nu := 2 \mu\,\partial_\nu u \cdot \nu - [p], \quad Q_\Gamma \cB^{\alpha, -1}(u,\,p) \cdot \nu := - [p]
	\end{array}
\end{equation*}
for $\alpha \in \{\,-1,\,0,\,+1\,\}$.
Note, that the normal derivative in the $L_p$-setting has to be understood as $\partial_\nu = [\nabla \,\cdot\,^{\sf{T}}] \nu$.
Also note, that the boundary operators $\cB^{\alpha, 0}$ with $\alpha \in \{\,-1,\,0,\,+1\,\}$ realize
the energy preserving boundary conditions (B1) for incompressible Newtonian flows.
Moreover, the boundary operators $\cB^{\alpha, -1}$ with $\alpha \in \{\,-1,\,0\,\}$ and $\cB^{\alpha, +1}$ with $\alpha \in \{\,0,\,+1\,\}$ realize
the additional energy preserving boundary conditions (B2) for incompressible Newtonian Stokes flows.
Finally, the boundary operators $\cB^{-1, +1}$ and $\cB^{+1, -1}$ realize the non-preserving boundary conditions (B3).

Now, if $u \in \bX_u(a)$ and $p \in \bX_p(a)$, we first obtain by trace theory
\begin{equation*}
	\begin{array}{c}
		P_\Gamma \cB^{0, \beta}(u,\,p) \in \bT^0_h(a), \quad \mbox{where} \\[0.5em]
		\bT^0_h(a) := W^{1 - 1/2p}_p((0,\,a),\,L_p(\Gamma,\,T\Gamma)) \cap L_p((0,\,a),\,W^{2 - 1/p}_p(\Gamma,\,T\Gamma))
	\end{array}
\end{equation*}
for $\beta \in \{\,-1,\,0,\,1\,\}$, where $T\Gamma$ denotes the tangent bundle of $\Gamma$,
\begin{equation*}
	\begin{array}{c}
		P_\Gamma \cB^{\pm 1, \beta}(u,\,p) \in \bT^{\pm 1}_h(a), \quad \mbox{where} \\[0.5em]
		\bT^{\pm 1}_h(a) := W^{1/2 - 1/2p}_p((0,\,a),\,L_p(\Gamma,\,T\Gamma)) \cap L_p((0,\,a),\,W^{1 - 1/p}_p(\Gamma,\,T\Gamma))
	\end{array}
\end{equation*}
for $\beta \in \{\,-1,\,0,\,1\,\}$,
\begin{equation*}
	\begin{array}{c}
		Q_\Gamma \cB^{\alpha, 0}(u,\,p) \in \bN^0_h(a), \quad \mbox{where} \\[0.5em]
		\bN^0_h(a) := W^{1 - 1/2p}_p((0,\,a),\,L_p(\Gamma,\,N\Gamma)) \cap L_p((0,\,a),\,W^{2 - 1/p}_p(\Gamma,\,N\Gamma))
	\end{array}
\end{equation*}
for $\alpha \in \{\,-1,\,0,\,1\,\}$, where $N\Gamma$ denotes the normal bundle of $\Gamma$,
and, finally,
\begin{equation*}
	Q_\Gamma \cB^{\alpha, \pm 1}(u,\,p) \in \bN^{\pm 1}_h(a) := L_p((0,\,a),\,\dot{W}^{1 - 1/p}_p(\Gamma,\,N\Gamma))
\end{equation*}
for $\beta \in \{\,-1,\,1\,\}$.
Therefore, the regularity class for the boundary data $h$ is given as
\begin{equation*}
	h \in \bY^{\alpha, \beta}_h(a) := \left\{\,\eta \in L_p((0,\,a),\,L_{p,loc}(\Gamma,\,\bR^n))\,: \begin{array}{c} \cP_\Gamma \eta \in \bT^\alpha_h(a), \\[0.5em] (\,\eta\,|\,\nu\,) \nu \in \bN^\beta_h(a) \end{array}\right\}
\end{equation*}
for $\alpha,\,\beta \in \{\,-1,\,0,\,1\,\}$.

However, some boundary conditions may impose additional regularity properties on the pressure trace $[p]$.
Since
\begin{equation*}
	Q_\Gamma \cB^{\alpha, +1}(u,\,p) \cdot \nu = 2 \mu\,\partial_\nu u \cdot \nu - [p]
\end{equation*}
for $\alpha \in \{\,-1,\,0,\,1\,\}$,
the regularity class of the data may also be chosen according to the regularity of $[\nabla u]$.
Obviously, the regularity of $[\nabla u]$ is significantly higher than the regularity of $[p]$.
Indeed, as will be shown, in case of the boundary conditions given by $\cB^{\alpha, +11}$ for $\alpha \in \{\,-1,\,0,\,1\,\}$
every regularity condition in between that two may be imposed, i.\,e. the pressure $p$ belongs to the regularity class
\begin{equation*}
	\bX_{p, \gamma}(a) := \left\{\,\pi \in \bX_p(a)\,:\,[\pi] \in W^\gamma_p((0,\,a),\,L_p(\Omega)) \cap L_p((0,\,a),\,W^{1 - 1/p}_p(\Gamma))\,\right\}
\end{equation*}
with $\gamma \in [0,\,1/2 - 1/2p]$, if and only if the boundary data satisfies
\begin{equation*}
	Q_\Gamma h \in \bN^{+1}_{h, \gamma}(a) := W^\gamma_p((0,\,a),\,L_p(\Gamma,\,N\Gamma)) \cap L_p((0,\,a),\,W^{1 - 1/p}_p(\Gamma,\,N\Gamma)).
\end{equation*}
In case of the boundary conditions given by $\cB^{\alpha, -1}$ for $\alpha \in \{\,-1,\,0,\,1\,\}$,
the normal part of the boundary data equals the trace of the pressure.
Hence, the pressure $p$ belongs to the regularity class $\bX_{p, \gamma}(a)$ with $\gamma \in [0,\,\infty)$, if and only if the boundary data satisfies
$Q_\Gamma h \in \bN^{-1}_{h, \gamma}(a) := \bN^{+1}_{h, \gamma}(a)$.
For simplification we will therefore employ the notations
\begin{equation*}
	\bX_{p, -\infty}(a) := \bX_p(a) \quad \mbox{and} \quad
	\bN^\beta_{h, -\infty}(a) := \bN^\beta_h(a)
\end{equation*}
for $\beta \in \{\,-1,\,0,\,1\,\}$ and require
\begin{equation*}
	h \in \bY^{\alpha, \beta}_{h, \gamma}(a) := \left\{\,\eta \in L_p((0,\,a),\,L_{p,loc}(\Gamma,\,\bR^n))\,: \begin{array}{c} \cP_\Gamma \eta \in \bT^\alpha_h(a), \\[0.5em] (\,\eta\,|\,\nu\,) \nu \in \bN^\beta_{h, \gamma}(a) \end{array}\right\}
\end{equation*}
with $\gamma = -\infty$ for $\beta = 0$, with $\gamma \in \{\,-\infty\,\} \cup [0,\,1/2 - 1/2p]$ for $\beta = 1$
or with $\gamma \in \{\,-\infty\,\} \cup [0,\,\infty)$ for $\beta = -1$.

Another issue concerning the pressure $p$, which has to be addressed, is the fact,
that its uniqueness and continuous dependence on the data can only be guaranteed,
if the pressure trace $[p]$ is uniquely determined via the boundary condition.
If the pressure does not appear in the boundary condition, uniqueness has to be understood as uniqueness up to a constant.
To account for this phenomenon, we define
\begin{equation*}
	\bX^0_{p, -\infty}(a) := \bX_{p, -\infty}(a) / \bR \quad \mbox{and} \quad \bX^{\pm 1}_{p, \gamma}(a) := \bX_{p, \gamma}(a),
\end{equation*}
which ensures a unique pressure $p \in \bX^\beta_{p, \gamma}(a)$,
if the boundary condition is given by the linear operator $\cB = \cB^{\alpha, \beta}$ with $\alpha,\,\beta \in \{\,-1,\,0,\,1\,\}$.
The additional regularity parameter $\gamma$ has to be chosen according to the constraints discussed above.
Note, that $\bX^0_{p, -\infty}(a)$ constitutes a Banach space,
whereas $\bX^{\pm 1}_{p, -\infty}(a)$ is semi-normed via
\begin{equation*}
	{|\pi|}_{\bX^{\pm 1}_{p, -\infty}(a)} = {|\pi|}_{\bX_p(a)}, \quad \pi \in \bX^{\pm 1}_{p, -\infty}(a).
\end{equation*}
However, for $\gamma \geq 0$, the spaces $\bX^{\pm 1}_{p, \gamma}(a)$ equipped with their natural norm
\begin{equation*}
	{\|\pi\|}_{\bX^{\pm 1}_{p, \gamma}(a)} = \mbox{max}\left\{\,{|\pi|}_{\bX_p(a)},\,{\|[\pi]\|}_{\bN^{\pm 1}_{h, \gamma}(a)}\,\right\},
	\quad \pi \in \bX^1_{p, \gamma}(a)
\end{equation*}
constitute Banach spaces, too.
Analogously, the data spaces $\bY^{\alpha, 0}_{h, -\infty}(a)$ and
$\bY^{\alpha, \beta}_{h, \gamma}(a)$ with $\alpha \in \{\,-1,\,0,\,1\,\}$, $\beta \in \{\,-1,\,1\,\}$ and $\gamma \geq 0$
constitute Banach spaces with their natural norm
\begin{equation*}
	{\|\eta\|}_{\bY^{\alpha, \beta}_{h, \gamma}(a)} = \mbox{max}\left\{\,{\|P_\Gamma \eta\|}_{\bT^\alpha_h(a)},\,{\|Q_\Gamma \eta\|}_{\bN^\beta_{h, \gamma}(a)}\,\right\},
	\quad \eta \in \bY^{\alpha, \beta}_{h, \gamma}(a),
\end{equation*}
whereas the spaces $\bY^{\alpha, \beta}_{h, -\infty}$ with $\alpha \in \{\,-1,\,0,\,1\,\}$, $\beta \in \{\,-1,\,1\,\}$ are semi-normed via
\begin{equation*}
	{|\eta|}_{\bY^{\alpha, \beta}_{h, -\infty}(a)} = \mbox{max}\left\{\,{\|P_\Gamma \eta\|}_{\bT^\alpha_h(a)},\,{|Q_\Gamma \eta|}_{\bN^\beta_{h, -\infty}(a)}\,\right\},
	\quad \eta \in \bY^{\alpha, \beta}_{h, \gamma}(a).
\end{equation*}
Hence, continuous dependence of the solution on the data has in some cases to be understood w.\,r.\,t. semi-norms,
regardless of its uniqueness, which will always be guaranteed.

In addition to the above regularity conditions, there are several compatibility conditions, which have to be satisfied by the data.
First of all, the compatibility condition
\begin{equation*}\tag*{${(\mbox{C1})}_{g, u_0}$}
	\mbox{div}\,u_0 = g(0)
\end{equation*}
is necessary and a boundary condition given by the linear operator $\cB = \cB^{\alpha, \beta}$ with $\alpha,\,\beta \in \{\,-1,\,0,\,1\,\}$
implies the compatibility conditions
\begin{equation*}\tag*{${(\mbox{C2})}^\alpha_{h, u_0}$}
	\begin{array}{rcll}
		P_\Gamma [u_0]                                        & = & P_\Gamma h(0), & \mbox{if}\ \alpha = 0 \     \mbox{and}\ p > \frac{3}{2}, \\[0.5em]
		\mu P_\Gamma (\nabla u_0 \pm \nabla u^{\sf{T}}_0) \nu & = & P_\Gamma h(0), & \mbox{if}\ \alpha = \pm 1 \ \mbox{and}\ p > 3
	\end{array}
\end{equation*}
to be necessary for \ref{eqn:s-data} to admit a maximal regular solution.

Last but not least, there is an additional compatibility condition, which stems from the divergence equation in \ref{eqn:s-data}.
To reveal it, we set
\begin{equation*}
	{}_0 H^{-1}_p(\Omega) := H^1_{p^\prime}(\Omega)^\prime
\end{equation*}
with $1/p + 1/p^\prime = 1$ and define the linear functional
\begin{equation*}
	(\,\cdot\,,\,\cdot\,): \bY_g(a) \times \bN^0_{h, -\infty}(a) \longrightarrow L_p((0,\,a),\,{}_0 H^{-1}_p(\Omega))
\end{equation*}
for $\psi \in \bY_g(a)$ and $\eta \in \bN^0_{h, -\infty}(a)$ via
\begin{equation*}
	\langle\,\phi\,|\,(\,\psi,\,\eta\,)\,\rangle := \int_\Gamma [\phi]\eta\,\mbox{d}\sigma - \int_\Omega \phi \psi\,\mbox{d}x,
	\quad \phi \in H^1_{p^\prime}(\Omega).
\end{equation*}
An integration by parts yields
\begin{equation*}
	  \langle\,\phi\,|\,(\,\mbox{div}\,u,\,[u] \cdot \nu\,)\,\rangle
	= \int_\Omega \nabla \phi \cdot u\,\mbox{d}x,
	\quad \phi \in H^1_{p^\prime}(\Omega)
\end{equation*}
and we infer
\begin{equation*}
 	|\langle\,\phi\,|\,(\,\mbox{div}\,u,\,[u] \cdot \nu\,)\,\rangle|
	\leq {\|u\|}_{\bX_u(a)} {|\phi|}_{\dot{H}^1_{p^\prime}(\Omega)},
	\quad \phi \in H^1_{p^\prime}(\Omega)
\end{equation*}
as well as
\begin{equation*}
	|\langle\,\phi\,|\,\partial_t (\,\mbox{div}\,u,\,[u] \cdot \nu\,)\,\rangle|
	\leq {\|u\|}_{\bX_u(a)} {|\phi|}_{\dot{H}^1_{p^\prime}(\Omega)},
	\quad \phi \in H^1_{p^\prime}(\Omega).
\end{equation*}
Hence, a boundary condition given by the linear operator $\cB = \cB^{\alpha, \beta}$ with $\alpha,\,\beta \in \{\,-1,\,0,\,1\,\}$
implies the compatibility condition
\begin{equation*}\tag*{${(\mbox{C3})}^\beta_{g, h, u_0}$}
	\begin{array}{c}
		\begin{array}{c}
			Q_\Gamma [u_0] = Q_\Gamma h(0), \quad \mbox{if}\ p > \frac{3}{2}, \quad \mbox{and} \\[0.5em]
			(\,g,\,Q_\Gamma h\,) \in H^1_p((0,\,a),\,(H^1_{p^\prime}(\Omega),\,{|\,\cdot\,|}_{\dot{H}^1_{p^\prime}(\Omega)})^\prime),
		\end{array} \qquad \mbox{if}\ \beta = 0, \\[2em]
		\begin{array}{c}
			\mbox{there exists} \ \eta \in \bN^0_{h, -\infty}(a) \ \mbox{such that} \\[0.5em]
			Q_\Gamma [u_0] = \eta(0), \quad \mbox{if}\ p > \frac{3}{2}, \quad \mbox{and} \\[0.5em]
			(\,g,\,\eta\,) \in H^1_p((0,\,a),\,(H^1_{p^\prime}(\Omega),\,{|\,\cdot\,|}_{\dot{H}^1_{p^\prime}(\Omega)})^\prime),
		\end{array} \qquad \mbox{if}\ \beta \in \{\,-1,\,1\,\}
	\end{array}
\end{equation*}
to be necessary for \ref{eqn:s-data} to admit a maximal regular solution.

On the other hand, the above regularity and compatibility conditions are also sufficient to construct a unique maximal regular solution
to the Stokes equations \ref{eqn:s-data} for all $(f,\,g,\,h,\,u_0) \in \bY^{\alpha, \beta}_\gamma(a)$,
where we choose the space $\bY^{\alpha, \beta}_\gamma(a)$ to consist of all
\begin{equation*}
	(f,\,g,\,h,\,u_0) \in \bY_f(a) \times \bY_g(a) \times \bY^{\alpha, \beta}_{h, \gamma}(a) \times \bY_u,
\end{equation*}
which satisfy the compatibility conditions {(C)}${}^{\alpha, \beta}_{g, h, u_0}$.
The detailed results will be stated as \Thmref{s} below.

With the above notations at hand, we are able to formulate our main results.
Concerning the local well-posedness of incompressible Newtonian flows subject to one of the boundary conditions (B) we will prove
\begin{theorem}\label{thm:n}{\bfseries Local Well-Posedness, Semi-flow, Energy Inequality.}\newline
Let $\Omega \subseteq \bR^n$ be a half-space, a bent half-space or a bounded domain.
Let $\Gamma = \partial \Omega$ be of class $C^{3-}$ and let $n + 2 < p < \infty$.
Let $\cB = \cB^{\alpha, \beta}$ with $\alpha,\,\beta \in \{\,-1,\,0,\,1\,\}$ be one of the linear operators ($\cB$),
which realizes one of the boundary conditions (B).
If $\beta =  0$ let $\gamma = -\infty$;
if $\beta =  1$ let $\gamma = 1/2 - 1/2p$;
if $\beta = -1$ let $\gamma \geq 0$.

Then the Navier-Stokes equations {(N)}${}^{\infty, \Omega, \cB}_{f, u_0}$ admit a unique local-in-time solution $(u,\,p)$
on a maximal time interval $[0,\,a^\ast(f,\,u_0))$,
whenever the data $f \in \bY_f(\infty)$ and $u_0 \in \bY_u$ satisfy the compatibility conditions {(C)}${}^{\alpha, \beta}_{0, 0, u_0}$.
The solution is in the maximal regularity class
\begin{equation*}
	u \in \bX_u(a), \quad p \in \bX^\beta_{p, \gamma}(a)
\end{equation*}
for all $0 < a < a^\ast(f,\,u_0)$.
The maximal existence time is characterized as
\begin{equation*}
	a^\ast(f,\,u_0) < \infty \quad \Rightarrow \quad \lim_{t \rightarrow a^\ast(f,\,u_0)} u(t) \quad \mbox{does not exist in}\ W^{2 - 2/p}_p(\Omega).
\end{equation*}
Moreover, the solution enjoys the following properties:
\begin{enumerate}
	\item If one of the energy preserving boundary conditions for incompressible Newtonian flows (B1) is imposed and $f \in \bY_f(\infty) \cap L_p((0,\,\infty),\,L_2(\Omega,\,\bR^n))$,
	then the energy inequality
	\begin{equation*}
		\frac{\mbox{d}}{\mbox{d}t} \int_\Omega {\textstyle{\frac{1}{2}}} \rho {|u|}^2\,\mbox{d}x \leq \int_\Omega (\,\rho u\,|\,f\,)\,\mbox{d}x
	\end{equation*}
	is valid.
	\item If $f = 0$, then the solution map $u_0 \mapsto u$ generates a local semi-flow in
	\begin{equation*}
		\bZ^{\alpha, \beta} = \left\{\,v \in W^{2 - 2/p}_p(\Omega)\,:\,v \ \mbox{satisfies} \ {\mbox{(C)}}^{\alpha, \beta}_{0, 0, v}\,\right\},
	\end{equation*}
	the natural phase space for {(N)}${}^{\infty, \Omega, \cB}_{0, u_0}$ in the $L_p$-setting.
\end{enumerate}
\end{theorem}
Note, that the main statement of \Thmref{n} is the existence of unique local-in-time solutions as can be seen as follows.
\begin{remark}\label{rem:n}
The claimed characterization of the maximal existence time is a direct consequence of the existence of unique local-in-time solutions.
Indeed, if $a^\ast(f,\,u_0) < \infty$ and $\lim_{t \rightarrow a^\ast(f,\,u_0)} u(t) =: v \in W^{2 - 2/p}_p(\Omega)$ would exist,
then there would also be a unique local-in-time solution starting at $t = a^\ast(f,\,u_0)$ with initial value $v$,
which would extend the solution starting at $t = 0$ with initial value $u_0$ beyond its maximal existence time.
Analogously, the semi-flow property of the solutions to {(N)}${}^{\infty, \Omega, \cB}_{0, u_0}$ is a direct consequence of the existence of unique local-in-time solutions as well.
Moreover, the claimed energy inequality is a direct consequence of the construction of the energy preserving boundary conditions (B1).
Hence, for a complete proof of \Thmref{n} to be established,
it is sufficient to prove the existence of unique local-in-time solutions with the claimed regularity properties.
\end{remark}
The construction of unique local-in-time solutions will be carried out in \Secref{n},
based on the maximal regularity property of \ref{eqn:s-data} in the $L_p$-setting,
which is provided by
\begin{theorem}\label{thm:s}{\bfseries $L_p$-maximal Regularity, Semi-flow, Energy Inequality.}\newline
Let $a > 0$ and let $\Omega \subseteq \bR^n$ be a half-space, a bent half-space or a bounded domain.
Let $\Gamma = \partial \Omega$ be of class $C^{3-}$ and let $1 < p < \infty$, $p \neq \frac{3}{2},\,3$.
Let $\cB = \cB^{\alpha, \beta}$ with $\alpha,\,\beta \in \{\,-1,\,0,\,1\,\}$ be one of the linear operators ($\cB$),
which realizes one of the boundary conditions (B).

If $\beta =  0$ let $\gamma = -\infty$;
if $\beta =  1$ let $\gamma \in \{\,-\infty\,\} \cup [0,\,1/2 - 1/2p]$;
if $\beta = -1$ let $\gamma \in \{\,-\infty\,\} \cup [0,\,\infty)$.

Then the Stokes equations {(S)}${}^{a, \Omega, \cB}_{f, g, h, u_0}$ admit a unique maximal regular solution
\begin{equation*}
	u \in \bX_u(a), \quad p \in \bX^\beta_{p, \gamma}(a),
\end{equation*}
if and only if the data satisfies the regularity conditions
\begin{equation*}
	f \in \bY_f(a), \quad g \in \bY_g(a), \quad h \in \bY^{\alpha, \beta}_{h, \gamma}(a), \quad u_0 \in \bY_u
\end{equation*}
and the compatibility conditions {(C)}${}^{\alpha, \beta}_{g, h, u_0}$.
The solution map
\begin{equation*}
	\bY^{\alpha, \beta}_\gamma(a) \longrightarrow \bX_u(a) \times \bX^\beta_{p, \gamma}(a)
\end{equation*}
is continuous and the solution enjoys the following properties:
\begin{enumerate}
	\item If one of the energy preserving boundary conditions for incompressible Newtonian Stokes flows (B1) or (B2) is imposed, and, additionally, $f \in \bY_f(a) \cap L_p((0,\,a),\,L_2(\Omega,\,\bR^n))$, $g = 0$ and $h = 0$,
	then the energy inequality
	\begin{equation*}
		\frac{\mbox{d}}{\mbox{d}t} \int_\Omega {\textstyle{\frac{1}{2}}} \rho {|u|}^2\,\mbox{d}x \leq \int_\Omega (\,\rho u\,|\,f\,)\,\mbox{d}x
	\end{equation*}
	is valid.
	\item If $f = 0$, $g = 0$ and $h = 0$, then the solution map $u_0 \mapsto u$ generates a semi-flow in
	\begin{equation*}
		\bZ^{\alpha, \beta} = \left\{\,v \in W^{2 - 2/p}_p(\Omega)\,:\,v \ \mbox{satisfies}\ {\mbox{(C)}}^{\alpha, \beta}_{0, 0, v}\,\right\},
	\end{equation*}
	the natural phase space for {(S)}${}^{a, \Omega, \cB}_{0, 0, 0, u_0}$ in the $L_p$-setting.
\end{enumerate}
\end{theorem}
Again, the main statement of \Thmref{s} is the existence of unique maximal regular solutions as can be seen as follows.
\begin{remark}\label{rem:s}
Once the existence of unique maximal regular solutions is proved,
the continuity of the solution map is a consequence of the open mapping principle.
Moreover, the semi-flow property follows from the existence of unique maximal regular solutions.
Last, but not least, the claimed energy inequality is a direct consequence of the construction of the energy preserving boundary conditions (B1) and (B2).
Hence, for a complete proof of \Thmref{s} to be established,
it is sufficient to prove the existence of unique maximal regular solutions.
\end{remark}

We close this section with some remarks, how the considered boundary conditions have already been treated in the literature.
Concerning strong solutions to the Navier-Stokes equations subject to a homogeneous Dirichlet condition,
{\scshape H.~Fujita} and {\scshape T.~Kato} as well as {\scshape P.~E.~Sobolevskii} have established unique local strong solutions in an $L_2$-setting
based on a semigroup approach already in 60's, cf.~\cite{Fujita-Kato:Navier-Stokes-Nonstationary, Fujita-Kato:Navier-Stokes-IVP, Sobolevskii:Navier-Stokes}.
Later on, {\scshape Y.~Giga} and {\scshape T.~Miyakawa} as well as {\scshape F.~B.~Weissler} generalized these results to the $L_p$-setting,
cf.~\cite{Giga:Stokes-Operator-Domains, Giga-Miyakawa:Navier-Stokes, Weissler:Navier-Stokes}.
The first approaches based on resolvent estimates in an $L_p$-setting are due to {\scshape V.~A.~Solonnikov}, {\scshape M.~McCracken}
respectively {\scshape S.~Ukai}, cf.~\cite{Solonnikov:Navier-Stokes, McCracken:Stokes, Ukai:Stokes}.
Finally, maximal $L_p$-regularity was established by {\scshape V.~A.~Solonnikov}, {\scshape W.~Borchers} and {\scshape T.~Miyakawa},
respectively {\scshape W.~Desch}, {\scshape M.~Hieber} and {\scshape J.~Pr{\"u}{\ss}},
cf.~\cite{Solonnikov:Navier-Stokes, Borchers-Miyakawa:Navier-Stokes, Desch-Hieber-Pruess:Stokes}.
A semigroup approach to the Stokes and Navier-Stokes equations, which in particular yields maximal $L_p$-$L_q$-regularity,
was developed by {\scshape T.~Kubo} and {\scshape Y.~Shibata}, cf.~\cite{Kubo-Shibata:Stokes, Kubo-Shibata:Navier-Stokes-Equations, Kubo-Shibata:Navier-Stokes-Flows}.
This approach was later generalized to Navier respectively Robin boundary conditions by {\scshape Y.~Shibata} and {\scshape R.~Shimada},
cf.~\cite{Shibata-Shimada:Stokes-Robin}.
The extremal cases, i.\,e.~the no-slip and the perfect slip condition,
are also covered and the results apply to a halfspace, bent and perturbed halfspaces as well as to bounded and exterior domains.

The Stokes and Navier-Stokes equations subject to Neumann boundary conditions
have first been considered by {\scshape V.~A.~Solonnikov} in a series of publications,
cf.~\cite{Solonnikov:Navier-Stokes-Free-Surface, Solonnikov:Stokes-Neumann, Solonnikov:Navier-Stokes-Neumann, Solonnikov:Stokes-Neumann-Free-Boundary}.
Later on, {\scshape Y.~Shibata} and {\scshape S.~Shimizu} established maximal $L_p$-regularity,
cf.~\cite{Shibata-Shimizu:Stokes-Neumann-Resolvent, Shibata-Shimizu:Stokes-Free-Surface, Shibata-Shimizu:Stokes-Neumann-Exterior-Domain, Shibata-Shimizu:Stokes-Neumann-Maximal-Regularity, Shibata-Shimizu:Stokes-Two-Phase-Model}.
In all cases, the motivation are either problems with free boundary or two-phase problem with an evolving phase-separating interface.
As has already been mentioned in the introduction and in \Secref{bcs},
the Stokes and Navier-Stokes equations subject to a Neumann boundary condition arise as a model problem in these situations.

A totally different approach to the Stokes and Navier-Stokes equations in an $L_p$-setting was developed in the 90's
by {\scshape G.~Grubb} and {\scshape V.~A.~Solonnikov},
cf.~\cite{Grubb-Solonnikov:Stokes, Grubb-Solonnikov:Pseudodifferential, Grubb-Solonnikov:Navier-Stokes-Nonstationary, Grubb-Solonnikov:Navier-Stokes, Grubb:Navier-Stokes, Grubb:Pseudodifferential-BCs, Grubb:Navier-Stokes-Nonhomogeneous}.
In this series of publications, the Stokes and Navier-Stokes equations are transformed into a system of pseudodifferential evolution equations and treated by an abstract pseudodifferential calculus.
This way, the authors are able to treat the Dirichlet, perfect slip and Neumann boundary conditions
and establish maximal $L_p$-regularity in each case.
As a remarkable fact, this method is applicable for mixed order boundary conditions, which have not been treated in the literature before.

Another popular boundary condition for impermeable walls, which is not mentioned in \Secref{bcs},
arises from the Navier condition,
if we additionally assume the boundary $\Gamma$ to be perfectly flat, i.\,e.~to coincide with a two-dimensional plane in $\bR^3$.
In this case, the Navier condition coincides with the {\itshape Robin Condition}.
For this condition, a complete $L_p$-theory including an $\cH^\infty$-calculus due to
{\scshape J.~Saal} is available by \cite{Saal:Navier-Stokes-Robin, Saal:Stokes-Robin-Non-Reflexive} and the monograph \cite{Saal:Stokes-Robin}.

A further well-known approach to the Navier-Stokes equations dates back to the fundamental works of {\scshape J.~Leray} and {\scshape E.~Hopf},
cf.~\cite{Leray:Navier-Stokes-Domain, Leray:Navier-Stokes-Wholespace, Hopf:Navier-Stokes}, who introduced the concept of {\itshape weak solutions}.
This way, one may construct global solutions without any smallness assumption on the initial datum.
However, the question of the uniqueness of these solutions still remains open in space dimension $n \geq 3$.
Therefore, weak solutions to the Navier-Stokes equations are still an active field of research in mathematics
and the concept is further developed nowadays.
Nevertheless, the literature focuses on (homogeneous) Dirichlet conditions.
Notable exceptions are the articles by {\scshape H.~Bellout}, {\scshape J.~Neustupa} and {\scshape P.~Penel},
cf.~\cite{Bellout-Neustupa-Penel:Navier-Stokes-Vorticity, Neustupa-Penel:Navier-Stokes-Impermeability, Neustupa-Penel:Navier-Stokes-Regularity, Neustupa-Penel:Navier-Stokes-Vorticity, Neustupa-Penel:Navier-Stokes-Navier},
who consider boundary conditions of Navier type as {\itshape generalized impermeability boundary conditions},
and the article by {\scshape F.~Boyer} and {\scshape P.~Fabrie}, cf.~\cite{Boyer-Fabrie:Navier-Stokes},
who consider Neumann type boundary conditions.
For a more detailed overview of the theory of weak solutions and its development we refer to the monographs by {\scshape O.~A.~Ladyzhenskaya}, {\scshape R.~Temam} and {\scshape G.~P.~Galdi},
cf.~\cite{Ladyzhenskaya:Navier-Stokes, Temam:Navier-Stokes, Galdi:Navier-Stokes-1, Galdi:Navier-Stokes-2}.

Several of the energy preserving boundary conditions (B)
are well-known to be applicable for artificial outflow boundaries in numerical simulations.
The vorticity conditions ($\alpha = -1$) and the pressure conditions ($\beta = -1$)
have first been used in \cite{Conca-Murat-Pironneau:Pressure-Conditions, Conca-Pares-Pironneau-Thiriet:Pressure-Conditions}.
However, a mathematically rigorous analysis of the resulting initial boundary value problems is not available up to now.

\section{A Splitting Scheme}\label{sec:splitting}
To construct a solution to \ref{eqn:s-data} with $\cB = \cB^{\alpha, \beta}$, where $\alpha,\,\beta \in \{\,-1,\,0,\,1\,\}$,
it will be convenient, to reduce the problem to the spacial case $f = 0$, $g = 0$, $P_\Gamma h = 0$, $u_0 = 0$ and $\gamma = 1/2 - 1/2p$, if $\beta = 1$.
To achieve this reduction, we decompose the desired solution $u \in \bX_u(a)$ and $p \in \bX^\beta_{p, \gamma}(a)$
as $u = v + \bar{u}$ with $v,\,\bar{u} \in \bX_u(a)$ and $p = q + \bar{p}$ with $q,\,\bar{p} \in \bX^\beta_{p, \gamma}(a)$,
where $v$ is determined as the strong solution to a suitable parabolic problem with data depending on $f$, $g$, $P_\Gamma h$ and $u_0$,
and $q$ is determined as the weak solution to a suitable elliptic problem with data depending on $f$, $g$ and $Q_\Gamma h$.
If $\beta = -1$, we may choose $\bar{u} = 0$ and $\bar{p}$ = 0, i.\,e. the solution is completely determined by the splitting scheme.
If $\beta \in \{\,0,\,+1\,\}$, the remaining part $\bar{u}$ and $\bar{p}$ of the solution is determined as the maximal regular solution
to {(S)}${}^{a, \Omega, \cB}_{0, 0, \bar{h}, 0}$, where $P_\Gamma \bar{h} = 0$.
Of course, $\bar{h}$ will depend on $v$ and $q$,
but we will ensure
\begin{equation*}
Q_\Gamma \bar{h} \in {}_0 H^1_p((0,\,a),\,\dot{W}^{-1/p}_p(\Gamma,\,N\Gamma)) \cap \bN^0_{h, -\infty}(a), \quad \mbox{if} \ \beta = 0,
\end{equation*}
resp.
\begin{equation*}
Q_\Gamma \bar{h} \in \bN^{+1}_{h, 1/2 - 1/2p}(a), \quad \mbox{if} \ \beta = +1.
\end{equation*}

Before we state our main theorem concerning the splitting scheme,
we need to prepare the suitable framework for the elliptic boundary value problems,
which will be solved for the pressure.
Depending on the boundary condition under consideration, we will construct solutions $q \in \dot{H}^1_p(\Omega)$ to the {\itshape Dirichlet problem}
\begin{equation*}
	\label{eqn:sd}\tag*{${(\mbox{DP})}_{f, h}$}
		-\Delta\,q = \mbox{div}\,f \quad \mbox{in}\ \Omega, \qquad 
		[q] = h \quad \mbox{on}\ \partial \Omega
\end{equation*}
with data $f \in L_p(\Omega,\,\bR^n)$ and $h \in \dot{W}^{1 - 1/p}_p(\partial \Omega)$.
Therefore, we want to recall the known results concerning this problem for a sufficiently large class of domains and allow $\Omega$ to be the halfspace $\bR^n_+$, a bent halfspace $\bR^n_\omega$,
where $\omega: \bR^{n - 1} \longrightarrow \bR$ is assumed to be of class $C^1$, or a bounded domain with boundary of class $C^1$.
For all these domains the trace space $\dot{W}^{1 - 1/p}_p(\partial \Omega)$ is well-defined and there exists a bounded linear trace operator
\begin{equation}
	\label{eqn:trace}
	[\,\cdot\,]: \dot{H}^1_p(\Omega) \longrightarrow \dot{W}^{1 - 1/p}_p(\partial \Omega).
\end{equation}
Due to the weak regularity assumptions on $q$ and $f$, problem (DP) has to be understood in a weak sense.
To achieve the corresponding weak formulation, we define
\begin{equation*}
	{}_0 \dot{H}^1_p(\Omega) := \mbox{cls}\,\left(C^\infty_0(\Omega),\,{|\,\cdot\,|}_{\dot{H}^1_p(\Omega)}\right) \subseteq \dot{H}^1_p(\Omega)
\end{equation*}
as the closure of the space $C^\infty_0(\Omega)$ of compactly supported, smooth functions in the semi-normed space $\dot{H}^1_p(\Omega)$.
Then, for all domains under consideration and all $1 < p < \infty$ the identity
\begin{equation*}
	{}_0 \dot{H}^1_p(\Omega) = \left\{\,\phi \in \dot{H}^1_p(\Omega)\,:\,[\phi] = 0\,\right\}
\end{equation*}
is valid, cf.~\cite{Simader-Sohr:Weak-Dirichlet}.
With this definition the (weak) Dirichlet problem {(DP)}${}_{f, h}$ is equivalent to its weak formulation
\begin{equation*}
	\label{eqn:wd}\tag*{${(\mbox{DP})}^w_{f, h}$}
	\begin{array}{c}
		-(\,\nabla \phi\,|\,\nabla q\,) = (\,\nabla \phi\,|\,f\,), \quad \phi \in {}_0 \dot{H}^1_{p^\prime}(\Omega), \\[0.5em]
		[q] = h \quad \mbox{on}\ \partial \Omega.
	\end{array}
\end{equation*}
Following the argumentation in \cite{Simader-Sohr:Weak-Dirichlet}, this is exactly the right setting to obtain maximal regular solutions to (DP) in a weak sense via
\begin{proposition}\label{prop:sd}{\bfseries The Weak Dirichlet Problem.}\newline
Let $\Omega \subseteq \bR^n$ be the half-space, a bent half-space or a bounded domain with boundary $\Gamma = \partial \Omega$ of class $C^1$ and let $1 < p < \infty$.
Then the (weak) Dirichlet problem {(DP)}${}_{f, h}$ resp.~{(DP)}${}^w_{f, h}$ admits a unique maximal regular solution $q \in \dot{H}^1_p(\Omega)$,
whenever $f \in L_p(\Omega,\,\bR^n)$ and $h \in \dot{W}^{1 - 1/p}_p(\partial \Omega)$.
\end{proposition}
\begin{proof}
First note, that the trace operator \eqnref{trace} is onto.
Now, given $f \in L_p(\Omega,\,\bR^n)$ and $h \in \dot{W}^{1 - 1/p}_p(\partial \Omega)$, we first choose $\bar{q} \in \dot{H}^1_p(\Omega)$ with
$[\bar{q}] = h$ and then solve
\begin{equation*}
	-(\,\nabla \phi\,|\,\nabla q - \nabla \bar{q}\,) = (\,\nabla \phi\,|\,f + \nabla \bar{q}\,), \quad \phi \in {}_0 \dot{H}^1_{p^\prime}(\Omega)
\end{equation*}
to obtain $q - \bar{q} \in {}_0 \dot{H}^1_p(\Omega)$.
The possibility of obtaining such a solution follows from \cite[Theorem II.1.1]{Simader-Sohr:Weak-Dirichlet},
which covers the case of a bounded domain, resp.~\cite[Lemma II.2.5]{Simader-Sohr:Weak-Dirichlet},
which covers the (bent) halfspace case.

Finally, if $q,\,\hat{q} \in \dot{H}^1_p(\Omega)$ are two solutions to {(DP)}${}^w_{f, h}$,
then $q - \hat{q} \in {}_0 \dot{H}^1_p(\Omega)$ solves
\begin{equation*}
	-(\,\nabla \phi\,|\,\nabla q - \nabla \hat{q}\,) = 0, \quad \phi \in {}_0 \dot{H}^1_{p^\prime}(\Omega),
\end{equation*}
which implies $q = \hat{q}$ again by \cite[Theorem II.1.1]{Simader-Sohr:Weak-Dirichlet} resp.~\cite[Lemma II.2.5]{Simader-Sohr:Weak-Dirichlet}.
\end{proof}

A main ingredient of the proof of the above result was the possibility to construct extensions in $\dot{H}^1_p(\Omega)$ to traces in $\dot{W}^{1 - 1/p}_p(\partial \Omega)$.
On the other hand, \Propref{sd} also implies this possibility in an even refined sense.
\begin{proposition}\label{prop:ext}{\bfseries The Trace Space of $\dot{H}^1_p(\Omega)$.}\newline
Let $\Omega \subseteq \bR^n$ be the half-space, a bent half-space or a bounded domain with boundary $\Gamma = \partial \Omega$ of class $C^1$ and let $1 < p < \infty$.
Then the trace operator
\begin{equation*}
	[\,\cdot\,]: \dot{H}^1_p(\Omega) \longrightarrow \dot{W}^{1 - 1/p}_p(\partial \Omega)
\end{equation*}
is onto and there exist a bounded linear extension operator
\begin{equation*}
	\dot{E}: \dot{W}^{1 - 1/p}_p(\partial \Omega) \longrightarrow \dot{H}^1_p(\Omega),
\end{equation*}
which is characterized by
\begin{equation*}
	-\Delta\,\dot{E} h = 0\ \mbox{in}\ \Omega, \quad [\dot{E} h] = h\ \mbox{on}\ \partial \Omega,
	\quad \quad h \in \dot{W}^{1 - 1/p}_p(\partial \Omega).
\end{equation*}
\end{proposition}
\begin{proof}
Given $h \in \dot{W}^{1 - 1/p}_p(\partial \Omega)$, define $\dot{E} h \in \dot{H}^1_p(\Omega)$ to be the unique solution to the (weak) Dirichlet problem {(DP)}${}_{0, h}$.
\end{proof}

Another valuable consequence of \Propref{sd} is the validity of the {\itshape Weyl decomposition}
\begin{equation*}
	\label{eqn:weyl}\tag{$\cW$}
	L_p(\Omega,\,\bR^n) = S_p(\Omega) \oplus \nabla {}_0 \dot{H}^1_p(\Omega)
\end{equation*}
for all domains under consideration.
Here
\begin{equation*}
	S_p(\Omega) := \left\{\,f \in L_p(\Omega,\,\bR^n)\,:\,\mbox{div}\,f = 0\,\right\}
\end{equation*}
denotes the space of $L_p$-functions, which are divergence-free in the sense of distributions.
Since both, $S_p(\Omega)$ and $\nabla {}_0 \dot{H}^1_p(\Omega)$ are closed subspaces of $L_p(\Omega,\,\bR^n)$,
this direct decomposition is topological and the thereby induced bounded linear projection
\begin{equation*}
	\cW_p: L_p(\Omega,\,\bR^n) \longrightarrow L_p(\Omega,\,\bR^n)
\end{equation*}
onto $S_p(\Omega)$ along $\nabla {}_0 \dot{H}^1_p(\Omega)$ is called {\itshape Weyl projection}.
Given $f \in L_p(\Omega,\,\bR^n)$, we have $\cW_p f = f - \nabla q \in S_p(\Omega)$,
where $q \in {}_0 \dot{H}^1_p(\Omega)$ is obtained as the unique solution to {(DP)}${}_{-f, 0}$.

On the other hand, we may construct solutions $q \in \hat{H}^1_p(\Omega) := \dot{H}^1_p(\Omega) / \bR$ to the {\itshape Neumann problem}
\begin{equation*}
	\label{eqn:sn}\tag*{${(\mbox{NP})}_{f, h}$}
		- \Delta\,q = \mbox{div}\,f \quad \mbox{in}\ \Omega, \qquad 
		\partial_\nu q + [f] \cdot \nu = h \quad \mbox{on}\ \partial \Omega
\end{equation*}
with data $f \in L_p(\Omega,\,\bR^n)$ and $h \in \dot{W}^{-1/p}_p(\partial \Omega) := \dot{W}^{1 - 1/p^\prime}_{p^\prime}(\partial \Omega)^\prime$,
where $1/p + 1/p^\prime = 1$.
The formulation of this problem requires a proper definition of the left hand side of the boundary condition,
since $\nabla q,\,f \in L_p(\Omega,\,\bR^n)$ do not possess traces in the sense of Sobolev spaces.
However, we may employ a {\itshape generalized normal trace}
\begin{equation*}
	{[\,\cdot\,]}_\nu: S_p(\Omega) \longrightarrow \dot{W}^{-1/p}_p(\partial \Omega),
\end{equation*}
which is defined via
\begin{equation*}
	\langle\,\psi\,|\,{[v]}_\nu\,\rangle := (\,\nabla \dot{E} \psi\,|\,v\,), \quad \quad \psi \in \dot{W}^{1 - 1/p^\prime}_{p^\prime}(\partial \Omega),\ v \in S_p(\Omega).
\end{equation*}
Note, that for $\phi \in \dot{H}^1_{p^\prime}(\Omega)$ we have $\dot{E} [\phi] - \phi \in {}_0 \dot{H}^1_{p^\prime}(\Omega)$ and, therefore,
\begin{equation*}
	\cW_{p^\prime} (\nabla \dot{E} [\phi] - \nabla \phi) = 0.
\end{equation*}
Hence,
\begin{equation*}
	\begin{array}{rcl}
		(\,\nabla \dot{E} [\phi]\,|\,v\,) & = & (\,\nabla \dot{E} [\phi]\,|\,\cW_p v\,) = (\,\cW_{p^\prime} \nabla \dot{E} [\phi]\,|\,v\,) \\[0.5em]
		                                  & = & (\,\cW_{p^\prime} \nabla \phi\,|\,v\,) = (\,\nabla \phi\,|\,\cW_p v\,) = (\,\nabla \phi\,|\,v\,), \quad \phi \in \dot{H}^1_{p^\prime}(\Omega),\ v \in S_p(\Omega)
	\end{array}
\end{equation*}
and the {\itshape generalized principle of partial integration}
\begin{equation*}
	\langle\,[\phi]\,|\,{[v]}_\nu\,\rangle = (\,\nabla \phi\,|\,v\,), \quad \quad \phi \in \dot{H}^1_{p^\prime}(\Omega),\ v \in S_p(\Omega)
\end{equation*}
is available.
Especially, we have
$\langle\,\psi\,|\,{[v]}_\nu\,\rangle = (\,\nabla \phi\,|\,v\,)$, for all $\psi \in \dot{W}^{1 - 1/p^\prime}_{p^\prime}(\partial \Omega)$, $\phi \in \dot{H}^1_{p^\prime}(\Omega)$ with $[\phi] = \psi$ and $v \in S_p(\Omega)$,
which implies the definition of the generalized normal trace to be independent of the particular choice of the extension operator.
The Neumann problem {(NP)}${}_{f, h}$ has to be understood as
\begin{equation*}
	\nabla q + f \in S_p(\Omega) \quad \mbox{and} \quad {[\nabla q + f]}_\nu = h \ \mbox{on} \ \partial \Omega
\end{equation*}
and is therefore equivalent to its weak formulation
\begin{equation*}
	\label{eqn:wn}\tag*{${(\mbox{NP})}^w_{f, h}$}
	- (\,\nabla \phi\,|\,\nabla q\,) = (\,\nabla \phi\,|\,f\,) - \langle\,[\phi]\,|\,h\,\rangle, \quad \phi \in \hat{H}^1_{p^\prime}(\Omega).
\end{equation*}
The next proposition, which is proved in \cite{Simader-Sohr:Weak-Neumann}, shows, that the (weak) Neumann problem also admits maximal regular solutions.
\begin{proposition}\label{prop:sn}{\bfseries The Weak Neumann Problem.}\newline
Let $\Omega \subseteq \bR^n$ be the half-space, a bent half-space or a bounded domain boundary $\Gamma = \partial \Omega$ of class $C^1$ and let $1 < p < \infty$.
Then the (weak) Neumann problem {(NP)}${}_{f, h}$ resp.~{(NP)}${}^w_{f, h}$ admits a unique maximal regular solution $q \in \hat{H}^1_p(\Omega)$,
whenever $f \in L_p(\Omega,\,\bR^n)$ and $h \in \dot{W}^{-1/p}_p(\partial \Omega)$. \hspace*{\fill} $\Box$
\end{proposition}

A first interesting consequence of \Propref{sn} is the pendant of \Propref{ext} for the generalized normal trace operator.
\begin{proposition}\label{prop:genext}{\bfseries The Trace Space of $S_p(\Omega)$.}\newline
Let $\Omega \subseteq \bR^n$ be the half-space, a bent half-space or a bounded domain with boundary $\Gamma = \partial \Omega$ of class $C^1$ and let $1 < p < \infty$.
Then the trace operator
\begin{equation*}
	{[\,\cdot\,]}_\nu: S_p(\Omega) \longrightarrow \dot{W}^{-1/p}_p(\partial \Omega)
\end{equation*}
is onto and there exist a bounded linear extension operator
\begin{equation*}
	\dot{E}_\nu: \dot{W}^{-1/p}_p(\partial \Omega) \longrightarrow S_p(\Omega)
\end{equation*}
with
\begin{equation*}
	\dot{E}_\nu h \in \nabla \dot{H}^1_p(\Omega), \quad {[\dot{E}_\nu h]}_\nu = h\ \mbox{on}\ \partial \Omega,
	\quad \quad h \in \dot{W}^{-1/p}_p(\partial \Omega).
\end{equation*}
\end{proposition}
\begin{proof}
Given $h \in \dot{W}^{-1/p}_p(\partial \Omega)$, define $q \in \hat{H}^1_p(\Omega)$ to be the unique solution to {(NP)}${}_{0, h}$ and set $\dot{E}_\nu h := \nabla q$.
\end{proof}

Last, but not least, \Propref{sn} implies the {\itshape Helmholtz decomposition}
\begin{equation*}
	L_p(\Omega,\,\bR^n) = {}_0 S_p(\Omega) \oplus \nabla \dot{H}^1_p(\Omega)
\end{equation*}
to be valid for all domains under consideration and all $1 < p < \infty$.
Here
\begin{equation*}
	{}_0 S_p(\Omega) := \mbox{cls}\,\left(\left\{\,\phi \in C^\infty_0(\Omega,\,\bR^n)\,:\,\mbox{div}\,\phi = 0\,\right\},\,{|\,\cdot\,|}_{L_p(\Omega, \bR^n)}\right) \subseteq L_p(\Omega,\,\bR^n)
\end{equation*}
denotes the closure of the space of solenoidal, compactly supported, smooth vector fields in $L_p(\Omega,\,\bR^n)$,
which may be characterized as
\begin{equation*}
	{}_0 S_p(\Omega) = \left\{\,\phi \in S_p(\Omega)\,:\,{[\phi]}_\nu = 0\,\right\}.
\end{equation*}
Since both, ${}_0 S_p(\Omega)$ and $\nabla \dot{H}^1_p(\Omega)$ are closed subspaces of $L_p(\Omega,\,\bR^n)$,
this direct decomposition is topological and the thereby induced bounded linear projection
\begin{equation*}
	\cH_p: L_p(\Omega,\,\bR^n) \longrightarrow L_p(\Omega,\,\bR^n)
\end{equation*}
onto ${}_0 S_p(\Omega)$ along $\nabla \dot{H}^1_p(\Omega)$ is called {\itshape Helmholtz projection}.
Given $f \in L_p(\Omega,\,\bR^n)$, we have $\cH_p f = f - \nabla q \in {}_0 S_p(\Omega)$,
where $q \in \hat{H}^1_p(\Omega)$ is obtained as the unique solution to {(NP)}${}_{-f, 0}$.

With the above definitions and results at hand, we may deal with an inhomogeneous divergence condition in \ref{eqn:s-data} by defining a suitable pressure,
whose gradient forces the velocity field to develop the requested divergence.
\begin{proposition}\label{prop:div}{\bfseries A Divergence Adjusting Pressure.}\newline
Let $a > 0$ and let $\Omega \subseteq \bR^n$ be the half-space, a bent half-space or a bounded domain with boundary $\Gamma = \partial \Omega$ of class $C^3$ and let $1 < p < \infty$.
Let
\begin{equation*}
	g \in W^{1/2}_p((0,\,a),\,L_p(\Omega)) \cap L_p((0,\,a),\,H^1_p(\Omega)),
\end{equation*}
such that there exists
\begin{equation*}
	\begin{array}{c}
		\eta \in W^{1 - 1/2p}_p((0,\,a),\,L_p(\partial \Omega)) \cap L_p((0,\,a),\,W^{2 - 1/p}_p(\partial \Omega)) \\[0.5em]
		\mbox{with} \quad (\,g,\,\eta\,) \in H^1_p((0,\,a),\,{(H^1_{p^\prime}(\Omega),\,{|\,\cdot\,|}_{\dot{H}^1_{p^\prime}(\Omega)})}^\prime),
	\end{array}
\end{equation*}
where $1/p + 1/p^\prime = 1$.
Then there exists
\begin{equation*}
	\begin{array}{c}
		q \in L_p((0,\,a),\,{}_0 \dot{H}^1_p(\Omega)) \\[0.5em]
		\mbox{with} \quad - \mbox{div}\,\nabla q = (\rho \partial_t - \mu \Delta) g \quad \mbox{in} \ \cD^\prime(\Omega).
	\end{array}
\end{equation*}
\end{proposition}
\begin{proof}
Since $H^1_{p^\prime}(\Omega)$ is a dense subspace of $\dot{H}^1_{p^\prime}(\Omega)$,
there exists a unique extension $F \in H^1_p((0,\,a),\,{}_0 \dot{H}^{-1}_p(\Omega))$ to $(\,g,\,\eta\,)$,
where ${}_0 \dot{H}^1_p(\Omega) := \dot{H}^1_{p^\prime}(\Omega)^\prime$.
Now, we define $\hat{q} \in L_p((0,\,a),\,\hat{H}^1_p(\Omega))$ be the unique solution to
\begin{equation*}
	(\,\nabla \phi\,|\,\nabla \hat{q}\,) = \langle\,\phi\,|\,F\,\rangle, \quad \phi \in \hat{H}^1_{p^\prime}(\Omega),
\end{equation*}
which is available thanks to \Propref{sn}, cf.~also \cite{Simader-Sohr:Weak-Neumann}.
On one hand, we have $\partial_t \nabla \hat{q} \in L_p((0,\,a),\,L_p(\Omega,\,\bR^n))$,
and, on the other hand,
\begin{equation*}
  (\,\nabla \phi\,|\,\nabla \hat{q}\,) = \langle\,\phi\,|\,F\,\rangle = \langle\,\phi\,|\,(\,g,\,\eta\,)\,\rangle = - \int_\Omega \phi g\,\mbox{d}x,
	\quad \phi \in C^\infty_0(\Omega)
\end{equation*}
implies $\Delta \hat{q} = \mbox{div}\,\nabla \hat{q} = g$ and hence
$\Delta \nabla \hat{q} = \nabla \Delta \hat {q} = \nabla g \in L_p((0,\,a),\,L_p(\Omega,\,\bR^n))$.
Finally, we define $q \in L_p((0,\,a),\,{}_0 \dot{H}^1_p(\Omega))$ via $\nabla q = - (\cI - \cW_p)(\rho \partial_t - \mu \Delta) \nabla \hat{q}$
and infer
\begin{equation*}
	\begin{array}{rcl}
		- \mbox{div}\,\nabla q & = & \mbox{div}\,(\cI - \cW_p)(\rho \partial_t - \mu \Delta) \nabla \hat{q} = \mbox{div}\,(\rho \partial_t - \mu \Delta) \nabla \hat{q} \\[0.5em]
		                       & = & (\rho \partial_t - \mu \Delta)\,\mbox{div}\,\nabla \hat{q} = (\rho \partial_t - \mu \Delta) g \quad \mbox{in} \ \cD^\prime(\Omega).
	\end{array}
\end{equation*}
\end{proof}

Using \Propref{div}, we may formulate the promised splitting scheme,
which may serve to eliminate almost all inhomogeneous data for the Stokes equations \ref{eqn:s-data}.
\begin{theorem}\label{thm:splitting}{\bfseries A Splitting Scheme.}\newline
Let $a > 0$ and let $\Omega \subseteq \bR^n$ be the half-space, a bent half-space or a bounded domain.
Let $\Gamma = \partial \Omega$ be of class $C^{3-}$ and let $1 < p < \infty$, $p \neq \frac{3}{2},\,3$.
Let $\cB = \cB^{\alpha, \beta}$ with $\alpha,\,\beta \in \{\,-1,\,0,\,1\,\}$ be one of the linear operators ($\cB$),
which realizes one of the boundary conditions (B).

If $\beta =  0$ let $\gamma = -\infty$;
if $\beta =  1$ let $\gamma \in \{\,-\infty\,\} \cup [0,\,1/2 - 1/2p]$;
if $\beta = -1$ let $\gamma \in \{\,-\infty\,\} \cup [0,\,\infty)$.

Then the unique maximal regular solution
\begin{equation*}
	u \in \bX_u(a), \quad p \in \bX^\beta_{p, \gamma}(a)
\end{equation*}
to the Stokes equations {(S)}${}^{a, \Omega, \cB}_{f, g, h, u_0}$ with data
\begin{equation*}
	f \in \bY_f(a), \quad g \in \bY_g(a), \quad h \in \bY^{\alpha, \beta}_{h, \gamma}(a), \quad u_0 \in \bY_u
\end{equation*}
satisfying the compatibility conditions {(C)}${}^{\alpha, \beta}_{g, h, u_0}$ may be obtained as follows:
\begin{enumerate}[1.]
	\item
		If $\beta = 0$, choose $\eta = h \cdot \nu$,
		define $\hat{q} \in L_p((0,\,a),\,{}_0 \dot{H}^1_p(\Omega))$ via \Propref{div} based on $g$ and $\eta$ and set $\bar{q} = \hat{q} + \bR$;
		if $\beta \in \{\,-1,\,1\,\}$, choose $\eta \in \bN^0_{h, -\infty}(a)$ according to the compatibility condition (C3)${}^\beta_{g, h, u_0}$,
		define $\hat{q} \in L_p((0,\,a),\,{}_0 \dot{H}^1_p(\Omega))$ via \Propref{div} based on $g$ and $\eta$ and set $\bar{q} = \hat{q} - \dot{E} (h \cdot \nu)$.
	\item
		Choose $v \in \bX_u(a)$ to be the unique maximal regular solution to the parabolic problem
		\begin{equation*}
			\begin{array}{c}
				\rho \partial_t v - \mu \Delta v = - \nabla \bar{q} + \rho \cW_p f, \quad \mbox{in}\ (0,\,a) \times \Omega, \\[0.5em]
				P_\Gamma \cB^{\alpha}(v) = P_\Gamma h, \quad [\mbox{div}\,v] = [g] \quad \mbox{on}\ (0,\,a) \times \partial \Omega, \\[0.5em]
				v(0) = u_0 \quad \mbox{in}\ \Omega.
			\end{array}
		\end{equation*}
		and define $q \in \bX^\beta_{p, \gamma}(a)$ by means of $\nabla q = \nabla \bar{q} + \rho (\cI - \cW_p) f$.
		Then, if $\beta = -1$, $u = v$ and $p = q$ constitute the unique maximal regular solution to the Stokes equations {(S)}${}^{a, \Omega, \cB}_{f, g, h, u_0}$.
	\item
		If $\beta \in \{\,0,\,1\,\}$, the solution to the Stokes equations {(S)}${}^{a, \Omega, \cB}_{f, g, h, u_0}$ splits as $u = v + \bar{u}$ and $p = q + \bar{p}$, where
		\begin{equation*}
			\bar{u} \in \bX_u(a), \quad \bar{p} \in \bX^\beta_{p, \gamma}(a)
		\end{equation*}
		denotes the unique maximal regular solution to the Stokes equations {(S)}${}^{a, \Omega, \cB}_{0, 0, \bar{h}, 0}$ with $\bar{h} = h - \cB^{\alpha, \beta}(v,\,q)$.
		In particular $P_\Gamma \bar{h} = 0$
		and
		\begin{equation*}
			Q_\Gamma \bar{h} \in {}_0 H^1_p((0,\,a),\,\dot{W}^{-1/p}_p(\Gamma,\,N\Gamma)) \cap \bN^0_{h, -\infty}(a), \quad \mbox{if} \ \beta = 0,
		\end{equation*}
		resp.
		\begin{equation*}
			Q_\Gamma \bar{h} \in \bN^{+1}_{h, 1/2 - 1/2p}(a), \quad \mbox{if} \ \beta = +1.
		\end{equation*}
\end{enumerate}
\end{theorem}
Before we prove this theorem, some remarks seem to be in order.
In step 2 we used the abbreviation
\begin{equation*}
	P_\Gamma \cB^\alpha(v) = P_\Gamma \cB^{\alpha, \beta}(v,\,q),
\end{equation*}
since the tangential part of the boundary condition neither depends on the parameter $\beta$ nor on the pressure $q$.
Analogously, we will make frequently use of the abbreviation
\begin{equation*}
	Q_\Gamma \cB^\beta(v,\,p) = Q_\Gamma \cB^{\alpha, \beta}(v,\,q),
\end{equation*}
since the normal part of the boundary condition does not depend on the parameter $\alpha$.
The parabolic problem employed in step 2 will be treated in the appendix.
Indeed, there exists a unique maximal regular solution $v \in \bX_u(a)$, thanks to the regularity assumptions on the data
and the compatibility conditions (C1)${}_{g, u_0}$ and (C2)${}^\alpha_{h, u_0}$.
The overall strategy for the proof of \Thmref{splitting} in connection with the proof of \Thmref{s} will be as follows.
\begin{remark}\label{rem:splitting}
In the following proof of \Thmref{splitting},
we will only show, that the solution constructed in steps 1 and 2 has the desired regularity properties and solves the Stokes equations,
where the normal part of the desired boundary condition may be violated if $\beta \in \{\,0,\,1\,\}$.
In these cases, the first two steps reduce the task to prove the existence of unique maximal regular solutions to the Stokes equations as claimed in \Thmref{s}
to the special case $f = 0$, $g = 0$, $P_\Gamma h = 0$ and $u_0 = 0$
with $Q_\Gamma h \in H^1_p((0,\,a),\,\dot{W}^{-1/p}_p(\Gamma,\,N\Gamma)) \cap \bN^0_{h, -\infty}(a)$, if $\beta = 0$, resp.~$\gamma = 1/2 - 1/2p$, if $\beta = 1$.
Once this additional task is accomplished, \Thmref{s} and \Thmref{splitting} will be completely proved in these cases, cf.~also \Remref{s}.
On the other hand, if $\beta = -1$, the first two steps establish the existence of a maximal regular solution to the Stokes equations and it remains to prove its uniqueness
to complete the proof of \Thmref{s} and \Thmref{splitting}.
However, if $(u,\,p) \in \bX_u(a) \times \bX^{-1}_{p, \infty}(a)$ constitutes a maximal regular solution to the Stokes equations {(S)}${}^{a, \Omega, \cB}_{0, 0, 0, 0}$
and $\beta = -1$, then $p \in L_p((0,\,a),\,{}_0 \dot{H}^1_p(\Omega))$ and the momentum balance delivers
\begin{equation*}
	-(\,\nabla \phi\,|\,\nabla p\,) = 0, \quad \quad \phi \in C^\infty_0(\Omega),
\end{equation*}
which implies $p = 0$ by \Propref{sn}, and, hence, $u = 0$ by uniqueness of the solution to the corresponding parabolic problem.
Therefore, the proofs of \Thmref{s} and \Thmref{splitting} for the case $\beta = -1$ will be complete.
\end{remark}
\begin{proof}[Proof of \Thmref{splitting}]
First note, that by construction we always have
\begin{equation*}
	\hat{q} \in L_p((0,\,a),\,{}_0 \dot{H}^1_p(\Omega)), \quad - \mbox{div}\,\nabla \hat{q} = (\rho \partial_t - \mu \Delta) g \ \mbox{in} \ \cD^\prime(\Omega)
\end{equation*}
and therefore
\begin{equation*}
	\bar{q} \in \bX^\beta_{p, \gamma}(a), \quad - \mbox{div}\,\nabla \bar{q} = (\rho \partial_t - \mu \Delta) g \ \mbox{in} \ \cD^\prime(\Omega)
\end{equation*}
and $\bar{q}$ satisfies the boundary condition
\begin{equation*}
	-[\bar{q}] = h \cdot \nu, \quad \mbox{if} \ \beta \in \{\,-1,\,1\,\}.
\end{equation*}
Hence, $q \in \bX^\beta_{p, \gamma}(a)$ enjoys the same property on the boundary.
Moreover, $v \in \bX_u(a)$ by the maximal regularity property of the parabolic system and
\begin{equation*}
	\rho \partial_t v - \mu \Delta v + \nabla q = \rho \partial_t v - \mu \Delta v + \nabla \bar{q} + \rho (\cI - \cW_p) f = \rho f \quad \mbox{in}\ (0,\,a) \times \Omega,
\end{equation*}
which shows, that the momentum balance is valid for $v$ and $q$.
Applying the divergence in the sense of distributions to the partial differential equation used to obtain $v$
and using the properties of $\bar{q}$ as well as the compatibility condition (C1)${}_{g, u_0}$, we derive
\begin{equation*}
	\begin{array}{c}
		\rho \partial_t (\mbox{div}\,v - g) - \mu \Delta (\mbox{div}\,v - g) = 0 \quad \mbox{in}\ (0,\,a) \times \Omega, \\[0.5em]
		[\mbox{div}\,v - g] = 0 \quad \mbox{on}\ (0,\,a) \times \partial \Omega, \\[0.5em]
		(\mbox{div}\,v - g)(0) = 0 \quad \mbox{in}\ \Omega.
	\end{array}
\end{equation*}
Hence, $\mbox{div}\,v = g$ by uniqueness of weak solutions to the diffusion equation with Dirichlet boundary conditions.
Last, but not least, $v$ satisfies the tangential part of the desired boundary condition and the initial condition $v(0) = u_0$.

Now, if $\beta = -1$, $v$ and $q$ constitute the unique maximal regular solution to the Stokes equations {(S)}${}^{a, \Omega, \cB}_{f, g, h, u_0}$.
On the other hand, if $\beta \in \{\,0,\,1\,\}$, $\bar{u} = u - v$ and $\bar{p} = p - q$ solve the Stokes equations
\begin{equation*}
	\begin{array}{c}
		\rho\partial_t \bar{u} - \mu \Delta \bar{u} + \nabla \bar{p} = 0,
		\quad \quad \mbox{div}\,\bar{u} = 0
			\quad \quad \mbox{in}\ (0,\,a) \times \Omega, \\[0.5em]
		P_\Gamma \cB^\alpha(\bar{u}) = 0, \quad \quad Q_\Gamma \cB^\beta(\bar{u},\,\bar{p}) = \hat{h}
			\quad \quad \mbox{on}\ (0,\,a) \times \partial\Omega, \\[0.5em]
		\bar{u}(0) = 0
			\quad \quad \mbox{in}\ \Omega
	\end{array}
\end{equation*}
with $\hat{h} = Q_\Gamma (h - [v])$, if $\beta = 0$, resp.~$\hat{h} = - 2 \mu\,\partial_\nu v \cdot \nu$, if $\beta = 1$.
\end{proof}

\section{The Halfspace Case}\label{sec:halfspace}
This section is devoted to the proof of \Thmref{s} in the halfspace case,
i.\,e. we assume $a > 0$, $\Omega = \bR^n_+$ and $1 < p < \infty$ with $p \neq \frac{3}{2},\,3$.
Following \Remref{s} it is sufficient to prove existence of a unique maximal regular solution for all data satisfying the stated regularity and compatibility conditions.
However, following \Remref{splitting} we may restrict the proof to the case $\alpha \in \{\,-1,\,0,\,+1\,\}$, $\beta \in \{\,0,\,+1\,\}$, $f = 0$, $g = 0$, $P_\Gamma h = 0$ and $u_0 = 0$.
Moreover, it is sufficient to consider the shifted equations with $a = \infty$,
since these also deliver maximal regular solutions to the Stokes equations on finite time intervals.

Due to the simple geometry of the domain we may split the spatial variable into a tangential part $x \in \bR^{n - 1}$ and a normal part $y > 0$.
Moreover, we may split the velocity field as $u = (v,\,w)$ into a tangential part $v: [0,\,\infty) \times \bR^n_+ \longrightarrow \bR^{n - 1}$ and a normal part $w: [0,\,\infty) \times \bR^n_+ \longrightarrow \bR$
to obtain the system of interior partial differential equations
\begin{equation*}
	\begin{array}{rcll}
		\rho \epsilon v + \rho \partial_t v - \mu \Delta_x v - \mu \partial^2_y v + \nabla_x p   & = & 0 & \quad \mbox{in} \ (0,\,\infty) \times \bR^n_+, \\[0.5em]
		\rho \epsilon w + \rho \partial_t w - \mu \Delta_x w - \mu \partial^2_y w + \partial_y p & = & 0 & \quad \mbox{in} \ (0,\,\infty) \times \bR^n_+, \\[0.5em]
		\nabla_x \cdot v + \partial_y w                                                          & = & 0 & \quad \mbox{in} \ (0,\,\infty) \times \bR^n_+,
	\end{array}
\end{equation*}
which have to be complemented by the given initial and boundary conditions.
The arbitrary parameter $\epsilon > 0$ denotes the shift.
Due to $v(0) = 0$ and $w(0) = 0$ we may employ a Laplace transformation in time and a Fourier transformation in the tangential variable to obtain the transformed system
\begin{equation*}
	\begin{array}{rcll}
		\omega^2 \hat{v} - \mu \partial^2_y \hat{v} + i \xi \hat{p}      & = & 0 & \quad \mbox{Re}\,\lambda \geq 0,\,\xi \in \bR^{n - 1},\,y > 0, \\[0.5em]
		\omega^2 \hat{w} - \mu \partial^2_y \hat{w} + \partial_y \hat{p} & = & 0 & \quad \mbox{Re}\,\lambda \geq 0,\,\xi \in \bR^{n - 1},\,y > 0, \\[0.5em]
		i \xi^{\sf{T}} \hat{v} + \partial_y \hat{w}                      & = & 0 & \quad \mbox{Re}\,\lambda \geq 0,\,\xi \in \bR^{n - 1},\,y > 0,
	\end{array}
\end{equation*}
where $\hat{v}$, $\hat{w}$ and $\hat{p}$ denote the transformed components of the solution,
$\lambda \in \bC$ denotes the Laplace co-variable and $\xi \in \bR^{n - 1}$ denotes the Fourier co-variable of $x$.
Moreover, we used the abbreviations
\begin{equation*}
	\lambda_\epsilon := \epsilon + \lambda, \qquad \omega := \sqrt{\rho \lambda_\epsilon + \mu |\xi|^2}.
\end{equation*}
Now, an exponential ansatz leads to the representation
\widearray
\begin{equation*}
	\left[\begin{array}{c} \hat{v}(\lambda,\,\xi,\,y) \\[0.5em] \hat{w}(\lambda,\,\xi,\,y) \\[0.5em] \hat{p}(\lambda,\,\xi,\,y) \end{array}\right]
	= \left[\begin{array}{rr} \omega & - i \zeta \\[0.5em] i \zeta^{\sf{T}} & |\zeta| \\[0.5em] 0 & \kappa \lambda_\epsilon \end{array}\right]
	  \left[\begin{array}{rl} \hat{z}_v(\lambda,\,\xi) & \!\!\!\!\!\!\! e^{- \frac{\omega}{\sqrt{\mu}} y} \\[0.5em] \hat{z}_w(\lambda,\,\xi) & \!\!\!\!\!\!\! e^{- |\xi| y} \end{array}\right],
	\ \begin{array}{l} \mbox{Re}\,\lambda \geq 0, \\[0.5em] \xi \in \bR^{n - 1},\,y > 0, \end{array}
\end{equation*}
\narrowarray
where we have set $\zeta := \sqrt{\mu}\,\xi$ and $\kappa := \rho \sqrt{\mu}$
and where $\hat{z}_v$ and $\hat{z}_w$ denote the transformed components of a function $z = (z_v,\,z_w): [0,\,\infty) \times \bR^{n - 1} \longrightarrow \bR^n$,
which has to be determined via the boundary conditions.
We will treat the various boundary conditions under consideration separately in the next subsections.
Note, that the normal component of the boundary datum $h$ is the only non-zero part of the data.
In the sequel we will denote it by $h_w = - h \cdot \nu = h_n$.
Also note, that every boundary condition will lead to a linear system
\begin{equation*}
	\hat{\cB}^{\alpha, \beta}(\lambda,\,\xi) \left[\begin{array}{c} \hat{z}_v(\lambda,\,\xi) \\[0.5em] \hat{z}_w(\lambda,\,\xi) \end{array}\right] = \left[\begin{array}{c} 0 \\[0.5em] \hat{h}_w(\lambda,\,\xi) \end{array}\right],
	\quad \mbox{Re}\,\lambda > 0,\,\xi \in \bR^{n - 1},
\end{equation*}
with a family of linear operators $\hat{\cB}^{\alpha, \beta}(\lambda,\,\xi)$, which will be shown to uniquely determine $\hat{z}(\lambda,\,\xi)$ for all $\mbox{Re}\,\lambda \geq 0$ and $\xi \in \bR^{n - 1}$.

\subsection*{The Case $\alpha = 0$ and $\beta = 0$}
Here we have to treat the plain Dirichlet conditions
\begin{equation*}
	[v] = 0 \quad \mbox{and} \quad [w] = h_w \quad \mbox{on} \ (0,\,\infty) \times \partial \bR^n_+,
\end{equation*}
and we may assume $h_w \in {}_0 \dot{H}^1_p((0,\,\infty),\,\dot{W}^{-1/p}_p(\bR^{n - 1})) \cap \bN^0_{h, -\infty}(\infty)$ by \Thmref{splitting}.
The boundary conditions lead to
\widearray
\begin{equation*}
	\hat{\cB}^{0, 0}(\lambda,\,\zeta)
	= \left[\begin{array}{rr} \omega & - i \zeta \\[1em] i \zeta^{\sf{T}} & |\zeta| \end{array}\right]
	= \left[\begin{array}{rr} \omega & 0 \\[0.5em] 0 & \omega \end{array}\right]
	  \left[\begin{array}{rr} 1 & - \frac{i \zeta}{\omega} \\[0.5em] \frac{i \zeta^{\sf{T}}}{\omega} & \frac{|\zeta|}{\omega} \end{array}\right].
\end{equation*}
\narrowarray
Now, we have
\widearray
\begin{equation*}
	{\left[\begin{array}{rr} 1 & - \frac{i \zeta}{\omega} \\[0.5em] \frac{i \zeta^{\sf{T}}}{\omega} & \frac{|\zeta|}{\omega} \end{array}\right]}^{-1}
	= {\left\{ \left( 1 - \frac{|\zeta|}{\omega} \right) \frac{|\zeta|}{\omega} \right\}}^{-1}
	  \left[\begin{array}{rr} \left( 1 - \frac{|\zeta|}{\omega} \right) \frac{|\zeta|}{\omega} - \frac{i \zeta \otimes i \zeta}{\omega^2} & \frac{i \zeta}{\omega} \\[0.5em] - \frac{i \zeta^{\sf{T}}}{\omega} & 1 \end{array}\right]
\end{equation*}
\narrowarray
and, hence,
\begin{equation*}
	\hat{z}_w
	= {\left\{ \left( 1 - \frac{|\zeta|}{\omega} \right) \frac{|\zeta|}{\omega} \right\}}^{-1} \omega^{-1} \hat{h}_w
	= {\left( 1 - \frac{|\zeta|}{\omega} \right)}^{-1} {|\zeta|}^{-1} \hat{h}_w.
\end{equation*}
This implies
\begin{equation*}
	\begin{array}{rcl}
		\widehat{\partial_\nu p} & = & - [\partial_y \hat{p}] = \frac{\kappa}{\sqrt{\mu}} \lambda_\epsilon |\zeta| \hat{z}_w = \rho \lambda_\epsilon {\left( 1 - \frac{|\zeta|}{\omega} \right)}^{-1} \hat{h}_w \\[1em]
		                         & = & \rho \lambda_\epsilon \ \frac{1 + \frac{|\zeta|}{\omega}}{1 - \frac{|\zeta|^2}{\omega^2}} \ \hat{h}_w = \omega(\omega + |\zeta|) \hat{h}_w
	\end{array}
\end{equation*}
and, hence, the desired solution may be obtained by solving the splitting scheme
\begin{equation*}
	\begin{array}{c}
		\rho \epsilon u + \rho \partial_t u - \mu \Delta u = - \nabla p, \quad - \Delta p = 0 \quad \mbox{in} \ (0,\,\infty) \times \bR^n_+ \\[0.5em]
		P_\Gamma \cB^0(u) = 0, \quad [\mbox{div}\,u] = 0, \quad \partial_\nu p = T^0 h_w \quad \mbox{on} \ (0,\,\infty) \times \partial \bR^n_+ \\[0.5em]
		u(0) = 0 \quad \mbox{in} \ \bR^n_+,
	\end{array}
\end{equation*}
where the bounded linear operator
\begin{equation*}
	T^0: {}_0 H^1_p((0,\,a),\,\dot{W}^{-1/p}_p(\bR^{n - 1})) \cap \bN^0_{h, -\infty}(\infty) \longrightarrow L_p((0,\,a),\,\dot{W}^{-1/p}_p(\bR^{n - 1}))
\end{equation*}
is defined via its Laplace-Fourier symbol
\begin{equation*}
	m_0(\lambda,\,\xi) = \omega(\omega + |\zeta|),
\end{equation*}
i.\,e. $T^0 = {(\rho \epsilon + \rho \partial_t - \mu \Delta_\Gamma)}^{1/2} {(\rho \epsilon + \rho \partial_t - 2 \mu \Delta_\Gamma)}^{1/2}$,
e.\,g. by the $\cH^\infty$-calculi of the involved operators.
Indeed, the calculations in Appendix~B imply the solution of the above splitting scheme to solve the shifted Stokes equations, cf.\ relation \eqnref{trace:0-0}.
On the other hand, any solution to the shifted Stokes equations with $h_w = 0$ satisfies the equations of the above splitting scheme with a boundary condition
\begin{equation*}
	\partial_\nu p = h_p \quad \mbox{on} \ (0,\,\infty) \times \partial \bR^n_+
\end{equation*}
for some $h_p \in L_p((0,\,a),\,\dot{W}^{-1/p}_p(\bR^{n - 1}))$.
Employing relation \eqnref{trace:0-0} once again, we infer $h_p = T^0 [w] = T^0 h_w = 0$.
Therefore, $p \equiv const.$ and $u = 0$.

\subsection*{The case $\alpha = \pm 1$ and $\beta = 0$}
In this case the boundary conditions read
\begin{equation*}
	\mp \mu [\partial_y v] - \mu \nabla _x [w] = 0 \quad \mbox{and} \quad [w] = h_w \quad \mbox{on} \ (0,\,\infty) \times \partial \bR^n_+
\end{equation*}
and we may assume $h_w \in {}_0 \dot{H}^1_p((0,\,\infty),\,\dot{W}^{-1/p}_p(\bR^{n - 1})) \cap \bN^0_{h, -\infty}(\infty)$ by \Thmref{splitting}.
The boundary conditions lead to
\widearray
\begin{equation*}
	\begin{array}{rcl}
		\cB^{\pm 1, 0}(\lambda,\,\xi)
		& = & \left[\begin{array}{rr} \pm \sqrt{\mu} \omega^2 - \sqrt{\mu} (i \zeta \otimes i \zeta) & - \sqrt{\mu} (i \zeta |\zeta| \pm i \zeta |\zeta|) \\[0.5em] i \zeta^{\sf{T}} & |\zeta| \end{array}\right] \\[2em]
		& = & \left[\begin{array}{rr} \pm \sqrt{\mu} \omega^2 & 0 \\[0.5em] 0 & \omega \end{array}\right] \left[\begin{array}{rr} 1 \mp \frac{i \zeta \otimes i \zeta}{\omega^2} & - (\frac{i \zeta}{\omega} \frac{|\zeta|}{\omega} \pm \frac{i \zeta}{\omega} \frac{|\zeta|}{\omega}) \\[0.5em] \frac{i \zeta^{\sf{T}}}{\omega} & \frac{|\zeta|}{\omega} \end{array}\right].
	\end{array}
\end{equation*}
\narrowarray
Now, we have
\widearray
\begin{equation*}
	\begin{array}{l}
		{\left[\begin{array}{rr} 1 \mp \frac{i \zeta \otimes i \zeta}{\omega^2} & - (\frac{i \zeta}{\omega} \frac{|\zeta|}{\omega} \pm \frac{i \zeta}{\omega} \frac{|\zeta|}{\omega}) \\[0.5em] \frac{i \zeta^{\sf{T}}}{\omega} & \frac{|\zeta|}{\omega} \end{array}\right]}^{-1} \\[2em]
		\quad = {\left\{ \left( 1 - \frac{|\zeta|^2}{\omega^2} \right) \frac{|\zeta|}{\omega} \right\}}^{-1}
		        \left[\begin{array}{rr} \left( 1 - \frac{|\zeta|^2}{\omega^2} \right) \frac{|\zeta|}{\omega} - \frac{i \zeta \otimes i \zeta}{\omega^2} \frac{|\zeta|}{\omega} & \frac{i \zeta}{\omega} \frac{|\zeta|}{\omega} \pm \frac{i \zeta}{\omega} \frac{|\zeta|}{\omega} \\[0.5em] - \frac{i \zeta^{\sf{T}}}{\omega} & 1 \pm \frac{|\zeta|^2}{\omega^2} \end{array}\right]
	\end{array}
\end{equation*}
\narrowarray
and, hence,
\begin{equation*}
	\begin{array}{rcl}
		\hat{z}_w
		& = & {\left\{ \left( 1 - \frac{|\zeta|^2}{\omega^2} \right) \frac{|\zeta|}{\omega} \right\}}^{-1} \left( 1 \pm \frac{|\zeta|^2}{\omega^2} \right) \omega^{-1} \hat{h}_w \\[1.5em]
		& = & {\left( 1 - \frac{|\zeta|^2}{\omega^2} \right)}^{-1} \left( 1 \pm \frac{|\zeta|^2}{\omega^2} \right) {|\zeta|}^{-1} \hat{h}_w.
	\end{array}
\end{equation*}
This implies
\begin{equation*}
	\begin{array}{rcl}
		\widehat{\partial_\nu p} & = & - [\partial_y \hat{p}] = \frac{\kappa}{\sqrt{\mu}} \lambda_\epsilon |\zeta| \hat{z}_w = \rho \lambda_\epsilon {\left( 1 - \frac{|\zeta|^2}{\omega^2} \right)}^{-1} \left( 1 \pm \frac{|\zeta|^2}{\omega^2} \right) \hat{h}_w \\[1em]
		                         & = & \rho \lambda_\epsilon \ \frac{1 \pm \frac{|\zeta|^2}{\omega^2}}{1 - \frac{|\zeta|^2}{\omega^2}} \ \hat{h}_w = (\omega^2 \pm |\zeta|^2) \hat{h}_w
	\end{array}
\end{equation*}
and, hence, the desired solution may be obtained by solving the splitting scheme
\begin{equation*}
	\begin{array}{c}
		\rho \epsilon u + \rho \partial_t u - \mu \Delta u = - \nabla p, \quad - \Delta p = 0 \quad \mbox{in} \ (0,\,\infty) \times \bR^n_+ \\[0.5em]
		P_\Gamma \cB^{\pm 1}(u) = 0, \quad [\mbox{div}\,u] = 0, \quad \partial_\nu p = T^{\pm 1} h_w \quad \mbox{on} \ (0,\,\infty) \times \partial \bR^n_+ \\[0.5em]
		u(0) = 0 \quad \mbox{in} \ \bR^n_+,
	\end{array}
\end{equation*}
where the bounded linear operators
\begin{equation*}
	T^{\pm 1}: {}_0 H^1_p((0,\,a),\,\dot{W}^{-1/p}_p(\bR^{n - 1})) \cap \bN^0_h(\infty) \longrightarrow L_p((0,\,a),\,\dot{W}^{-1/p}_p(\bR^{n - 1}))
\end{equation*}
are defined via their Laplace-Fourier symbols
\begin{equation*}
	m_{\pm 1}(\lambda,\,\xi) = \omega^2 \pm |\zeta|^2,
\end{equation*}
i.\,e. $T^{+1} = (\rho \epsilon + \rho \partial_t - 2 \mu \Delta_\Gamma)$ and $T^{-1} = (\rho \epsilon + \rho \partial_t)$,
e.\,g. by the $\cH^\infty$-calculi of the involved operators.
Indeed, the calculations in Appendix~B imply the solution of the above splitting scheme to solve the shifted Stokes equations, cf.\ relation \eqnref{trace:1-0}.
On the other hand, any solution to the shifted Stokes equations with $h_w = 0$ satisfies the equations of the above splitting scheme with a boundary condition
\begin{equation*}
	\partial_\nu p = h_p \quad \mbox{on} \ (0,\,\infty) \times \partial \bR^n_+
\end{equation*}
for some $h_p \in L_p((0,\,a),\,\dot{W}^{-1/p}_p(\bR^{n - 1}))$.
Employing relation \eqnref{trace:1-0} once again, we infer $h_p = T^{\pm 1} [w] = T^{\pm 1} h_w = 0$.
Therefore, $p \equiv const.$ and $u = 0$.

\subsection*{The Case $\alpha = 0$ and $\beta = +1$}
In this case the boundary conditions read
\begin{equation*}
	[v] = 0 \quad \mbox{and} \quad - 2 \mu [\partial_y w] + [p] = h_w \quad \mbox{on} \ (0,\,\infty) \times \partial \bR^n_+,
\end{equation*}
and we may assume $\gamma = 1/2 - 1/2p$ by \Thmref{splitting}.
The boundary conditions read
\widearray
\begin{equation*}
	\begin{array}{rcl}
		\hat{\cB}^{0, +1}(\lambda,\,\zeta)
		& = & \left[\begin{array}{rr} \omega & - i \zeta \\[1em] 2 \sqrt{\mu} \omega i \zeta^{\sf{T}} & \kappa \lambda_\epsilon + 2 \sqrt{\mu} |\zeta|^2 \end{array}\right] \\[2em]
		& = & \left[\begin{array}{rr} \omega & 0 \\[0.5em] 0 & 2 \sqrt{\mu} \omega^2 \end{array}\right]
	        \left[\begin{array}{rr} 1 & - \frac{i \zeta}{\omega} \\[0.5em] \frac{i \zeta^{\sf{T}}}{\omega} & \frac{1}{2} + \frac{1}{2} \frac{|\zeta|^2}{\omega^2} \end{array}\right].
	\end{array}
\end{equation*}
\narrowarray
Now, we have
\widearray
\begin{equation*}
	{\left[\begin{array}{rr} 1 & - \frac{i \zeta}{\omega} \\[0.5em] \frac{i \zeta^{\sf{T}}}{\omega} & \frac{1}{2} + \frac{1}{2} \frac{|\zeta|^2}{\omega^2} \end{array}\right]}^{-1}
	= {\left\{ {\textstyle{\frac{1}{2} \left( 1 - \frac{|\zeta|^2}{\omega^2} \right)}} \right\}}^{-1}
	  \left[\begin{array}{rr} \frac{1}{2} \left( 1 - \frac{|\zeta|^2}{\omega^2} \right) - \frac{i \zeta \otimes i \zeta}{\omega^2} & \frac{i \zeta}{\omega} \\[0.5em] - \frac{i \zeta^{\sf{T}}}{\omega} & 1 \end{array}\right]
\end{equation*}
\narrowarray
and, hence,
\begin{equation*}
	\hat{z}_w
	= {\left\{ \frac{1}{2} \left( 1 - \frac{|\zeta|^2}{\omega^2} \right) \right\}}^{-1} \frac{1}{2 \sqrt{\mu}} \omega^{-2} \hat{h}_w
	= {\left( 1 - \frac{|\zeta|^2}{\omega^2} \right)}^{-1} \frac{1}{\sqrt{\mu}} \omega^{-2} \hat{h}_w.
\end{equation*}
This implies
\begin{equation*}
	\begin{array}{rcl}
		[\hat{p}] & = & \kappa \lambda_\epsilon \hat{z}_w = \rho \lambda_\epsilon {\left( 1 - \frac{|\zeta|^2}{\omega^2} \right)}^{-1} \omega^{-2} \hat{h}_w \\[1em]
		          & = & \rho \lambda_\epsilon \ \frac{1}{1 - \frac{|\zeta|^2}{\omega^2}} \ \omega^{-2} \hat{h}_w = \hat{h}_w
	\end{array}
\end{equation*}
and, hence, the desired solution may be obtained by solving the splitting scheme
\begin{equation*}
	\begin{array}{c}
		\rho \epsilon u + \rho \partial_t u - \mu \Delta u = - \nabla p, \quad - \Delta p = 0 \quad \mbox{in} \ (0,\,\infty) \times \bR^n_+ \\[0.5em]
		P_\Gamma \cB^0(u) = 0, \quad [\mbox{div}\,u] = 0, \quad [p] = S^0 h_w \quad \mbox{on} \ (0,\,\infty) \times \partial \bR^n_+ \\[0.5em]
		u(0) = 0 \quad \mbox{in} \ \bR^n_+,
	\end{array}
\end{equation*}
where the bounded linear operator
\begin{equation*}
	S^0: \bN^1_{h, 1/2 - 1/2p}(\infty) \longrightarrow \bN^1_{h, 1/2 - 1/2p}(\infty)
\end{equation*}
is simply the identity.
Obviously, a solution of the above splitting scheme solves the shifted Stokes equations.
On the other hand, any solution to the shifted Stokes equations with $h_w = 0$ satisfies the equations of the above splitting scheme with a boundary condition
\begin{equation*}
	[p] = 0 \quad \mbox{on} \ (0,\,\infty) \times \partial \bR^n_+
\end{equation*}
and we infer $p = 0$ as well as $u = 0$.

\subsection*{The case $\alpha = \pm 1$ and $\beta = +1$}
In this case the boundary conditions read
\begin{equation*}
	\mp \mu [\partial_y v] - \mu \nabla _x [w] = 0 \quad \mbox{and} \quad - 2 \mu [\partial_y w] + [p] = h_w \quad \mbox{on} \ (0,\,\infty) \times \partial \bR^n_+
\end{equation*}
and we may assume $\gamma = 1/2 - 1/2p$ by \Thmref{splitting}.
The boundary conditions read
\widearray
\begin{equation*}
	\begin{array}{rcl}
		\hat{\cB}^{\pm 1, +1}(\lambda,\,\zeta) \hspace*{-1em}
		& = \hspace*{-1em} & \left[\begin{array}{rr} \pm \sqrt{\mu} \omega^2 - \sqrt{\mu} (i \zeta \otimes i \zeta) & - \sqrt{\mu} (i \zeta |\zeta| \pm i \zeta |\zeta|) \\[1em] 2 \sqrt{\mu} \omega i \zeta^{\sf{T}} & \kappa \lambda_\epsilon + 2 \sqrt{\mu} |\zeta|^2 \end{array}\right] \\[2em]
		& = \hspace*{-1em} & \left[\begin{array}{rr} {\scriptstyle{\pm \sqrt{\mu} \omega^2}} & 0 \\[0.5em] 0 & {\scriptstyle{2 \sqrt{\mu} \omega^2}} \end{array}\right] \left[\begin{array}{rr} 1 \mp \frac{i \zeta \otimes i \zeta}{\omega^2} & - (\frac{i \zeta}{\omega} \frac{|\zeta|}{\omega} \pm \frac{i \zeta}{\omega} \frac{|\zeta|}{\omega}) \\[0.5em] \frac{i \zeta^{\sf{T}}}{\omega} & \frac{1}{2} + \frac{1}{2} \frac{|\zeta|^2}{\omega^2} \end{array}\right].
	\end{array}
\end{equation*}
\narrowarray
Now, we have
\widearray
\begin{equation*}
	\begin{array}{l}
		{\left[\begin{array}{rr} 1 \mp \frac{i \zeta \otimes i \zeta}{\omega^2} & - (\frac{i \zeta}{\omega} \frac{|\zeta|}{\omega} \pm \frac{i \zeta}{\omega} \frac{|\zeta|}{\omega}) \\[0.5em] \frac{i \zeta^{\sf{T}}}{\omega} & \frac{1}{2} + \frac{1}{2} \frac{|\zeta|^2}{\omega^2} \end{array}\right]}^{-1} \\[2em]
		\qquad = \delta^{-1} \left[\begin{array}{rr} \delta \pm \left\{ \frac{1}{2} \left( 1 + \frac{|\zeta|^2}{\omega^2} \right) \mp \left( \frac{|\zeta|}{\omega} \pm \frac{|\zeta|}{\omega} \right) \right\} \frac{i \zeta \otimes i \zeta}{\omega^2} & \frac{i \zeta}{\omega} \frac{|\zeta|}{\omega} \pm \frac{i \zeta}{\omega} \frac{|\zeta|}{\omega} \\[0.5em] - \frac{i \zeta^{\sf{T}}}{\omega} & 1 \pm \frac{|\zeta|^2}{\omega^2} \end{array}\right]
	\end{array}
\end{equation*}
\narrowarray
with
\begin{equation*}
	\delta = \frac{1}{2} \left( 1 + \frac{|\zeta|^2}{\omega^2} \right) \left( 1 \pm \frac{|\zeta|^2}{\omega^2} \right) - \left( \frac{{|\zeta|}^3}{\omega^3} \pm \frac{{|\zeta|}^3}{\omega^3} \right)
\end{equation*}
and, hence,
\begin{equation*}
	\begin{array}{rcl}
		\hat{z}_w
		& = & {\left\{ \frac{1}{2} \left( 1 + \frac{|\zeta|^2}{\omega^2} \right) \left( 1 \pm \frac{|\zeta|^2}{\omega^2} \right) - \left( \frac{{|\zeta|}^3}{\omega^3} \pm \frac{{|\zeta|}^3}{\omega^3} \right) \right\}}^{-1} \left( 1 \pm \frac{|\zeta|^2}{\omega^2} \right) {\displaystyle{\frac{1}{2 \sqrt{\mu}}}} \omega^{-2} \hat{h}_w \\[1.5em]
		& = & \left\{ \begin{array}{rl} \frac{1}{\sqrt{\mu}} \frac{\omega^2 + |\zeta|^2}{{\left( \omega^2 + |\zeta|^2 \right)}^2 - 4 \omega {|\zeta|}^3} \hat{h}_w, & \quad \mbox{if} \ \alpha = +1, \\[1em] \frac{1}{\sqrt{\mu}} \frac{1}{\omega^2 + |\zeta|^2} \hat{h}_w, & \quad \mbox{if} \ \alpha = -1. \end{array} \right.
	\end{array}
\end{equation*}
This implies
\begin{equation*}
	\begin{array}{rcl}
		[p] & = & \kappa \lambda_\epsilon \hat{z}_w = \frac{\kappa}{\sqrt{\mu}} \lambda_\epsilon \frac{\omega^2 + |\zeta|^2}{{\left( \omega^2 + |\zeta|^2 \right)}^2 - 4 \omega {|\zeta|}^3} \hat{h}_w = \rho \lambda_\epsilon \frac{\omega^2 + |\zeta|^2}{{\left( (\omega^2 - |\zeta|^2) + 2 |\zeta|^2 \right)}^2 - 4 \omega {|\zeta|}^3} \hat{h}_w \\[1em]
		    & = & (\omega^2 - |\zeta|^2) \frac{\omega^2 + |\zeta|^2}{{(\omega^2 - |\zeta|^2)}^2 + 4 (\omega^2 - |\zeta|^2) |\zeta|^2 + 4 {|\zeta|}^4 - 4 \omega {|\zeta|}^3} \hat{h}_w\\[1em]
		    & = & \frac{\omega^2 + |\zeta|^2}{(\omega^2 - |\zeta|^2) + 4 \frac{\omega}{\omega + |\zeta|} |\zeta|^2} \hat{h}_w\\[1em]
	\end{array}
\end{equation*}
for $\alpha = +1$ and
\begin{equation*}
	[p] = \kappa \lambda_\epsilon \hat{z}_w = {\textstyle{\frac{\kappa}{\sqrt{\mu}}}} \lambda_\epsilon \frac{1}{\omega^2 + |\zeta|^2} \hat{h}_w = \frac{\omega^2 - |\zeta|^2}{\omega^2 + |\zeta|^2} \hat{h}_w
\end{equation*}
for $\alpha = -1$.
Therefore, the desired solution may be obtained by solving the splitting scheme
\begin{equation*}
	\begin{array}{c}
		\rho \epsilon u + \rho \partial_t u - \mu \Delta u = - \nabla p, \quad - \Delta p = 0 \quad \mbox{in} \ (0,\,\infty) \times \bR^n_+ \\[0.5em]
		P_\Gamma \cB^{\pm 1}(u) = 0, \quad [\mbox{div}\,u] = 0, \quad [p] = S^{\pm 1} h_w \quad \mbox{on} \ (0,\,\infty) \times \partial \bR^n_+ \\[0.5em]
		u(0) = 0 \quad \mbox{in} \ \bR^n_+,
	\end{array}
\end{equation*}
where the bounded linear operators
\begin{equation*}
	S^{\pm 1}: \bN^1_h(\infty) \longrightarrow L_p((0,\,\infty),\,\dot{W}^{1 - 1/p}_p(\bR^{n - 1}))
\end{equation*}
are defined via their Laplace-Fourier symbols
\begin{equation*}
	m_{\pm 1}(\lambda,\,\xi) = \frac{\omega^2 \pm |\zeta|^2}{(\omega^2 \mp |\zeta|^2) + 2 \frac{\omega}{\omega + |\zeta|} (|\zeta|^2 \pm |\zeta|^2)},
\end{equation*}
i.\,e. $S^{+1} = (\rho \epsilon + \rho \partial_t - 2 \mu \Delta_\Gamma){(\rho \epsilon + \rho \partial_t - 4 \mu S \Delta_\Gamma)}^{-1}$ with
\begin{equation*}
	S := {(\rho \epsilon + \rho \partial_t - \mu \Delta_\Gamma)}^{1/2} {\left\{ {(\rho \epsilon + \rho \partial_t - \mu \Delta_\Gamma)}^{1/2} + {(- \mu \Delta_\Gamma)}^{1/2} \right\}}^{-1}
\end{equation*}
and $S^{-1} = (\rho \epsilon + \rho \partial_t) {(\rho \epsilon + \rho \partial_t - 2 \mu \Delta_\Gamma)}^{-1}$,
e.\,g. by the $\cH^\infty$-calculi of the involved operators.
Indeed, the calculations in Appendix~B imply the solution of the above splitting scheme to solve the shifted Stokes equations, cf.\ relation \eqnref{trace:1-1}.
On the other hand, any solution to the shifted Stokes equations with $h_w = 0$ satisfies the equations of the above splitting scheme with a boundary condition
\begin{equation*}
	[p] = h_p \quad \mbox{on} \ (0,\,\infty) \times \partial \bR^n_+
\end{equation*}
for some $h_p \in L_p((0,\,\infty),\,\dot{W}^{1 - 1/p}_p(\bR^{n - 1}))$.
Employing relation \eqnref{trace:1-1} once again, we infer $h_p = S^{\pm 1} (- 2 \mu [\partial_y w] + [p]) = S^{\pm 1} h_w = 0$.
Therefore, $p = 0$ and $u = 0$.

\section{The Bent Halfspace Case}\label{sec:bent-halfspace}
In this section, we prove \Thmref{s} in the bent halfspace case,
i.\,e.\ we assume $a > 0$, $\Omega = \bR^n_\omega$ with a sufficiently flat function $\omega \in BUC^{3-}(\bR^{n - 1})$ and $1 < p < \infty$ with $p \neq \frac{3}{2},\,3$.
Following \Remref{s} it is sufficient to prove existence of a unique maximal regular solution for all data satisfying the stated regularity and compatibility conditions.
However, following \Remref{splitting} we may restrict the proof to the case $\alpha \in \{\,-1,\,0,\,+1\,\}$, $\beta \in \{\,0,\,+1\,\}$, $f = 0$, $g = 0$, $\cP_\Gamma h = 0$ and $u_0 = 0$,
i.\,e. we have to construct a unique maximal regular solution to the Stokes equations
\begin{equation*}
	\label{eqn:s-bent}
	\begin{array}{c}
		\rho\partial_t u - \mu \Delta u + \nabla p = 0,
		\quad \quad \mbox{div}\,u = 0
			\quad \quad \mbox{in}\ (0,\,a) \times \bR^n_\omega, \\[0.5em]
		P_\Gamma \cB^\alpha(u) = 0,
		\quad \quad Q_\Gamma \cB^\beta(u,\,p) = h_w,
			\quad \quad \mbox{on}\ (0,\,a) \times \partial \bR^n_\omega, \\[0.5em]
		u(0) = 0
			\quad \quad \mbox{in}\ \bR^n_\omega,
	\end{array}
\end{equation*}
where $h_v := \cP_\Gamma h = 0$ and $h_w := Q_\Gamma h \in {}_0 H^1_p((0,\,a),\,\dot{W}^{-1/p}_p(\Gamma,\,N\Gamma)) \cap \bN^0_{h, -\infty}(a)$,
if $\beta = 0$ resp. \ $h_w \in \bN^{+1}_{h, 1/2 - 1/2p}(a)$, if $\beta = +1$.

To solve the above bent halfspace problem, we employ the pull-backs
$\bar{u} := u \circ \Theta_\omega$ and $\bar{p} := p \circ \Theta_\omega$, where
\begin{equation*}
	\Theta_\omega(x,\,y) := (x,\,y + \omega(x)), \quad x \in \bR^{n - 1},\ y > 0,
\end{equation*}
to reduce the problem to the halfspace case.
Since $u = \bar{u} \circ \Theta^{-1}_\omega$
we have
\begin{equation*}
	\begin{array}{rcl}
		\partial_t   u & = & (\partial_t   \bar{u}) \circ \Theta^{-1}_\omega \\[0.5em]
		\partial_k   u & = & (\partial_k   \bar{u}) \circ \Theta^{-1}_\omega - (\partial_k \omega)\left\{\,(\partial_y \bar{u}) \circ \Theta^{-1}_\omega\,\right\} \\[0.5em]
		\partial^2_k u & = & (\partial^2_k \bar{u}) \circ \Theta^{-1}_\omega - 2 (\partial_k \omega)\left\{\,(\partial_k \partial_y \bar{u}) \circ \Theta^{-1}_\omega\,\right\} \\[0.5em]
		               &   & + \ {(\partial_k \omega)}^2\left\{(\partial^2_y \bar{u}) \circ \Theta^{-1}_\omega\,\right\} - (\partial^2_k \omega) \left\{\,(\partial_y \bar{u}) \circ \Theta^{-1}_\omega\,\right\} \\[0.5em]
		\partial_y   u & = & (\partial_y   \bar{u}) \circ \Theta^{-1}_\omega \\[0.5em]
		\partial^2_y u & = & (\partial^2_y \bar{u}) \circ \Theta^{-1}_\omega
	\end{array}
\end{equation*}
for $k = 1,\,2,\,\dots,\,n - 1$ and, hence,
\begin{equation*}
	\begin{array}{rcl}
		\mbox{div}\,u & = & (\mbox{div}\,\bar{u}) \circ \Theta^{-1}_\omega - (\,\nabla \omega\,|\,(\partial_y \bar{v}) \circ \Theta^{-1}_\omega\,) \\[0.5em]
		\Delta u      & = & (\Delta \bar{u}) \circ \Theta^{-1}_\omega - 2 (\,\nabla \omega\,|\,(\nabla_x \partial_y \bar{u}) \circ \Theta^{-1}_\omega\,) \\[0.5em]
		              &   & + \ {|\nabla \omega|}^2 \left\{\,(\partial^2_y \bar{u}) \circ \Theta^{-1}_\omega\,\right\} - (\Delta \omega) \left\{\,(\partial_y \bar{u}) \circ \Theta^{-1}_\omega\,\right\}
	\end{array}
\end{equation*}
in $(0,\,a) \times \bR^n_\omega$.
Here we again decomposed the solution as $u = (v,\,w)$ and the pull-back as $\bar{u} = (\bar{v},\,\bar{w})$ into the first $n - 1$ components and the remaining component.
Note, that $\bar{u}$ will be constructed as a solution to the halfspace problem.
Hence $\bar{v}$ constitutes its tangential part and $\bar{w}$ constitutes its normal part.

Using this transformation, we first derive the system of interior partial differential equations
\begin{equation*}
	\begin{array}{rcll}
		\rho \partial_t \bar{v} - \mu \Delta \bar{v} + \nabla_x   \bar{p} & = & \rho F_v(\bar{v},\,\bar{p}) & \quad \mbox{in} \ (0,\,a) \times \bR^n_\omega, \\[0.5em]
		\rho \partial_t \bar{w} - \mu \Delta \bar{w} + \partial_y \bar{p} & = & \rho F_v(\bar{w})           & \quad \mbox{in} \ (0,\,a) \times \bR^n_\omega, \\[0.5em]
		                                              \mbox{div}\,\bar{u} & = & G(\bar{v})                  & \quad \mbox{in} \ (0,\,a) \times \bR^n_\omega,
	\end{array}
\end{equation*}
where
\begin{equation*}
	\begin{array}{rcl}
		\rho F_v(\bar{v},\,\bar{p}) & := & - \ 2 \mu (\,\nabla \omega\,|\,\nabla_x\,) \partial_y \bar{v} + \mu {|\nabla \omega|}^2\,\partial^2_y \bar{v} - \mu (\Delta \omega)\,\partial_y \bar{v} + (\nabla \omega)\,\partial_y \bar{p}, \\[0.5em]
		\rho F_w(\bar{w})           & := & - \ 2 \mu (\,\nabla \omega\,|\,\nabla_x\,) \partial_y \bar{w} + \mu {|\nabla \omega|}^2\,\partial^2_y \bar{w} - \mu (\Delta \omega)\,\partial_y \bar{w}, \\[0.5em]
		G(\bar{v})                  & := & (\,\nabla \omega\,|\,\partial_y \bar{v}\,).
	\end{array}
\end{equation*}
Now, if we assume ${\|\nabla \omega\|}_{L_\infty(\bR^{n - 1})} < \delta$ with $\delta \in (0,\,1)$, we may estimate
\begin{equation*}
	\begin{array}{rcl}
		{\|(\,\nabla \omega\,|\,\nabla_x\,) \partial_y \bar{v}\|}_{\bY_f(a)} & \leq & \delta {\|\bar{v}\|}_{\bX_u(a)}, \\[0.5em]
		          {\|{|\nabla \omega|}^2\,\partial^2_y \bar{v}\|}_{\bY_f(a)} & \leq & \delta {\|\bar{v}\|}_{\bX_u(a)}, \\[0.5em]
		                {\|(\nabla \omega)\,\partial_y \bar{p}\|}_{\bY_f(a)} & \leq & \delta {\|\bar{p}\|}_{\bX^{|\beta|}_{p, \gamma}(a)}
	\end{array}
\end{equation*}
and analogous estimates are valid for the terms appearing in the definition of $F_w$.
Moreover, the embedding chain
\begin{equation*}
	\begin{array}{l}
		{}_0 H^{1/2}_p((0,\,a),\,L_p(\bR^n_+,\,\bR^{n - 1})) \cap L_p((0,\,a),\,H^1_p(\bR^n_+,\,\bR^{n - 1})) \\[0.5em]
		\qquad \hookrightarrow {}_0 BUC((0,\,a),\,L_p(\bR^n_+,\,\bR^{n - 1})) \\[0.5em]
		\qquad \hookrightarrow \bY_f(a)
	\end{array}
\end{equation*}
is valid and the corresponding estimates read
\begin{equation*}
	\begin{array}{rcl}
		{\|\partial_y \bar{v}\|}_{\bY_f(a)}
			& \leq & a^{1/p} {\|\partial_y \bar{v}\|}_{{}_0 BUC((0, a), L_p(\bR^n_+, \bR^{n - 1}))} \\[0.5em]
			& \leq & c a^{1/p} {\|\partial_y \bar{v}\|}_{{}_0 H^{1/2}_p((0, a), L_p(\bR^n_+, \bR^{n - 1})) \cap L_p((0, a), H^1_p(\bR^n_+, \bR^{n - 1}))},
	\end{array}
\end{equation*}
where the constant $c > 0$ is independent of $a \in (0,\,1]$ thanks to the homogeneous initial condition.
Hence,
\begin{equation*}
	{\|(\Delta \omega)\,\partial_y \bar{v}\|}_{\bY_f(a)} \leq c a^{1/p} {\|\Delta \omega\|}_{L_\infty(\bR^{n - 1})} {\|\bar{v}\|}_{\bX_u(a)}
\end{equation*}
and an analogous estimate is valid for the term appearing in the definition of $F_w$.
Summing up the above estimates, we derive
\begin{subequations}
\label{eqn:bh-estimates}
\begin{equation}
	\begin{array}{rcl}
		{\|F_v(\bar{v},\,\bar{p})\|}_{\bY_f(a)} & \leq & \mu (3 \delta + c a^{1/p} {\|\Delta \omega\|}_{L_\infty(\bR^{n - 1})}) {\|\bar{v}\|}_{\bX_u(a)} + \delta {\|\bar{p}\|}_{\bX^\beta_{p, \gamma}(a)}, \\[0.5em]
		          {\|F_w(\bar{w})\|}_{\bY_f(a)} & \leq & \mu (3 \delta + c a^{1/p} {\|\Delta \omega\|}_{L_\infty(\bR^{n - 1})}) {\|\bar{w}\|}_{\bX_u(a)}.
	\end{array}
\end{equation}
Furthermore, we have
\begin{equation*}
	{\|(\,\nabla \omega\,|\,\partial_y \bar{v}\,)\|}_{{}_0 H^{1/2}_p((0, a), L_p(\bR^n_+))} \leq \delta {\|\partial_y \bar{v}\|}_{{}_0 H^{1/2}_p((0, a), L_p(\bR^n_+))}
\end{equation*}
and
\begin{equation*}
	\begin{array}{rcl}
	{\|(\,\nabla \omega\,|\,\partial_y \bar{v}\,)\|}^p_{L_p((0, a), H^1_p(\bR^n_+))}
		& = & {\|(\,\nabla \omega\,|\,\partial_y \bar{v}\,)\|}^p_{L_p((0, a), L_p(\bR^n_+))} \\[0.5em]
		&   & \ + \ {\displaystyle{\sum^{n - 1}_{k = 1}}} {\|(\,\partial_k \nabla \omega\,|\,\partial_y \bar{v}\,)\|}^p_{L_p((0, a), L_p(\bR^n_+))} \\[1.5em]
		&   & \ + \ {\displaystyle{\sum^{n - 1}_{k = 1}}} {\|(\,\nabla \omega\,|\,\partial_k \partial_y \bar{v}\,)\|}^p_{L_p((0, a), L_p(\bR^n_+))} \\[1.5em]
		&   & \ + \ {\|(\,\nabla \omega\,|\,\partial^2_y \bar{v}\,)\|}^p_{L_p((0, a), L_p(\bR^n_+))}
	\end{array}
\end{equation*}
with the same arguments as above implies
\begin{equation*}
	{\|(\,\nabla \omega\,|\,\partial_y \bar{v}\,)\|}_{L_p((0, a), H^1_p(\bR^n_+))} \leq 2^{1/p} (\delta + c a^{1/p} {\|\nabla^2 \omega\|}_{L_\infty(\bR^{n - 1})}) {\|\bar{v}\|}_{\bX_u(a)},
\end{equation*}
where the constant $c > 0$ is again independent of $a \in (0,\,1]$.
Hence, the estimate
\begin{equation}
	{\|G(\bar{v})\|}_{\bY_g(a)} \leq 2^{1/p} (\delta + c a^{1/p} {\|\nabla^2 \omega\|}_{L_\infty(\bR^{n - 1})}) {\|\bar{v}\|}_{\bX_u(a)}
\end{equation}
is valid.

Now, the outer unit normal of $\bR^n_\omega$ is given as
\begin{equation*}
	\nu_\Gamma(x,\,\omega(x)) = \frac{1}{\sqrt{1 + {\|\nabla \omega\|}^2}} \left[\begin{array}{c} \nabla \omega(x) \\[0.5em] -1 \end{array}\right], \quad x \in \bR^{n - 1}
\end{equation*}
and we infer
\begin{equation*}
	 \kappa_\omega \left\{\,\nu_\Gamma \circ \Theta_\omega\,\right\} = \nu_\Sigma + \left[\begin{array}{c} \nabla \omega \\[0.5em] 0 \end{array}\right] =: \nu_\Sigma + N(\nabla \omega),
\end{equation*}
where we have set $\kappa_\omega := \sqrt{1 + {\|\nabla \omega\|}^2}$ and $\nu_\Sigma$ denotes the outer unit normal of $\bR^n_+$ on $\Sigma := \partial \bR^n_+$.
Moreover,
\begin{equation*}
	P_\Gamma(x,\,\omega(x)) = 1 - \frac{1}{1 + {\|\nabla \omega\|}^2} \left[\begin{array}{cc} \nabla \omega(x) \otimes \nabla \omega(x) & - \nabla \omega(x) \\[0.5em] -\nabla \omega(x)^{\sf{T}} & 1 \end{array}\right], \quad x \in \bR^{n - 1}
\end{equation*}
and we infer
\begin{equation*}
	P_\Gamma \circ \Theta_\omega = P_\Sigma - \frac{1}{1 + {|\nabla \omega|}^2} \left[\begin{array}{cc} \nabla \omega \otimes \nabla \omega & - \nabla \omega \\[0.5em] - \nabla \omega & - {|\nabla \omega|}^2 \end{array}\right] =: P_\Sigma - L(\nabla \omega)
\end{equation*}

If $\alpha = 0$, the tangential boundary condition $P_\Gamma [u] = 0$ on $(0,\,a) \times \partial \bR^n_\omega$ is equivalent to
\begin{equation*}
	P_\Sigma [\bar{u}] - L(\nabla \omega) [\bar{u}] = P_\Gamma [u] \circ \Theta_\omega = 0
\end{equation*}
on $(0,\,a) \times \partial \bR^n_+$, which may be rewritten as
\begin{equation*}
	[\bar{v}] = L_v(\nabla \omega) [\bar{u}], \quad \quad L_w(\nabla \omega) [\bar{u}] = 0
\end{equation*}
on $(0,\,a) \times \partial \bR^n_+$, where $L_v$ denotes the first $n - 1$ rows of $L$ and $L_w$ denotes the last row of $L$.
However, the first equation already implies the second one, which is not surprising, since a tangential boundary condition on $\Gamma = \partial \Omega$
may not lead to $n$ linearly independent transformed boundary conditions on $\Sigma = \partial \bR^n_+$.
The original boundary condition in the bent halfspace is therefore equivalent to
\begin{equation*}
		[\bar{v}] = H^0_v(\bar{v},\,\bar{w}) \quad \mbox{on} \ (0,\,a) \times \partial \bR^n_\omega
\end{equation*}
where
\begin{equation*}
	H^0_v(\bar{v},\,\bar{w}) := \frac{\nabla \omega}{1 + {\|\nabla \omega\|}^2} \left\{(\,\nabla \omega\,|\,[\bar{v}]\,) - [\bar{w}]\right\}.
\end{equation*}
Observe, that all terms above either carry a factor $\nabla \omega$ or are of lower order.
With the same arguments as above we therefore infer
\begin{equation}
	{\|H^0_v(\bar{v},\,\bar{w})\|}_{\bT^0_h(a)} \leq c (\delta + a^{1/p} {\|\nabla^2 \omega\|}_{BUC^{1-}(\bR^{n - 1})}) {\|\bar{u}\|}_{\bX_u(a)}
\end{equation}
with some constant $c > 0$ independent of $a \in (0,\,1]$.

If $\alpha = \pm 1$, we have to treat the terms
\begin{equation*}
	\begin{array}{rll}
		[\nabla u \pm \nabla u^{\sf{T}}] \circ \Theta_\omega
			& =  & [\nabla \bar{u} \pm \nabla \bar{u}^{\sf{T}}] - \left[\begin{array}{cc} \nabla \omega \otimes [\partial_y \bar{v}] \pm [\partial_y \bar{v}] \otimes \nabla \omega & (\nabla \omega)\,[\partial_y \bar{w}] \\[0.5em] \pm (\,\nabla \omega\,|\,[\partial_y \bar{v}]\,) & 0 \end{array}\right] \\[2em]
			& =: & [\nabla \bar{u} \pm \nabla \bar{u}^{\sf{T}}] - M(\nabla \omega,\,[\partial_y \bar{v}],\,[\partial_y \bar{w}])
	\end{array}
\end{equation*}
and the tangential boundary condition $\mu P_\Gamma [\nabla u \pm \nabla u^{\sf{T}}] \nu_\Gamma = 0$ on $(0,\,a) \times \partial \bR^n_\omega$ is equivalent to
\begin{equation*}
	\begin{array}{rcl}
		\mu P_\Sigma [\nabla \bar{u} \pm \nabla \bar{u}^{\sf{T}}] \nu_\Sigma
			& - & \mu L(\nabla \omega) \left\{\,[\nabla u \pm \nabla u^{\sf{T}}] \circ \Theta_\omega\,\right\} \nu_\Sigma \\[0.5em]
			& - & \mu P_\Sigma M(\nabla \omega,\,[\partial_y \bar{v}],\,[\partial_y \bar{w}]) \nu_\Sigma \\[0.5em]
			& + & \mu \left\{\,P_\Gamma [\nabla u \pm \nabla u^{\sf{T}}] \circ \Theta_\omega \right\} N(\nabla \omega) \\[0.5em]
			& = & \kappa_\omega \mu \left\{\,P_\Gamma [\nabla u \pm \nabla u^{\sf{T}}] \nu_\Gamma \circ \Theta_\omega \right\} \\[0.5em]
			& = & 0,
	\end{array}
\end{equation*}
on $(0,\,a) \times \partial \bR^n_+$, which may be rewritten as
\begin{equation*}
	\begin{array}{rcl}
		\mp \mu [\partial_y v] - \mu \nabla _x [w]
			& = & \mu L_v(\nabla \omega) \left\{\,[\nabla u \pm \nabla u^{\sf{T}}] \circ \Theta_\omega\,\right\} \nu_\Sigma \\[0.5em]
			& + & \mu M_v(\nabla \omega,\,[\partial_y \bar{v}],\,[\partial_y \bar{w}]) \nu_\Sigma \\[0.5em]
			& - & \mu \left\{\,P^v_\Gamma [\nabla u \pm \nabla u^{\sf{T}}] \circ \Theta_\omega \right\} N(\nabla \omega)
	\end{array}
\end{equation*}
and
\begin{equation*}
	\mu L_w(\nabla \omega) \left\{\,[\nabla u \pm \nabla u^{\sf{T}}] \circ \Theta_\omega\,\right\} \nu_\Sigma = \mu \left\{\,P^w_\Gamma [\nabla u \pm \nabla u^{\sf{T}}] \circ \Theta_\omega \right\} N(\nabla \omega)
\end{equation*}
on $(0,\,a) \times \partial \bR^n_+$, where $M_v$ denotes the first $n - 1$ rows of $P_\Sigma M$, $P^v_\Gamma$ denotes the first $n - 1$ rows of $P_\Gamma$ and $P^w_\Gamma$ denotes the last row of $P_\Gamma$.
Again, the first equation already implies the second one and the original boundary condition in the bent halfspace is therefore equivalent to
\begin{equation*}
	\mp \mu [\partial_y v] - \mu \nabla _x [w] = H^{\pm 1}_v(\bar{v},\,\bar{w}) \quad \mbox{on} \ (0,\,a) \times \partial \bR^n_\omega
\end{equation*}
where
\begin{equation*}
	\begin{array}{rcl}
		H^{\pm 1}_v(\bar{v},\,\bar{w})
			& := & \mu L_v(\nabla \omega) \left\{\,[\nabla u \pm \nabla u^{\sf{T}}] \circ \Theta_\omega\,\right\} \nu_\Sigma + \mu M_v(\nabla \omega,\,[\partial_y \bar{v}],\,[\partial_y \bar{w}]) \nu_\Sigma \\[0.5em]
			&    & \qquad - \ \mu \left\{\,P^v_\Gamma [\nabla u \pm \nabla u^{\sf{T}}] \circ \Theta_\omega \right\} N(\nabla \omega).
	\end{array}
\end{equation*}
Observe, that all terms above again either carry a factor $\nabla \omega$ or are of lower order.
With the same arguments as above we therefore infer
\begin{equation}
	{\|H^{\pm 1}_v(\bar{v},\,\bar{w})\|}_{\bT^1_h(a)} \leq c (\delta + a^{1/p} {\|\nabla^2 \omega\|}_{BUC^{1-}(\bR^{n - 1})}) {\|\bar{u}\|}_{\bX_u(a)}
\end{equation}
with some constant $c > 0$ independent of $a \in (0,\,1]$.

Concerning the normal boundary condition, we start with the case $\beta = 0$, i.\,e. $[u] \cdot \nu_\Gamma = h \cdot \nu_\Gamma$ on $(0,\,a) \times \partial \bR^n_\omega$.
This condition is equivalent to
\begin{equation*}
	-[\bar{w}] + \nabla \omega \cdot [\bar{v}] = [\bar{u}] \cdot \nu_\Sigma + [\bar{u}] \cdot N(\nabla \omega) = \kappa_\omega \left\{\,([u] \cdot \nu_\Gamma) \circ \Theta_\omega\,\right\} = \kappa_\omega \bar{h}_w
\end{equation*}
on $(0,\,a) \times \partial \bR^n_+$ with $\bar{h}_w := (h \cdot \nu_\Gamma) \circ \Theta_\omega$.
Therefore, the original boundary condition in the bent halfspace is equivalent to
\begin{equation*}
	[\bar{w}] = \nabla \omega \cdot [\bar{v}] - \kappa_\omega \bar{h}_w =: H^0_w(\bar{v}) - \kappa_\omega \bar{h}_w
\end{equation*}
on $(0,\,a) \times \partial \bR^n_+$.
As above, all appearing terms either carry a factor $\nabla \omega$ or are of lower order and we infer
\begin{equation}
	{\|H^0_w(\bar{v})\|}_{\bN^0_h(a)} \leq c (\delta + a^{1/p} {\|\nabla^2 \omega\|}_{BUC^{1-}(\bR^{n - 1})}) {\|\bar{u}\|}_{\bX_u(a)}
\end{equation}
with some constant $c > 0$ independent of $a \in (0,\,1]$.

Finally, if $\beta = +1$, the normal boundary condition $2 \mu\,\partial_\nu u \cdot \nu_\Gamma - [p] = h \cdot \nu_\Gamma$ on $(0,\,a) \times \partial \bR^n_\omega$ is equivalent to
\begin{equation*}
	\begin{array}{rcl}
		2 \mu\,\partial_\nu \bar{u} \cdot \nu_\Sigma - [\bar{p}]
			& - & 2 \mu\,Q(\nabla \omega,\,[\partial_y \bar{v}],\,[\partial_y \bar{w}]) \nu_\Sigma \cdot \nu_\Sigma \\[0.5em]
			& + & 2 \mu\,\left\{[\nabla u] \circ \Theta_\omega\right\} N(\nabla \omega) \cdot \nu_\Sigma \\[0.5em]
			& + & 2 \mu\,\left\{[\nabla u] \nu_\Gamma \circ \Theta_\omega\right\} \cdot \kappa_\omega N(\nabla \omega) \\[0.5em]
			& - & {|\nabla \omega|}^2 [\bar{p}] \\[0.5em]
			& = & \kappa^2_\omega \left\{(2 \mu\,\partial_\nu u \cdot \nu_\Gamma) \circ \Theta_\omega - [p] \circ \Theta_\omega \right\} \\[0.5em]
			& = & \kappa^2_\omega \bar{h}_w
	\end{array}
\end{equation*}
on $(0,\,a) \times \partial \bR^n_+$ with $\bar{h}_w := h \cdot \nu_\Gamma \circ \Theta_\omega$ and
\begin{equation*}
	[\nabla u] \circ \Theta_\omega
		=  [\nabla \bar{u}] - \left[\begin{array}{cc} \nabla \omega \otimes [\partial_y \bar{v}] & (\nabla \omega)\,[\partial_y \bar{w}] \\[0.5em] 0 & 0 \end{array}\right]
		=: [\nabla \bar{u}] - Q(\nabla \omega,\,[\partial_y \bar{v}],\,[\partial_y \bar{w}]).
\end{equation*}
Therefore, the original boundary condition in the bent halfspace is equivalent to
\begin{equation*}
	\begin{array}{rcl}
		2 \mu\,\partial_\nu \bar{u} \cdot \nu_\Sigma - [\bar{p}]
			& =  & 2 \mu\,Q(\nabla \omega,\,[\partial_y \bar{v}],\,[\partial_y \bar{w}]) \nu_\Sigma \cdot \nu_\Sigma \\[0.5em]
			& -  & 2 \mu\,\left\{[\nabla u] \circ \Theta_\omega\right\} N(\nabla \omega) \cdot \nu_\Sigma \\[0.5em]
			& -  & 2 \mu\,\left\{[\nabla u] \nu_\Gamma \circ \Theta_\omega\right\} \cdot \kappa_\omega N(\nabla \omega) \\[0.5em]
			& +  & {|\nabla \omega|}^2 [\bar{p}] + \kappa^2_\omega \bar{h}_w \\[0.5em]
			& =: & H^{+1}_w(\bar{v},\,\bar{w},\,\bar{p}) + \kappa^2_\omega \bar{h}_w
	\end{array}
\end{equation*}
on $(0,\,a) \times \partial \bR^n_+$.
As above, all appearing terms either carry a factor $\nabla \omega$ or are of lower order and we infer
\begin{equation}
	{\|H^{+1}_w(\bar{v},\,\bar{w},\,\bar{p})\|}_{\bN^1_h(a)} \leq c (\delta + a^{1/p} {\|\nabla^2 \omega\|}_{BUC^{1-}(\bR^{n - 1})}) ({\|\bar{u}\|}_{\bX_u(a)} + {\|\bar{p}\|}_{\bX^\beta_{p, \gamma}(a)})
\end{equation}
with some constant $c > 0$ independent of $a \in (0,\,1]$.
\end{subequations}

Summarizing the above considerations, the desired maximal regularity property of the bent halfspace problem is equivalent to the existence of a unique solution
$(\bar{v},\,\bar{w}) \in {}_0 H^1_p((0,\,a),\,L_p(\bR^n_+,\,\bR^n)) \cap L_p((0,\,a),\,H^2_p(\bR^n_+,\,\bR^n))$ together with an appropriate pressure $\bar{p}$,
whose regularity depends on $\beta$ and $\gamma$, to the fixed point problem
\begin{equation*}
	L_a(\bar{v},\,\bar{w},\,\bar{p}) = R_a(\bar{v},\,\bar{w},\,\bar{p}),
\end{equation*}
where $L_a$ resp. $R_a$ denote the left resp. right hand side of the transformed system, i.\,e. $L_a$ is the isomorphism arising from the halfspace problem.
Now, the estimates \eqnref{bh-estimates} imply, that we may first choose $\delta \in (0,\,1)$ and then, given $\omega \in BUC^{3-}(\bR^{n - 1})$ with ${|\nabla \omega|}_{L_\infty(\bR^{n - 1})} < \delta$,
we may choose $a \in (0,\,1]$, such that the fixed point problem admits exactly one solution $(\bar{v},\,\bar{w},\,\bar{p})$ in the maximal regularity class,
thanks to the contraction mapping principle.
Hence, the original problem in the bent halfspace has the desired maximal regularity property at least on small time intervals.
Observe, that $a$ may depend on $\omega$, but not on the data.
Therefore, we may construct a unique solution on any arbitrary time interval by successively solving the system on sufficiently small time intervals.
Finally, the bent halfspace problem has the desired maximal regularity property for all $a > 0$ and all $\omega \in BUC^{3-}(\bR^{n - 1})$ with ${|\nabla \omega|}_{L_\infty(\bR^{n - 1})} < \delta$.

\section{The Case of a Bounded Domain}\label{sec:domain}
In this section, we prove \Thmref{s} in the bounded domain case,
i.\,e.\ we assume $a > 0$, $\Omega \subseteq \bR^n$ to be a bounded domain with boundary $\Gamma = \partial \Omega$ of class $C^{3-}$ and $1 < p < \infty$, $p \neq \frac{3}{2},\,3$.
Following \Remref{s} it is sufficient to prove existence of a unique maximal regular solution for all data satisfying the stated regularity and compatibility conditions.
However, following \Remref{splitting} we may restrict the proof to the case $\alpha \in \{\,-1,\,0,\,1\,\}$, $\beta \in \{\,0,\,1\,\}$, $f = 0$, $g = 0$, $P_\Gamma h = 0$ and $u_0 = 0$,
i.\,e. we have to construct a unique maximal regular solution to the Stokes equations
\begin{equation}
	\label{eqn:s-dom-0}
	\begin{array}{c}
		\rho\partial_t u - \mu \Delta u + \nabla p = 0,
		\quad \quad \mbox{div}\,u = 0
			\quad \quad \mbox{in}\ (0,\,a) \times \Omega, \\[0.5em]
		P_\Gamma \cB^\alpha(u) = 0,
		\quad \quad Q_\Gamma \cB^\beta(u,\,p) = h_w,
			\quad \quad \mbox{on}\ (0,\,a) \times \partial \Omega, \\[0.5em]
		u(0) = 0
			\quad \quad \mbox{in}\ \Omega,
	\end{array}
\end{equation}
where $h_v := P_\Gamma h = 0$ and $h_w := Q_\Gamma h \in \bN^\beta_{h, \gamma}$
with $h_w \in {}_0 H^1_p((0,\,a),\,\dot{W}^{-1/p}_p(\Gamma,\,N\Gamma))$ and $\gamma = -\infty$, if $\beta = 0$ resp. \ $\gamma = 1/2 - 1/2p$, if $\beta = 1$.

If $\beta = 0$, we first choose an offset
\begin{equation*}
	\nabla \eta \in {}_0 H^1_p((0,\,a),\,L_p(\Omega,\,\bR^n)) \cap L_p((0,\,a),\,H^2_p(\Omega,\,\bR^n))
\end{equation*}
as a solution to the elliptic problem
\begin{equation*}
		- \Delta \eta = 0 \quad \mbox{in} \ (0,\,a) \times \Omega, \qquad 
		\partial_\nu \eta = h_w \cdot \nu \quad \mbox{on} \ (0,\,a) \times \partial \Omega,
\end{equation*}
we set $\bar{h}_v := - P_\Gamma \cB^\alpha(\nabla \eta)$,
and then construct a unique maximal regular solution $(\bar{u},\,\bar{p}) := (u - \nabla \eta,\,p + \rho \partial_t \eta - \mu \Delta \eta)$ to the Stokes equations
\begin{equation}
	\label{eqn:s-dom-1}
	\begin{array}{c}
		\rho\partial_t u - \mu \Delta u + \nabla p = 0,
		\quad \quad \mbox{div}\,u = 0
			\quad \quad \mbox{in}\ (0,\,a) \times \Omega, \\[0.5em]
		P_\Gamma \cB^\alpha(u) = h_v,
		\quad \quad Q_\Gamma \cB^\beta(u,\,p) = 0,
			\quad \quad \mbox{on}\ (0,\,a) \times \partial \Omega, \\[0.5em]
		u(0) = 0
			\quad \quad \mbox{in}\ \Omega,
	\end{array}
\end{equation}
where we dropped the bars again.

Now, if we consider \eqnref{s-dom-1} in case $\beta = 0$ and \eqnref{s-dom-0} in case $\beta = 1$, any maximal regular solution, where we may assume
\begin{equation}
	\label{eqn:pressure-mean}
	{(\,p\,)}_\Omega = \frac{1}{|\Omega|} \int_\Omega p\,\mbox{d}x = 0,
\end{equation}
if $\beta = 0$, enjoys the additional time regularity
\begin{equation*}
	p \in {}_0 H^\sigma_p((0,\,a),\,L_p(\Omega)) \cap L_p((0,\,a),\,H^1_p(\Omega))
\end{equation*}
for all $\sigma \in (0,\,1/2 - 1/2p)$, which can be seen as follows.
Given $\psi \in L_{p^\prime}(\Omega)$, where $1/p + 1/p^\prime = 1$, we choose $\phi \in H^2_{p^\prime}(\Omega)$ to be a solution to
\begin{equation*}
	- \Delta \phi = \psi_0 \quad \mbox{in} \ \Omega, \qquad \partial_\nu \phi = 0 \quad \mbox{on} \ \partial \Omega,
\end{equation*}
if $\beta = 0$, where $\psi_0 := \psi - {(\psi)}_\Omega$, resp.
\begin{equation*}
	- \Delta \phi = \psi \quad \mbox{in} \ \Omega, \qquad [\phi] = 0 \quad \mbox{on} \ \partial \Omega,
\end{equation*}
if $\beta = 1$.
Observe, that \eqnref{pressure-mean} implies $(\,p\,|\,\psi_0\,) = (\,p\,|\,\psi\,)$, if $\beta = 0$.
Using integration by parts and the interior equations we infer
\begin{equation*}
	\begin{array}{rcl}
		(\,p\,|\,\psi\,)
			& = & - (\,p\,|\,\Delta \phi\,) = (\,\nabla p\,|\,\nabla \phi\,) - {\displaystyle{\int_\Gamma [p]\,\partial_\nu \phi\,\mbox{d}\sigma}} \\[1em]
			& = & -\rho \partial_t (\,u\,|\,\nabla \phi\,) + \mu (\,\Delta u\,|\,\nabla \phi\,) - {\displaystyle{\int_\Gamma [p]\,\partial_\nu \phi\,\mbox{d}\sigma}} \\[1em]
			& = & \mu (\,\Delta u\,|\,\nabla \phi\,) - \rho \partial_t {\displaystyle{\int_\Gamma ([u] \cdot \nu)\,[\phi]\,\mbox{d}\sigma}} - {\displaystyle{\int_\Gamma [p]\,\partial_\nu \phi\,\mbox{d}\sigma}}
	\end{array}
\end{equation*}
and the boundary conditions imply
\begin{equation*}
	(\,p\,|\,\psi\,) = \mu (\,\Delta u\,|\,\nabla \phi\,),
\end{equation*}
if $\beta = 0$, resp.
\begin{equation*}
	(\,p\,|\,\psi\,) = \mu (\,\Delta u\,|\,\nabla \phi\,) - {\displaystyle{\int_\Gamma q\,\partial_\nu \phi\,\mbox{d}\sigma}},
\end{equation*}
if $\beta = 1$, with $q := (\,h_w\,|\,\nu\,) - 2 \mu\,\partial_\nu u \cdot \nu$, i.\,e.
\begin{equation*}
	 q \in {}_0 W^{1/2 - 1/2p}_p((0,\,a),\,L_p(\partial \Omega)) \cap L_p((0,\,a),\,W^{1 - 1/p}_p(\partial \Omega)).
\end{equation*}
Employing an integration by parts again, we have
\begin{equation*}
	\mu (\,\Delta u\,|\,\nabla \phi\,) = \mu {\displaystyle{\int_\Gamma \partial_\nu u \cdot [\nabla \phi]\,\mbox{d}\sigma}} - \mu (\,\nabla u\,|\,\nabla^2 \phi\,)
\end{equation*}
and we infer
\begin{equation*}
	(\,p\,|\,\psi\,) = \mu {\displaystyle{\int_\Gamma \partial_\nu u \cdot [\nabla \phi]\,\mbox{d}\sigma}} - {\displaystyle{\int_\Gamma q\,\partial_\nu \phi\,\mbox{d}\sigma}} - \mu (\,\nabla u\,|\,\nabla^2 \phi\,),
\end{equation*}
where we have set $q := 0$, if $\beta = 0$.
Now, we may use the regularity of the functions involved on the right hand side and apply the operator $\partial^\sigma_t$ to obtain the estimate
\begin{equation*}
	\begin{array}{l}
		{\|\partial^\sigma_t p\|}_{L_p((0,a) \times \Omega)} \\[0.5em]
		\qquad \leq c \left({\|\partial^\sigma_t \partial_\nu u\|}_{L_p((0,a) \times \Gamma)}\!+\!{\|\partial^\sigma_t q\|}_{L_p((0,a) \times \Gamma)}\!+\!{\|\partial^\sigma_t \nabla u\|}_{L_p((0,a) \times \Omega)}\right)
	\end{array}
\end{equation*}
and, hence, the desired regularity property of the pressure $p$.
To be precise, we have the estimate
\begin{equation*}
	{\|p\|}_{{}_0 H^\sigma_p((0, a), L_p(\Omega))} \leq c \left( {\|u\|}_{{}_0 \bX_u(a)} + \beta {\|h\|}_{{}_0 \bY^{\alpha, \beta}_{h, \gamma}(a)} \right).
\end{equation*}
Note, that for an inhomogeneous right hand side $f \in {}_0 H^\sigma_p((0, a),\,L_p(\Omega))$ in the momentum equation of \eqnref{s-dom-0} respectively \eqnref{s-dom-1},
an analogous computation yields
\begin{equation*}
	{\|p\|}_{{}_0 H^\sigma_p((0, a), L_p(\Omega))} \leq c \left( {\|u\|}_{{}_0 \bX_u(a)} + {\|f\|}_{{}_0 H^\sigma_p((0, a), L_p(\Omega))} + \beta {\|h\|}_{{}_0 \bY^{\alpha, \beta}_{h, \gamma}(a)} \right).
\end{equation*}
Note, that the embeddings
\begin{equation*}
	{}_0 H^{\sigma/2}_p((0,\,a),\,X), \ {}_0 H^{\sigma}_p((0,\,a),\,X) \hookrightarrow L_q((0,\,a),\,X)
\end{equation*}
are available for some $p < q < \infty$ and every Banach space $X$, where the embedding constant does not depend on $a > 0$ thanks to the homogeneous initial condition.
Hence, we always have
\begin{equation*}
	{}_0 H^{\sigma/2}_p((0,\,a),\,X), \ {}_0 H^{\sigma}_p((0,\,a),\,X) \hookrightarrow L_p((0,\,a),\,X),
\end{equation*}
with embedding constant $c a^\tau$, $c > 0$ being independent of $a > 0$; $\tau = 1/p - 1/q$.

To prove the existence of a unique maximal regular solution to the Stokes equations \eqnref{s-dom-0} resp.\ \eqnref{s-dom-1}
we first choose finitely many points $x_1,\,x_2,\,\dots,\,x_N \in \Gamma$, such that $\Gamma \cap B_r(x_k)$ is the graph of a $BUC^{3-}$-function $\omega_k$ over the tangent plane $T_{x_k} \Gamma$ for $k = 1,\,2,\,\dots,\,N$.
Moreover, we choose $r > 0$ sufficiently small, such that ${\|\nabla \omega_k\|}_{L_\infty(\bR^{n - 1})} < \delta$ for $k = 1,\,2,\,\dots,\,N$ with $\delta \in (0,\,1)$ as in the previous section.
The open sets $U_k := B_r(x_k)$ then constitute a covering of $\Gamma$, which may be completed by an open set $U_0 \subseteq \Omega$ to a covering of $\Omega$.
Finally, we choose a partition of unity $\phi_0,\,\phi_1,\,\dots,\,\phi_N \in C^\infty_0(\bR^n)$ subordinate to the covering $U_0,\,U_1,\,\dots,\,U_N$ of $\Omega$
and cut-off functions $\psi_0,\,\psi_1,\,\dots,\,\psi_N \in C^\infty_0(\bR^n)$ with $\mbox{spt}\,\psi_k \subseteq U_k$ and $\psi_k \equiv 1$ on $\mbox{spt}\,\phi_k$.

Now, every maximal regular solution $(u,\,p)$ of system \eqnref{s-dom-0} resp. \eqnref{s-dom-1} may be decomposed as $u = u_0 + u_1 + \dots + u_N$ and $p = p_0 + p_1 + \dots + p_N$ with $u_k = \phi_k u$ and $p_k = \phi_k p$ for $k = 0,\,1,\,\dots,\,N$.
By construction, $(u_0,\,p_0)$ is a maximal regular solution to the whole space Stokes equations
\begin{equation*}
	\begin{array}{c}
		\rho \partial_t u_0 - \mu \Delta u_0 + \nabla p_0 = \rho F_0(u,\,p), \quad \mbox{div}\,u_0 = G_0(u)
			\quad \quad \mbox{in}\ (0,\,a) \times \bR^n, \\[0.5em]
		u_0(0) = 0
			\quad \quad \mbox{in}\ \bR^n,
	\end{array}
\end{equation*}
and $(u_k,\,p_k)$ is a maximal regular solution to the bent halfspace Stokes equations
\begin{equation*}
	\begin{array}{c}
		\rho \partial_t u_k - \mu \Delta u_k + \nabla p_k = \rho F_k(u,\,p), \quad \mbox{div}\,u_k = G_k(u)
			\quad \quad \mbox{in}\ (0,\,a) \times \bR^n_{\omega_k}, \\[0.5em]
		\begin{array}{c} P_{\Gamma_k} \cB^\alpha(u_k) = H^\alpha_{\tau, k}(u) + h_{\tau, k}, \\[0.5em]
		Q_{\Gamma_k} \cB^\beta(u_k,\,p_k) = H^\beta_{\nu, k}(u) + h_{\nu, k} \end{array}
			\quad \quad \mbox{on}\ (0,\,a) \times \partial \bR^n_{\omega_k}, \\[1.5em]
		u_k(0) = 0
			\quad \quad \mbox{in}\ \bR^n_{\omega_k}
	\end{array}
\end{equation*}
for $k = 1,\,\dots,\,N$ with
\begin{equation*}
	\begin{array}{rcl}
		\rho F_k(u,\,p)        & := & - \mu [\Delta,\,\phi_k] u + [\nabla,\,\phi_k] p
		                          =   - \mu (\Delta \phi_k) u - 2 \mu (\nabla \phi_k \cdot \nabla) u + (\nabla \phi_k) p, \\[0.5em]
		G_k(u)                 & := & [\mbox{div},\,\phi_k] u
		                          =   \nabla \phi_k \cdot u, \\[0.5em]
		H^0_{\tau, k}(u)       & := & [P_{\Gamma_k} \cB^0,\,\phi_k] u
		                          =   0, \\[0.5em]
		H^{\pm 1}_{\tau, k}(u) & := & [P_{\Gamma_k} \cB^{\pm 1},\,\phi_k] u
		                          =   \mu P_{\Gamma_k} (\nabla \phi_k \otimes u \pm u \otimes \nabla \phi_k) \nu_k, \\[0.5em]
		H^0_{\nu, k}(u)        & := & [Q_{\Gamma_k} \cB^0,\,\phi_k] (u)
		                          =   0, \\[0.5em]
		H^{+1}_{\nu, k}(u)     & := & [Q_{\Gamma_k} \cB^1,\,\phi_k] (u)
		                          =   2 \mu ([\nabla \phi_k \otimes u] \nu_k \cdot \nu_k) \nu_k
	\end{array}
\end{equation*}
and $h_{\tau, k} = \phi_k h_v$, $h_{\nu, k} = 0$, if $\beta = 0$, resp.~$h_{\tau, k} = 0$, $h_{\nu, k} = \phi_k h_w$, if $\beta = 1$.
Therefore, $u_0 = v_0 + \nabla \eta_0$, $p_0 = q_0 - \rho \partial_t \eta_0 + \mu \Delta \eta_0$, and $u_k = v_k + \nabla \eta_k + \bar{v}_k$, $p_k = q_k - \rho \partial_t \eta_k + \mu \Delta  \eta_k + \bar{q}_k$ for $k = 1,\,\dots,\,N$,
where
\begin{equation*}
	\eta_0 \in {}_0 H^1_p((0,\,a),\,H^2_p(\bR^n)) \cap H^{1/2}_p((0,\,a),\,H^3_p(\bR^n))
\end{equation*}
is a solution to the wholespace problem
\begin{equation*}
	- \Delta \eta_0 = - G_0(u) \quad \quad \mbox{in}\ (0,\,a) \times \bR^n
\end{equation*}
and
\begin{equation*}
	\eta_k \in {}_0 H^1_p((0,\,a),\,H^2_p(\bR^n_{\omega_k})) \cap H^{1/2}_p((0,\,a),\,H^3_p(\bR^n_{\omega_k}))
\end{equation*}
is a solution to the bent halfspace problem
\begin{equation*}
	- \Delta \eta_k = - G_k(u) \quad \quad \mbox{in}\ (0,\,a) \times \bR^n_{\omega_k}
\end{equation*}
complemented by the boundary condition
\begin{equation*}
	\partial_\nu \eta_k = 0 \quad \quad \mbox{on}\ (0,\,a) \times \partial \bR^n_{\omega_k},
\end{equation*}
if $\beta = 0$, resp.
\begin{equation*}
	[\eta_k] = 0 \quad \quad \mbox{on}\ (0,\,a) \times \partial \bR^n_{\omega_k},
\end{equation*}
if $\beta = 1$ for $k = 1,\,\dots,\,N$.
Moreover $(v_0,\,p_0)$ is a maximal regular solution to the wholespace Stokes equations
\begin{equation}
	\label{eqn:part-problem-ws-mod}
	\begin{array}{c}
		\rho \partial_t v_0 - \mu \Delta v_0 + \nabla q_0 = \rho F_0(u,\,p), \quad \mbox{div}\,v_0 = 0
			\quad \quad \mbox{in}\ (0,\,a) \times \bR^n, \\[0.5em]
		v_0(0) = 0
			\quad \quad \mbox{in}\ \bR^n,
	\end{array}
\end{equation}
$(v_k,\,q_k)$ is a maximal regular solution to the bent halfspace Stokes equations
\begin{equation}
	\label{eqn:part-problem-bhs-mod}
	\begin{array}{c}
		\rho \partial_t v_k - \mu \Delta v_k + \nabla q_k = \rho F_k(u,\,p), \quad \mbox{div}\,v_k = 0
			\quad \quad \mbox{in}\ (0,\,a) \times \bR^n_{\omega_k}, \\[0.5em]
		\begin{array}{c} P_{\Gamma_k} \cB^\alpha(v_k) = H^\alpha_{\tau, k}(u) - P_{\Gamma_k} \cB^\alpha(\nabla \eta_k), \\[0.5em]
		\begin{array}{l} Q_{\Gamma_k} \cB^\beta(v_k,\,q_k) = H^\beta_{\nu, k}(u) \\[0.5em] \quad - \ Q_{\Gamma_k} \cB^\beta(\nabla \eta_k,\,\mu \Delta \eta_k - \rho \partial_t \eta_k) \end{array} \end{array}
			\quad \quad \mbox{on}\ (0,\,a) \times \partial \bR^n_{\omega_k}, \\[2.5em]
		v_k(0) = 0
			\quad \quad \mbox{in}\ \bR^n_{\omega_k}
	\end{array}
\end{equation}
and $(\bar{v}_k,\,\bar{q}_k)$ is a maximal regular solution to the bent halfspace Stokes equations
\begin{equation}
	\label{eqn:part-problem-bhs-data}
	\begin{array}{c}
		\rho \partial_t \bar{v}_k - \mu \Delta \bar{v}_k + \nabla \bar{q}_k = 0, \quad \mbox{div}\,\bar{v}_k = 0
			\quad \quad \mbox{in}\ (0,\,a) \times \bR^n_{\omega_k}, \\[0.5em]
		\begin{array}{c} P_{\Gamma_k} \cB^\alpha(\bar{v}_k) = h_{\tau, k}, \\[0.5em]
		Q_{\Gamma_k} \cB^\beta(\bar{v}_k,\,\bar{q}_k) = h_{\nu, k} \end{array}
			\quad \quad \mbox{on}\ (0,\,a) \times \partial \bR^n_{\omega_k}, \\[1.5em]
		\bar{v}_k(0) = 0
			\quad \quad \mbox{in}\ \bR^n_{\omega_k}
	\end{array}
\end{equation}
for $k = 1,\,\dots,\,N$.
Now, every term on the right hand side of \eqnref{part-problem-ws-mod} and \eqnref{part-problem-bhs-mod} carries additional time regularity.
To be precise, we may exploit the extra time regularity of the pressure to obtain
\begin{equation*}
	\begin{array}{rcl}
		{\|\rho F_k(u,\,p)\|}_{{}_0 H^{\sigma}_p((0, a), L_p(\bR^n_{\omega_k},\,\bR^n))}
			& \leq & c \left( {\|u\|}_{{}_0 \bX_u(a)} + {\|p\|}_{{}_0 H^{\sigma}_p((0, a), L_p(\Omega))} \right) \\[1em]
			& \leq & c \left( {\|u\|}_{{}_0 \bX_u(a)} + \beta {\|h\|}_{{}_0 \bY^{\alpha, \beta}_{h, \gamma}(a)} \right)
	\end{array}
\end{equation*}
and, therefore,
\begin{equation*}
	{\|\rho F_k(u,\,p)\|}_{L_p((0,\,a) \times \bR^n_{\omega_k},\,\bR^n)} \leq c a^\tau \left( {\|u\|}_{{}_0 \bX_u(a)} + \beta {\|h\|}_{{}_0 \bY^{\alpha, \beta}_{h, \gamma}(a)} \right)
\end{equation*}
for $k = 0,\,\dots,\,N$.
The commutators of the boundary conditions are of lower order and we have
\begin{equation*}
	{\|H^\alpha_{\tau, k}(u) + H^\beta_{\nu, k}(u)\|}_{{}_0 \bY^{\alpha, \beta}_{h, \gamma}(a)} \leq c a^\tau {\|u\|}_{{}_0 \bX_u(a)}
\end{equation*}
for $k = 1,\,\dots,\,N$.
If $\beta = 0$, we have $Q_{\Gamma_k} \cB^\beta(\nabla \eta_k,\,\mu \Delta \eta_k - \rho \partial_t \eta_k) = 0$
and, thus,
\begin{equation*}
	{\|v_k\|}_{{}_0 \bX_u(a)} + {\|q_k\|}_{\bX^\beta_{p, \gamma}(a)} \leq c a^\tau {\|u\|}_{{}_0 \bX_u(a)}
\end{equation*}
as well as
\begin{equation*}
	\begin{array}{rcl}
		{\|q_k\|}_{{}_0 H^{\sigma}_p((0, a), L_p(\bR^n_{\omega_k}))}
			& \leq & c \left( {\|v_k\|}_{{}_0 \bX_u(a)} + {\|\rho F_k(u,\,p)\|}_{{}_0 H^{\sigma}_p((0, a), L_p(\bR^n_{\omega_k},\,\bR^n))} \right)  \\[1em]
			& \leq & c {\|u\|}_{{}_0 \bX_u(a)}.
	\end{array}
\end{equation*}
for $k = 0,\,\dots,\,N$.
If $\beta = 1$, we have
\begin{equation*}
	\begin{array}{rcl}
		Q_{\Gamma_k} \cB^\beta(\nabla \eta_k,\,\mu \Delta \eta_k - \rho \partial_t \eta_k)
			& = & 2 \mu [\nabla^2 \eta_k]\nu \cdot \nu - [\mu \Delta \eta_k - \rho \partial_t \eta_k] \\[0.5em]
			& = & 2 \mu [\nabla^2 \eta_k]\nu \cdot \nu - \mu [\Delta \eta_k].
	\end{array}
\end{equation*}
However,
\begin{equation*}
	\begin{array}{l}
		{\|\nabla^2 \eta_k\|}_{{}_0 H^{1/2}_p((0, a), L_p(\bR^n_{\omega_k})) \cap L_p((0, a), H^1_p(\bR^n_{\omega_k}))} \\[0.5em]
		\qquad \leq c a^\tau {\|\nabla^2 w_k\|}_{{}_0 H^1_p((0, a), L_p(\bR^n_{\omega_k})) \cap H^{1/2}_p((0, a), H^1_p(\bR^n_{\omega_k}))} \\[0.5em]
		\qquad \leq c a^\tau {\|u\|}_{{}_0 \bX_u(a)}
	\end{array}
\end{equation*}
and we again infer
\begin{equation*}
	{\|v_k\|}_{{}_0 \bX_u(a)} + {\|q_k\|}_{\bX^\beta_{p, \gamma}(a)} \leq c a^\tau {\|u\|}_{{}_0 \bX_u(a)}
\end{equation*}
as well as
\begin{equation*}
	\begin{array}{rcl}
		{\|q_k\|}_{{}_0 H^{\sigma}_p((0, a), L_p(\bR^n_{\omega_k}))}
			& \leq & c \left( {\|v_k\|}_{{}_0 \bX_u(a)} + {\|\rho F_k(u,\,p)\|}_{{}_0 H^{\sigma}_p((0, a), L_p(\bR^n_{\omega_k},\,\bR^n))} \right) \\[1em]
			&      & \qquad + \ c {\|[\nabla^2 \eta_k]\|}_{{}_0 \bY^{\alpha, \beta}_{h, \gamma}(a)} \\[0.5em]
			& \leq & c {\|u\|}_{{}_0 \bX_u(a)}.
	\end{array}
\end{equation*}
for $k = 1,\,\dots,\,N$.
Finally, we have
\begin{equation*}
	{\|\bar{v}_k\|}_{{}_0 \bX_u(a)} + {\|\bar{q}_k\|}_{\bX^\beta_{p, \gamma}(a)} \leq c {\|h\|}_{{}_0 \bY^{\alpha, \beta}_{h, \gamma}(a)}
\end{equation*}
and
\begin{equation*}
		{\|\bar{q}_k\|}_{H^{\sigma}_p((0, a), L_p(\bR^n_{\omega_k}))}
			\leq c \left( {\|\bar{v}_k\|}_{{}_0 \bX_u(a)} + \beta {\|h\|}_{{}_0 \bY^{\alpha, \beta}_{h, \gamma}(a)} \right) 
			\leq c {\|h\|}_{{}_0 \bY^{\alpha, \beta}_{h, \gamma}(a)}
\end{equation*}
for $k = 1,\,\dots,\,N$.
Now, these estimates imply
\begin{equation*}
	\begin{array}{rcl}
	{\|\mu \Delta \eta_0 - \rho \partial_t \eta_0\|}_{{}_0 H^{\sigma}_p((0, a), L_p(\bR^n_{\omega_k}))}
		&    = & {\|p_0 - q_0\|}_{{}_0 H^{\sigma}_p((0, a), L_p(\bR^n_{\omega_k}))} \\[1em]
		& \leq & c \left( {\|u\|}_{{}_0 \bX_u(a)} + {\|h\|}_{{}_0 \bY^{\alpha, \beta}_{h, \gamma}(a)} \right)
	\end{array}
\end{equation*}
and
\begin{equation*}
	\begin{array}{rcl}
		{\|\mu \Delta \eta_k - \rho \partial_t \eta_k\|}_{{}_0 H^{\sigma}_p((0, a), L_p(\bR^n_{\omega_k}))}
			& =    & {\|p_k - q_k - \bar{q}_k\|}_{{}_0 H^{\sigma}_p((0, a), L_p(\bR^n_{\omega_k}))} \\[0.5em]
			& \leq & c \left( {\|u\|}_{{}_0 \bX_u(a)} + {\|h\|}_{{}_0 \bY^{\alpha, \beta}_{h, \gamma}(a)} \right)
	\end{array}
\end{equation*}
for $k = 1,\,\dots,\,N$ and since
\begin{equation*}
	{\|\Delta \eta_k\|}_{{}_0 H^1_p((0, a), L_p(\bR^n_{\omega_k}))} \leq c {\|u\|}_{{}_0 \bX_u(a)}
\end{equation*}
we further infer
\begin{equation*}
	{\|\rho \partial_t \eta_k\|}_{{}_0 H^{\sigma}_p((0, a), L_p(\bR^n_{\omega_k}))},{\|\mu \Delta \eta_k\|}_{{}_0 H^{\sigma}_p((0, a), L_p(\bR^n_{\omega_k}))} \leq c \left( {\|u\|}_{{}_0 \bX_u(a)} + {\|h\|}_{{}_0 \bY^{\alpha, \beta}_{h, \gamma}(a)} \right)
\end{equation*}
for $k = 0,\,\dots,\,N$.
Thus, using
\begin{equation*}
	{\|\partial_t \eta_k\|}_{L_p((0, a), H^2_p(\bR^n_{\omega_k}))}, \ {\|\Delta \eta_k\|}_{{}_0 H^1_p((0, a), L_p(\bR^n_{\omega_k})) \cap L_p((0, a), H^2_p(\bR^n_{\omega_k}))} \leq c {\|u\|}_{{}_0 \bX_u(a)}
\end{equation*}
we also have
\begin{equation*}
	\begin{array}{rcl}
		{\|\partial_t \nabla \eta_k\|}_{L_p((0, a), L_p(\bR^n_{\omega_k}))}
			& \leq & c a^\tau {\|\partial_t \eta_k\|}_{{}_0 H^{\sigma/2}_p((0, a), H^1_p(\bR^n_{\omega_k}))} \\[0.5em]
			& \leq & c a^\tau {\|\partial_t \eta_k\|}_{{}_0 H^{\sigma}_p((0, a), L_p(\bR^n_{\omega_k})) \cap L_p((0, a), H^2_p(\bR^n_{\omega_k}))} \\[0.5em]
			& \leq & c a^\tau {\|u\|}_{{}_0 \bX_u(a)} + c {\|h\|}_{{}_0 \bY^{\alpha, \beta}_{h, \gamma}(a)}
	\end{array}
\end{equation*}
and
\begin{equation*}
	\begin{array}{rcl}
		{\|\Delta \nabla \eta_k\|}_{L_p((0, a), L_p(\bR^n_{\omega_k}))}
			& \leq & c a^\tau {\|\Delta \eta_k\|}_{{}_0 H^{1/2}_p((0, a), H^1_p(\bR^n_{\omega_k}))} \\[0.5em]
			& \leq & c a^\tau {\|\Delta \eta_k\|}_{{}_0 H^1_p((0, a), L_p(\bR^n_{\omega_k})) \cap L_p((0, a), H^2_p(\bR^n_{\omega_k}))} \\[0.5em]
			& \leq & c a^\tau {\|u\|}_{{}_0 \bX_u(a)}
	\end{array}
\end{equation*}
for $k = 0,\,\dots,\,N$.
Hence, we have
\begin{equation*}
	{\|\nabla \eta_k\|}_{{}_0 \bX_u(a)} + {\|\mu \Delta \eta_k - \rho \partial_t \eta_k\|}_{\bX^\beta_{p, \gamma}(a)} \leq c a^\tau {\|u\|}_{{}_0 \bX_u(a)} + c {\|h\|}_{{}_0 \bY^{\alpha, \beta}_{h, \gamma}(a)},
\end{equation*}
which together with the above estimates implies
\begin{equation*}
	{\|u_k\|}_{{}_0 \bX_u(a)} + {\|p_k\|}_{\bX^\beta_{p, \gamma}(a)} \leq c a^\tau {\|u\|}_{{}_0 \bX_u(a)} + c {\|h\|}_{{}_0 \bY^{\alpha, \beta}_{h, \gamma}(a)},
\end{equation*}
Finally, we have
\begin{equation*}
	(u,\,p) = \sum^N_{k = 0} (\psi_k u_k,\,\psi_k p_k) = \sum^N_{k = 0} (u_k,\,p_k)
\end{equation*}
and choosing $a > 0$ sufficiently small, we deduce the uniqueness of the solution as well as the maximal regularity estimate.

To prove solvability of the Stokes equations,
let $\alpha \in \{\,-1,\,0,\,1\,\}$, $\beta \in \{\,0,\,1\,\}$ and $\gamma = -\infty$, if $\beta = 0$, resp.~$\gamma = 1/2 - 1/2p$, if $\beta = 1$.
We consider the bounded linear operator
\begin{equation*}
	L^{\alpha, \beta}: {}_0 \bX^\beta_\gamma(a) \longrightarrow {}_0 \bY^{\alpha, \beta}_\gamma(a)
\end{equation*}
defined by the left hand side of the Stokes equations,
where we set
\begin{equation*}
	{}_0 \bX^\beta_\gamma(a) := \left\{\,(u,\,p) \in \bX_u(a) \times \bX^\beta_{p, \gamma}(a)\,:\,u(0) = 0\,\right\}
\end{equation*}
and define
\begin{equation*}
	{}_0 \bY^{\alpha, \beta}_\gamma(a) := \left\{\,(f,\,g,\,h) \in \bY_f(a) \times \bY_g(a) \times \bY^{\alpha, \beta}_{h, \gamma}(a)\,:\,(f,\,g,\,h,\,0) \in \bY^{\alpha, \beta}_\gamma(a)\,\right\}.
\end{equation*}
By the above considerations, $L^{\alpha, \beta}$ is injective with closed range and it remains to prove surjectivity.
To accomplish this, it is convenient to construct a bounded linear right inverse
\begin{equation*}
	S^{\alpha, \beta}: {}_0 \bY^{\alpha, \beta}_\gamma(a) \longrightarrow {}_0 \bX^\beta_\gamma(a).
\end{equation*}

First, let $\beta = 0$ and $\gamma = -\infty$.
By \Thmref{splitting} and the construction at the beginning of this section, there exists a bounded linear operator
\begin{equation*}
	S^{\alpha, \beta}_1: {}_0 \bY^{\alpha, \beta}_\gamma(a) \longrightarrow {}_0 \bX^\beta_\gamma(a),
\end{equation*}
such that $(\bar{v},\,\bar{p}) = S^{\alpha, \beta}_1(f,\,g,\,h)$ satisfies
\begin{equation*}
	\begin{array}{c}
		\rho\partial_t \bar{v} - \mu \Delta \bar{v} + \nabla \bar{q} = \rho f,
		\quad \quad \mbox{div}\,\bar{v} = g
			\quad \quad \mbox{in}\ (0,\,a) \times \Omega, \\[0.5em]
		Q_\Gamma \cB^\beta(\bar{v}) = Q_\Gamma h,
			\quad \quad \mbox{on}\ (0,\,a) \times \partial \Omega, \\[0.5em]
		\bar{v}(0) = 0
			\quad \quad \mbox{in}\ \Omega.
	\end{array}
\end{equation*}
Now, we construct a bounded linear operator
\begin{equation*}
	S^{\alpha, \beta}_2: {}_\tau \bY^{\alpha, \beta}_{h, \gamma}(a) := \left\{\,\eta \in \bY^{\alpha, \beta}_{h, \gamma}(a)\,:\,Q_\Gamma \eta = 0,\,\eta(0) = 0\,\right\} \longrightarrow {}_0 \bX^\beta(a)
\end{equation*}
and then set
\begin{equation*}
	S^{\alpha, \beta}(f,\,g,\,h) := S^{\alpha, \beta}_1(f,\,g,\,h) + S^{\alpha, \beta}_2 (P_\Gamma h - P_\Gamma \cB^\alpha S^{\alpha, \beta}_1(f,\,g,\,h))
\end{equation*}
for $(f,\,g,\,h) \in {}_0 \bY^{\alpha, \beta}_\gamma(a)$.
Given $\bar{h} \in {}_\tau \bY^{\alpha, \beta}_{h, \gamma}(a)$, we define $(v_k,\,q_k)$ to be the solution to the bent halfspace Stokes equations
\begin{equation*}
	\begin{array}{c}
		\rho \partial_t v_k - \mu \Delta v_k + \nabla q_k = 0, \quad \mbox{div}\,v_k = 0
			\quad \quad \mbox{in}\ (0,\,a) \times \bR^n_{\omega_k}, \\[0.5em]
		P_{\Gamma_k} \cB^\alpha(v_k) = \phi_k P_\Gamma \bar{h},
		\quad Q_{\Gamma_k} \cB^\beta(v_k,\,q_k) = 0
			\quad \quad \mbox{on}\ (0,\,a) \times \partial \bR^n_{\omega_k}, \\[0.5em]
		v_k(0) = 0
			\quad \quad \mbox{in}\ \bR^n_{\omega_k}
	\end{array}
\end{equation*}
for $k = 1,\,\dots,\,N$.
Then we define $\eta_k \in {}_0 H^1_p((0,\,a),\,H^2_p(\Omega)) \cap H^{1/2}_p((0,\,a),\,H^3_p(\Omega))$ to be the unique solution to
\begin{equation*}
	- \Delta \eta_k = - \nabla \psi_k \cdot v_k \quad \mbox{in} \ \Omega, \qquad \partial_\nu \eta_k = 0 \quad \mbox{on} \ \partial \Omega
\end{equation*}
again for $k = 1,\,\dots,\,N$.
Finally, we set
\begin{equation*}
	(v,\,q) = \sum^N_{k = 1} (\psi_k v_k - \nabla \eta_k,\,q_k + \rho \partial_t \eta_k - \mu \Delta \eta_k)
\end{equation*}
and define $S^{\alpha, \beta}_2 \bar{h} := (v,\,q)$.
This way for $(f,\,g,\,h) \in {}_0 \bY^{\alpha, \beta}_\gamma(a)$ we have
\begin{equation*}
	\begin{array}{l}
		L^{\alpha, \beta} S^{\alpha, \beta}(f,\,g,\,h) \\[0.5em]
			\qquad = L^{\alpha, \beta} S^{\alpha, \beta}_1(f,\,g,\,h) + L^{\alpha, \beta} S^{\alpha, \beta}_2 (P_\Gamma h - P_\Gamma \cB^\alpha S^{\alpha, \beta}_1(f,\,g,\,h)) \\[0.5em]
			\qquad = L^{\alpha, \beta} (\bar{v},\,\bar{q}) + L^{\alpha, \beta} (v,\,q)
	\end{array}
\end{equation*}
with
\begin{equation*}
	(\bar{v},\,\bar{q}) = S^{\alpha, \beta}_1(f,\,g,\,h), \qquad (v,\,q) = S^{\alpha, \beta}_2 (P_\Gamma h - P_\Gamma \cB^\alpha \bar{v}).
\end{equation*}
Now observe,
\begin{equation*}
	\rho \partial_t (\bar{v} + v) - \mu \Delta (\bar{v} + v) + \nabla (\bar{q} + q) = \rho f + \sum^N_{k = 1} \Big( - \mu [\Delta,\,\psi_k] v_k + [\nabla,\,\psi_k] q_k \Big)
\end{equation*}
by definition of $S^{\alpha, \beta}_1$ and $S^{\alpha, \beta}_2$.
Furthermore
\begin{equation*}
	\mbox{div}\,(\bar{v} + v) = g,\ P_\Gamma \cB^\alpha (\bar{v} + v) = P_\Gamma h + \sum^N_{k = 1} \Big( [P_{\Gamma_k}\cB^\alpha,\,\psi_k] v_k - P_\Gamma \cB^\alpha(\nabla \eta_k) \Big)
\end{equation*}
as well as
\begin{equation*}
	Q_\Gamma \cB^\beta(\bar{v} + v,\,\bar{q} + q) = Q_\Gamma h \quad \mbox{and} \quad (\bar{v} + v)(0) = 0.
\end{equation*}
Hence, $L^{\alpha, \beta} S^{\alpha, \beta} = I + R^{\alpha, \beta}$ with
\begin{equation*}
	R^{\alpha, \beta}_f(f,\,g,\,h) = \sum^N_{k = 1} \Big( - \mu [\Delta,\,\psi_k] v_k + [\nabla,\,\psi_k] q_k \Big)
\end{equation*}
and $R^{\alpha, \beta}_g(f,\,g,\,h) = 0$, $Q_\Gamma R^{\alpha, \beta}_h(f,\,g,\,h) = 0$ as well as
\begin{equation*}
	P_\Gamma R^{\alpha, \beta}_h(f,\,g,\,h) = \sum^N_{k = 1} \Big( [P_{\Gamma_k}\cB^\alpha,\,\psi_k] v_k - P_\Gamma \cB^\alpha(\nabla \eta_k) \Big).
\end{equation*}
The subscripts denote the different components of $R^{\alpha, \beta}$.
Now,
\begin{equation*}
	{\|v_k\|}_{{}_0 \bX_u(a)} + {\|q_k\|}_{\bX^\beta_{p, \gamma}(a)} \leq c {\|\bar{h}\|}_{{}_0 \bY^{\alpha, \beta}_{h, \gamma}(a)}
\end{equation*}
with $\bar{h} = P_\Gamma h - P_\Gamma \cB^\alpha \bar{v}$ and
\begin{equation*}
		{\|q_k\|}_{{}_0 H^{\sigma}_p((0, a), L_p(\bR^n_{\omega_k}))}
			\leq c {\|v_k\|}_{{}_0 \bX_u(a)} 
			\leq c {\|\bar{h}\|}_{{}_0 \bY^{\alpha, \beta}_{h, \gamma}(a)},
\end{equation*}
which implies
\begin{equation*}
	\begin{array}{rcl}
		{\|[\nabla,\,\psi_k] q_k - \mu [\Delta,\,\psi_k] v_k\|}_{\bY_f(a)}
			& \leq & c a^\tau \left( {\|v_k\|}_{{}_0 \bX_u(a)} + {\|q_k\|}_{{}_0 H^{\sigma}_p((0, a), L_p(\bR^n_{\omega_k}))} \right) \\[1em]
			& \leq & c a^\tau {\|\bar{h}\|}_{{}_0 \bY^{\alpha, \beta}_{h, \gamma}(a)},
	\end{array}
\end{equation*}
since the commutator terms are of lower order.
Analogously
\begin{equation*}
	{\|[P_{\Gamma_k}\cB^\alpha,\,\psi_k] v_k\|}_{{}_0 \bY^{\alpha, \beta}_{h, \gamma}(a)} \leq c a^\tau {\|v_k\|}_{{}_0 \bX_u(a)} \leq c a^\tau {\|\bar{h}\|}_{{}_0 \bY^{\alpha, \beta}_{h, \gamma}(a)}.
\end{equation*}
To estimate the final term $P_\Gamma \cB^\alpha(\nabla \eta_k)$ we set $\bar{v}_k := \psi_k v_k - \nabla \eta_k$,
$\bar{q}_k := \psi_k q_k + \rho \partial_t \eta_k - \mu \Delta \eta_k$ and observe, that
\begin{equation*}
	\begin{array}{c}
		\rho \partial_t \bar{v}_k - \mu \Delta \bar{v}_k + \nabla \bar{q}_k = - \mu [\Delta,\,\psi_k] v_k + [\nabla,\,\psi_k] q_k, \quad \mbox{div}\,\bar{v}_k = 0
			\quad \mbox{in}\ (0,\,a) \times \Omega, \\[0.5em]
		\begin{array}{c} P_\Gamma \cB^\alpha(\bar{v}_k) = \phi_k P_\Gamma \bar{h} + [P_{\Gamma_k}\cB^\alpha,\,\psi_k] v_k - P_\Gamma \cB^\alpha(\nabla \eta_k), \\[0.5em]
		Q_{\Gamma_k} \cB^\beta(\bar{v}_k,\,\bar{q}_k) = 0 \end{array}
			\quad \mbox{on}\ (0,\,a) \times \partial \Omega, \\[1.5em]
		\bar{v}_k(0) = 0
			\quad \mbox{in}\ \Omega.
	\end{array}
\end{equation*}
Now,
\begin{equation*}
	\begin{array}{rcl}
		{\|\bar{v}_k\|}_{{}_0 \bX_u(a)}
			& \leq & c \left( {\|v_k\|}_{{}_0 \bX_u(a)} + {\|q_k\|}_{L_p((0,\,a) \times \bR^n_{\omega_k}))} + {\|\nabla \eta_k\|}_{{}_0 \bX_u(a)} + {\|\bar{h}\|}_{{}_0 \bY^{\alpha, \beta}_{h, \gamma}(a)} \right) \\[1em]
			& \leq & c {\|\bar{h}\|}_{{}_0 \bY^{\alpha, \beta}_{h, \gamma}(a)}
	\end{array}
\end{equation*}
and
\begin{equation*}
	\begin{array}{l}
		{\|[\nabla,\,\psi_k] q_k - \mu [\Delta,\,\psi_k] v_k\|}_{{}_0 H^{\sigma}_p((0, a), L_p(\bR^n_{\omega_k}))} \\[1em]
			\qquad \leq c \left( {\|v_k\|}_{{}_0 \bX_u(a)} + {\|q_k\|}_{{}_0 H^{\sigma}_p((0, a), L_p(\bR^n_{\omega_k}))} \right)
			\leq c {\|\bar{h}\|}_{{}_0 \bY^{\alpha, \beta}_{h, \gamma}(a)}
	\end{array}
\end{equation*}
imply
\begin{equation*}
	\begin{array}{l}
		{\|\bar{q}_k\|}_{{}_0 H^{\sigma}_p((0, a), L_p(\bR^n_{\omega_k}))} \\[1em]
			\qquad \leq c \left( {\|\bar{v}_k\|}_{{}_0 \bX_u(a)} + {\|[\nabla,\,\psi_k] q_k - \mu [\Delta,\,\psi_k] v_k\|}_{{}_0 H^{\sigma}_p((0, a), L_p(\bR^n_{\omega_k}))} \right) \\[1em]
			\qquad \leq c {\|\bar{h}\|}_{{}_0 \bY^{\alpha, \beta}_{h, \gamma}(a)}
	\end{array}
\end{equation*}
and we infer
\begin{equation*}
	{\|\mu \Delta \eta_k - \rho \partial_t \eta_k\|}_{{}_0 H^{\sigma}_p((0, a), L_p(\Omega))} = {\|\psi_k q_k - \bar{q}_k\|}_{{}_0 H^{\sigma}_p((0, a), L_p(\Omega))} \leq c {\|\bar{h}\|}_{{}_0 \bY^{\alpha, \beta}_{h, \gamma}(a)}.
\end{equation*}
Since
\begin{equation*}
	{\|\Delta \eta_k\|}_{{}_0 H^1_p((0, a), L_p(\Omega))} \leq c {\|v_k\|}_{{}_0 \bX_u(a)} \leq c {\|\bar{h}\|}_{{}_0 \bY^{\alpha, \beta}_{h, \gamma}(a)}
\end{equation*}
we further infer
\begin{equation*}
	{\|\rho \partial_t \eta_k\|}_{{}_0 H^{\sigma}_p((0, a), L_p(\Omega))}, \ {\|\mu \Delta \eta_k\|}_{{}_0 H^{\sigma}_p((0, a), L_p(\Omega))} \leq c {\|\bar{h}\|}_{{}_0 \bY^{\alpha, \beta}_{h, \gamma}(a)}.
\end{equation*}
Thus, using
\begin{equation*}
	\begin{array}{l}
		{\|\partial_t \eta_k\|}_{L_p((0, a), H^2_p(\Omega))}, \ {\|\Delta \eta_k\|}_{{}_0 H^1_p((0, a), L_p(\Omega)) \cap L_p((0, a), H^2_p(\Omega))} \\[0.5em]
			\qquad \leq c {\|v_k\|}_{{}_0 \bX_u(a)} \leq c {\|\bar{h}\|}_{{}_0 \bY^{\alpha, \beta}_{h, \gamma}(a)}
	\end{array}
\end{equation*}
we also have
\begin{equation*}
	\begin{array}{rcl}
		{\|\partial_t \nabla \eta_k\|}_{L_p((0, a), L_p(\Omega))}
			& \leq & c a^\tau {\|\partial_t \eta_k\|}_{{}_0 H^{\sigma/2}_p((0, a), H^1_p(\Omega))} \\[0.5em]
			& \leq & c a^\tau {\|\partial_t \eta_k\|}_{{}_0 H^{\sigma}_p((0, a), L_p(\Omega)) \cap L_p((0, a), H^2_p(\Omega))} \\[0.5em]
			& \leq & c a^\tau {\|\bar{h}\|}_{{}_0 \bY^{\alpha, \beta}_{h, \gamma}(a)}
	\end{array}
\end{equation*}
and
\begin{equation*}
	\begin{array}{rcl}
		{\|\Delta \nabla \eta_k\|}_{L_p((0, a), L_p(\Omega))}
			& \leq & c a^\tau {\|\Delta \eta_k\|}_{{}_0 H^{1/2}_p((0, a), H^1_p(\Omega))} \\[0.5em]
			& \leq & c a^\tau {\|\Delta \eta_k\|}_{{}_0 H^1_p((0, a), L_p(\Omega)) \cap L_p((0, a), H^2_p(\Omega))} \\[0.5em]
			& \leq & c a^\tau {\|\bar{h}\|}_{{}_0 \bY^{\alpha, \beta}_{h, \gamma}(a)}.
	\end{array}
\end{equation*}
Hence, we have
\begin{equation*}
	{\|\nabla \eta_k\|}_{{}_0 \bX_u(a)} + {\|\mu \Delta \eta_k - \rho \partial_t \eta_k\|}_{\bX^\beta_{p, \gamma}(a)} \leq c a^\tau {\|\bar{h}\|}_{{}_0 \bY^{\alpha, \beta}_{h, \gamma}(a)},
\end{equation*}
which implies
\begin{equation*}
	{\|P_\Gamma \cB^\alpha(\nabla \eta_k)\|}_{{}_0 \bY^{\alpha, \beta}_{h, \gamma}} \leq c a^\tau {\|\bar{h}\|}_{{}_0 \bY^{\alpha, \beta}_{h, \gamma}(a)}.
\end{equation*}
Since
\begin{equation*}
	{\|\bar{h}\|}_{{}_0 \bY^{\alpha, \beta}_{h, \gamma}(a)} \leq c {\|(f,\,g,\,h)\|}_{{}_0 \bY^{\alpha, \beta}_\gamma(a)}
\end{equation*}
we obtain
\begin{equation*}
	{\|R^{\alpha, \beta}(f,\,g,\,h)\|}_{{}_0 \bY^{\alpha, \beta}_\gamma(a)} \leq c a^\tau {\|(f,\,g,\,h)\|}_{{}_0 \bY^{\alpha, \beta}_\gamma(a)}.
\end{equation*}
Thus, choosing $a > 0$ sufficiently small, $R^{\alpha, \beta}$ is invertible by a Neumann series, which yields the right inverse $S^{\alpha, \beta} (I + R^{\alpha, \beta})^{-1}$ for $L^{\alpha, \beta}$.

The case $\beta = 1$ and $\gamma = 1/2 - 1/2p$ may be treated analogously.
The only difference is to use a Dirichlet problem to construct the $\eta_k$ instead of a Neumann problem.
However, a more elegant proof is also available by a homotopy argument.
In fact, one may repeat the arguments of Sections \ref{sec:halfspace} to \ref{sec:domain},
to infer, that the bounded linear operators
\begin{equation*}
	L^\alpha_\tau: {}_0 \bX^{+1}_{1/2 - 1/2p}(a) \longrightarrow {}_0 \bY^{\alpha, +1}_{1/2 - 1/2p}(a)
\end{equation*}
defined by the left hand side of the Stokes equations
\begin{equation*}
	\begin{array}{c}
		\rho\partial_t u - \mu \Delta u + \nabla p = \rho f,
		\quad \quad \mbox{div}\,u = g
			\quad \quad \mbox{in}\ (0,\,a) \times \Omega, \\[0.5em]
		P_\Gamma \cB^\alpha u = P_\Gamma h,
		\quad \quad \tau\,2 \mu\,\partial_\nu u \cdot \nu - [p] = h \cdot \nu,
			\quad \quad \mbox{on}\ (0,\,a) \times \partial \Omega, \\[0.5em]
		u(0) = 0
			\quad \quad \mbox{in}\ \Omega.
	\end{array}
\end{equation*}
are injective with closed range for all $\tau \in [0,\,1]$.
Hence, these operators are semi-Fredholm.
Thus, $L^\alpha_1 = L^{\alpha, +1}$ is Fredholm of index zero and, hence, an isomorphism, since $L^\alpha_0 = L^{\alpha, -1}$ enjoys this property by \Thmref{splitting}.
This completes the proof of \Thmref{s}.

\section{Local Well-Posedness of the Navier-Stokes Equations}\label{sec:n}
Finally, we prove \Thmref{n} for a bounded domain $\Omega \subseteq \bR^n$ with boundary $\Gamma := \partial \Omega$ of class $C^{3-}$.
Following \Remref{n}, it is sufficient to construct unique local-in-time strong solutions to the Navier-Stokes equations
\begin{equation*}
	\tag*{${(\mbox{N})}^{\infty, \Omega, \cB}_{f, u_0}$}
	\begin{array}{c}
		\rho \partial_t u + \mbox{div}(\rho u \otimes u - S) = \rho f,
			\quad \quad \mbox{div}\,u = 0
			\quad \quad \mbox{in}\ (0,\,\infty) \times \Omega, \\[0.5em]
		\cB(u,\,p) = 0
			\quad \quad \mbox{on}\ (0,\,\infty) \times \partial\Omega, \\[0.5em]
		u(0) = u_0
			\quad \quad \mbox{in}\ \Omega
	\end{array}
\end{equation*}
where $\cB = \cB^{\alpha, \beta}$ with $\alpha,\,\beta \in \{\,-1,\,0,\,1\,\}$ realizes one of the boundary conditions (B),
for given data $f \in L_p((0,\,\infty) \times \Omega)$ and $u_0 \in W^{2 - 2/p}_p(\Omega)$,
where we assume $n + 2 < p < \infty$.

First, let $(u^\ast,\,p^\ast)$ be the unique maximal regular solution to the Stokes equations ${(\mbox{S})}^{1, \Omega, \cB}_{f, 0, 0, u_0}$.
Then, $(\bar{u},\,\bar{p}) := (u - u^\ast,\,p - p^\ast)$ is a solution to
\begin{equation*}
	\begin{array}{c}
		\rho \partial_t \bar{u} - \mu \Delta \bar{u} + \nabla \bar{p} = N^\ast(\bar{u}),
			\quad \quad \mbox{div}\,\bar{u} = 0
			\quad \quad \mbox{in}\ (0,\,a) \times \Omega, \\[0.5em]
		\cB(\bar{u},\,\bar{p}) = 0
			\quad \quad \mbox{on}\ (0,\,a) \times \partial \Omega, \\[0.5em]
		\bar{u}(0) = 0
			\quad \quad \mbox{in}\ \Omega
	\end{array}
\end{equation*}
for all $a \in (0,\,1]$, where the nonlinearity on the right-hand side is given as
\begin{equation*}
	N^\ast(\bar{u}) = - ((\bar{u} + u^\ast) \cdot \nabla) (\bar{u} + u^\ast).
\end{equation*}
This is the well-known nonlinear perturbation, which always occurs, if the Navier-Stokes equations are reduced to the Stokes equations.
Since
\begin{equation*}
	\bar{u} \in {}_0 H^1_p((0,\,a),\,L_p(\Omega,\,\bR^n)) \cap L_p((0,\,a),\,H^2_p(\Omega,\,\bR^n))
\end{equation*}
and $p > n + 2$, we may employ the standard estimates of the nonlinearity $N^\ast$ to choose $a \in (0,\,1]$ sufficiently small,
such that existence and uniqueness of a local-in-time solution follows by a contraction mapping argument.
We do not want to repeat these well-known arguments here and close the proof of \Thmref{n}.

\appendix
\section*{Appendix A \\ Parabolic Systems with Prescribed Divergence: $L_p$-Maximal Regularity}\label{sec:parabolic-div-maxreg}\renewcommand{\thesection}{A}
In this appendix, we establish the maximal regularity property of the parabolic problem
\begin{equation*}
	\label{eqn:ins}\tag*{${(\mbox{P})}^{a, \alpha}_{f, h, u_0}$}
	\begin{array}{c}
		\rho \partial_t u - \mu \Delta u = \rho f,
			\quad \quad \mbox{in}\ (0,\,a) \times \Omega, \\[0.5em]
		\cB^\alpha_{\textrm{\tiny div}}(u) = h
			\quad \quad \mbox{on}\ (0,\,a) \times \partial \Omega, \\[0.5em]
		u(0) = u_0
			\quad \quad \mbox{in}\ \Omega
	\end{array}
\end{equation*}
with $\alpha \in \{\,-1,\,0,\,1\,\}$, where the boundary condition is given by the linear operator
\begin{equation*}
	\cB^0_{\textrm{\tiny div}}(u) := P_\Gamma [u] + [\mbox{div}\,u] \nu
\end{equation*}
resp.
\begin{equation*}
	\cB^{\pm 1}_{\textrm{\tiny div}}(u) := \mu P_\Gamma (\nabla u \pm \nabla u^{\sf{T}}) \nu + [\mbox{div}\,u] \nu
\end{equation*}
for $\alpha = \pm 1$.
We assume $a > 0$ and $\Omega \subseteq \bR^n$ to be a halfspace, a bent halfspace or a bounded domain with boundary $\Gamma := \partial \Omega$ of class $C^{3-}$ and $1 < p < \infty$ with $p \neq \frac{3}{2},\,3$.
The claimed maximal regularity property of the system ${(\mbox{P})}^{a, \alpha}_{f, h, u_0}$ is therefore equivalent to the existence of a unique maximal regular solution
\begin{equation*}
	u \in H^1_p((0,\,a),\,L_p(\Omega,\,\bR^n)) \cap L_p((0,\,a),\,H^2_p(\Omega,\,\bR^n)),
\end{equation*}
whenever
\begin{equation*}
	f \in L_p((0,\,a),\,L_p(\Omega,\,\bR^n)), \ h \in L_p((0,\,a),\,L_p(\Gamma,\,\bR^n)) \ \mbox{and} \ u_0 \in W^{2 - 2/p}_p(\Omega,\,\bR^n)
\end{equation*}
with
\begin{equation*}
	\begin{array}{rcll}
		P_\Gamma h & \in & W^{1 - 1/2p}_p((0,\,a),\,L_p(\Gamma,\,T\Gamma)) \cap L_p((0,\,a),\,W^{2 - 1/p}_p(\Gamma,\,T\Gamma)),   & \ \mbox{if} \ \alpha = 0,    \\[0.5em]
		P_\Gamma h & \in & W^{1/2 - 1/2p}_p((0,\,a),\,L_p(\Gamma,\,T\Gamma)) \cap L_p((0,\,a),\,W^{1 - 1/p}_p(\Gamma,\,T\Gamma)), & \ \mbox{if} \ \alpha = \pm 1
	\end{array}
\end{equation*}
and
\begin{equation*}
	Q_\Gamma h \in W^{1/2 - 1/2p}_p((0,\,a),\,L_p(\Gamma,\,N\Gamma)) \cap L_p((0,\,a),\,W^{1 - 1/p}_p(\Gamma,\,N\Gamma))
\end{equation*}
satisfy the compatibility conditions
\begin{equation*}
	\begin{array}{rclll}
		P_\Gamma [u_0]                                          & = & P_\Gamma h(0), & \mbox{if}\ \alpha = 0     & \mbox{and}\ p > \frac{3}{2}, \\[0.5em]
		2 \mu P_\Gamma (\nabla u_0 \pm \nabla u^{\sf{T}}_0) \nu & = & P_\Gamma h(0), & \mbox{if}\ \alpha = \pm 1 & \mbox{and}\ p > 3
	\end{array}
\end{equation*}
and
\begin{equation*}
	[\mbox{div}\,u_0] = h(0) \cdot \nu, \ \mbox{if}\ p > 3.
\end{equation*}
We will prove the claimed maximal regularity property of the system ${(\mbox{P})}^a_{f, h, u_0}$ in the following subsections for the different cases $\alpha = 0$ and $\alpha = \pm 1$.

\subsection{The Case $\boldsymbol{\alpha = 0}$}
We treat the system like the Stokes equations in the previous subsections by a localization procedure and start with the halfspace case $\Omega = \bR^n_+$.
We decompose the desired solution as $u = (v,\,w)$ into a tangential part $v: (0,\,a) \times \bR^n_+ \longrightarrow \bR^{n - 1}$
and a normal part $w: (0,\,a) \times \bR^n_+ \longrightarrow \bR$.
The system ${(\mbox{P})}^{a, 0}_{f, h, u_0}$ then reads
\begin{equation*}
	\begin{array}{c}
		\rho \partial_t v - \mu \Delta v = \rho f_v, \quad
		\rho \partial_t w - \mu \Delta w = \rho f_w,
			\quad \quad \mbox{in}\ (0,\,a) \times \bR^n_+, \\[0.5em]
		[v] = h_v, \quad
		\mbox{div}_x [v] + [\partial_y w] = h_w,
			\quad \quad \mbox{on}\ (0,\,a) \times \partial \bR^n_+, \\[0.5em]
		v(0) = v_0, \quad
		w(0) = w_0
			\quad \quad \mbox{in}\ \bR^n_+,
	\end{array}
\end{equation*}
where we decomposed $f = (f_v,\,f_w)$, $h = (h_v,\,h_w)$ and $u_0 = (v_0,\,w_0)$ into a tangential part and a normal part analogously to the solution $u$.
Moreover, we decomposed the spatial variable into a tangential part $x \in \bR^{n - 1}$ and a normal part $y > 0$.
Of course, $\mbox{div}_x$ denotes the divergence w.\,r.\,t. $x$.
Obviously, this system may first be solved for a unique maximal regular solution $v$, which may then be used as part of the data of the parabolic boundary value problem for $w$.
Hence, $w$ may in a second step be obtained as the unique maximal regular solution to this remaining problem.

Now, we assume $\Omega = \bR^n_\omega$ to be a bent halfspace with a sufficiently flat function $\omega \in BUC^{3-}(\bR^{n - 1})$.
We may first solve the plain Dirichlet problem
\begin{equation*}
	\begin{array}{c}
		\rho \partial_t u - \mu \Delta u = \rho f,
			\quad \quad \mbox{in}\ (0,\,a) \times \Omega, \\[0.5em]
		[u] = P_\Gamma h + (e^{t \Delta_\Gamma}[u_0] \cdot \nu) \nu
			\quad \quad \mbox{on}\ (0,\,a) \times \partial\Omega, \\[0.5em]
		u(0) = u_0
			\quad \quad \mbox{in}\ \Omega,
	\end{array}
\end{equation*}
which forms a system of $n$ decoupled parabolic initial boundary value problems with Dirichlet condition.
Here, $\Delta_\Gamma$ denotes the Laplace-Beltrami operator on $\Gamma$.
Hence, we may in the sequel assume $f = 0$, $P_\Gamma h = 0$ and $u_0 = 0$.
Therefore, we may repeat exactly the same arguments as in \Secref{bent-halfspace} to reduce the problem to the halfspace case.

Finally, we assume $\Omega \subseteq \bR^n$ to be a bounded domain with compact boundary $\Gamma = \partial \Omega$ of class $C^{3-}$.
Since we may solve the above plain Dirichlet problem,
we may again assume $f = 0$, $P_\Gamma h = 0$ and $u_0 = 0$.
Therefore, we may repeat exactly the same arguments as in \Secref{domain} to reduce the problem to the whole space case resp.~the bent halfspace case.
The treatment of the case $\alpha = 0$ is therefore complete.

\subsection{The Case $\boldsymbol{\alpha = \pm 1}$}
Here all boundary conditions are of order one.
Therefore, the system fits into the context of the abstract parabolic problems, which are completely treated in \cite{Denk-Hieber-Pruess:Maximal-Regularity, Denk-Hieber-Pruess:Maximal-Regularity-Inhomogeneous}.
Indeed, the system satisfies all necessary and sufficient conditions, which reveal $L_p$-maximal regularity --
especially the {\itshape Lopatinskii-Shapiro condition}.
For more details we refer to \cite{Denk-Hieber-Pruess:Maximal-Regularity, Denk-Hieber-Pruess:Maximal-Regularity-Inhomogeneous}.

\section*{Appendix B \\ Parabolic Systems with Prescribed Divergence: Special Solutions}\label{sec:parabolic-div-special}
In this appendix, we study the parabolic system
\begin{equation*}
	\tag*{${(\mbox{P})}^{\infty, \bR^n_+, \alpha}_{- \nabla p, 0, 0}$}
	\begin{array}{c}
		\rho \epsilon u + \rho \partial_t u - \mu \Delta u = - \nabla p
			\quad \quad \mbox{in}\ (0,\,\infty) \times \bR^n_+, \\[0.5em]
		\cB^\alpha_{\textrm{\tiny div}}(u) = 0
			\quad \quad \mbox{on}\ (0,\,\infty) \times \partial \bR^n_+, \\[0.5em]
		u(0) = 0
			\quad \quad \mbox{in}\ \bR^n_+
	\end{array}
\end{equation*}
with $\alpha \in \{\,-1,\,0,\,1\,\}$, where the boundary condition is defined as in the previous appendix,
i.\,e. it is given by the linear operator
\begin{equation*}
	\cB^0_{\textrm{\tiny div}}(u) := P_\Gamma [u] + [\mbox{div}\,u] \nu
\end{equation*}
resp.
\begin{equation*}
	\cB^{\pm 1}_{\textrm{\tiny div}}(u) := \mu P_\Gamma (\nabla u \pm \nabla u^{\sf{T}}) \nu + [\mbox{div}\,u] \nu
\end{equation*}
for $\alpha = \pm 1$.
Moreover, the parameter $\epsilon > 0$ denotes an arbitrary shift.
We will limit the following considerations to the special case,
where the right hand side $- \nabla p \in L_p((0,\,\infty),\,L_p(\bR^n_+,\,\bR^n))$ is obtained by either solving the Dirichlet problem
\begin{equation*}
	\tag*{${(\mbox{E}_D)}^{\infty, \bR^n_+}_{h}$}
		- \Delta p = 0 \quad \quad \mbox{in} \ (0,\,\infty) \times \bR^n_+, \qquad 
		[p] = h \quad \quad \mbox{on} \ (0,\,\infty) \times \partial \bR^n_+
\end{equation*}
for a given function $h \in L_p((0,\,\infty),\,\dot{W}^{1 - 1/p}_p(\partial \bR^n_+))$, or the Neumann problem
\begin{equation*}
	\tag*{${(\mbox{E}_N)}^{\infty, \bR^n_+}_{h}$}
		- \Delta p = 0 \quad \quad \mbox{in} \ (0,\,\infty) \times \bR^n_+, \qquad 
		\partial_\nu p = h \quad \quad \mbox{on} \ (0,\,\infty) \times \partial \bR^n_+
\end{equation*}
for a given function $h \in L_p((0,\,\infty),\,\dot{W}^{-1/p}_p(\partial \bR^n_+))$.
Now, Appendix~A ensures the existence of a unique maximal regular solution
\begin{equation*}
	u \in {}_0 H^1_p((0,\,\infty),\,L_p(\bR^n_+,\,\bR^n)) \cap L_p((0,\,\infty),\,H^2_p(\bR^n_+,\,\bR^n)),
\end{equation*}
which satisfies $\mbox{div}\,u = 0$, thanks to the boundary condition.
Splitting the spatial variable into a tangential part $x \in \bR^{n - 1}$ and a normal part $y > 0$ as well as
the solution as $u = (v,\,w)$ into a tangential part $v: [0,\,\infty) \times \bR^n_+ \longrightarrow \bR^{n - 1}$
and a normal part $w: [0,\,\infty) \times \bR^n_+ \longrightarrow \bR$, we obtain by a Laplace transformation in time and a Fourier transformation in the tangential variable
\begin{equation*}
	\begin{array}{rcll}
		\omega^2 \hat{v}(\lambda,\,\xi,\,y) - \mu \partial^2_y \hat{v}(\lambda,\,\xi,\,y) & = & - i \xi \hat{p}(\lambda,\,\xi,\,y),      & \quad \mbox{Re}\,\lambda \geq 0,\,\xi \in \bR^{n - 1},\,y > 0, \\[0.5em]
		\omega^2 \hat{w}(\lambda,\,\xi,\,y) - \mu \partial^2_y \hat{w}(\lambda,\,\xi,\,y) & = & - \partial_y \hat{p}(\lambda,\,\xi,\,y), & \quad \mbox{Re}\,\lambda \geq 0,\,\xi \in \bR^{n - 1},\,y > 0,
	\end{array}
\end{equation*}
where we abbreviated as usual
\begin{equation*}
\lambda_\epsilon := \epsilon + \lambda, \qquad \omega := \sqrt{\rho \lambda_\epsilon + \mu |\xi|^2}.
\end{equation*}
If $\alpha = 0$, then the boundary conditions read
\begin{equation*}
	\begin{array}{rcll}
		                                                     [\hat{v}](\lambda,\,\xi) & = & 0, & \quad \mbox{Re}\,\lambda \geq 0,\,\xi \in \bR^{n - 1}, \\[0.5em]
		i \xi^{\sf{T}} [\hat{v}](\lambda,\,\xi) + [\partial_y \hat{w}](\lambda,\,\xi) & = & 0, & \quad \mbox{Re}\,\lambda \geq 0,\,\xi \in \bR^{n - 1}
	\end{array}
\end{equation*}
and we obtain
\begin{equation*}
	\begin{array}{rcll}
		\hat{v}(\lambda,\,\xi,\,y) = - {\displaystyle{\int^\infty_0}} G_-(\lambda,\,\xi,\,y,\,\eta)\,i \xi \hat{p}(\lambda,\,\xi,\,\eta)\,\mbox{d}\eta,      & \quad \mbox{Re}\,\lambda \geq 0,\,\xi \in \bR^{n - 1},\,y > 0, \\[1.5em]
		\hat{w}(\lambda,\,\xi,\,y) = - {\displaystyle{\int^\infty_0}} G_+(\lambda,\,\xi,\,y,\,\eta)\,\partial_y \hat{p}(\lambda,\,\xi,\,\eta)\,\mbox{d}\eta, & \quad \mbox{Re}\,\lambda \geq 0,\,\xi \in \bR^{n - 1},\,y > 0
	\end{array}
\end{equation*}
where
\begin{equation*}
	G(\lambda,\,\xi,\,y,\,\eta) = \frac{1}{2 \sqrt{\mu} \omega} e^{- \frac{\omega}{\sqrt{\mu}} |y - \eta|}, \quad \mbox{Re}\,\lambda \geq 0,\,\xi \in \bR^{n - 1},\,y \in \bR,\,\eta \in \bR
\end{equation*}
denotes the fundamental solution to the ordinary differential equation
\begin{equation*}
	\omega^2 \hat{\phi}(\lambda,\,\xi,\,y) - \mu \partial^2_y \hat{\phi}(\lambda,\,\xi,\,y) = \hat{f}(\lambda,\,\xi,\,y), \quad \mbox{Re}\,\lambda \geq 0,\,\xi \in \bR^{n - 1},\,y \in \bR
\end{equation*}
and
\begin{equation*}
	\begin{array}{l}
		G_{\pm}(\lambda,\,\xi,\,y,\,\eta) = G(\lambda,\,\xi,\,y,\,\eta) \pm G(\lambda,\,\xi,\,y,\,-\eta), \\[0.5em]
		\qquad \qquad \qquad \qquad \qquad \qquad \mbox{Re}\,\lambda \geq 0,\,\xi \in \bR^{n - 1},\,y > 0,\,\eta > 0.
	\end{array}
\end{equation*}
Hence, if $p$ is obtained as a solution to the Neumann problem ${(\mbox{E}_N)}^{\infty, \bR^n_+}_{h}$, we have
\begin{equation*}
	- \partial_y \hat{p}(\lambda,\,\xi,\,y) = |\xi| e^{- |\xi| y}\,\widehat{{(- \Delta_\Gamma)}^{-1/2} h}(\lambda,\,\xi), \quad \mbox{Re}\,\lambda \geq 0,\,\xi \in \bR^{n - 1},\,y > 0,
\end{equation*}
which implies
\begin{equation*}
	\label{eqn:trace:0-0}\tag{T${}_{0, 0}$}
	\begin{array}{rcl}
		[\hat{w}](\lambda,\,\xi) & = & {\displaystyle{\int^\infty_0}} G_+(\lambda,\,\xi,\,0,\,\eta)\,|\xi| e^{- |\xi| \eta}\,\widehat{{(- \Delta_\Gamma)}^{-1/2} h}(\lambda,\,\xi)\,\mbox{d}\eta \\[1.5em]
                             & = & {\displaystyle{\frac{1}{\omega(\omega + |\zeta|)}}} \hat{h} (\lambda,\,\xi), \quad \mbox{Re}\,\lambda \geq 0,\,\xi \in \bR^{n - 1},
	\end{array}
\end{equation*}
where $\zeta := \sqrt{\mu}\,\xi$, i.\,e. $T^0 [w] = h$ with $T^0$ as defined in \Secref{halfspace}.
If $\alpha = \pm 1$, then the boundary conditions read
\begin{equation*}
	\begin{array}{rcll}
		\mp \mu [\partial_y \hat{v}](\lambda,\,\xi) - \mu \nabla _x [\hat{w}](\lambda,\,\xi) & = & 0, & \quad \mbox{Re}\,\lambda \geq 0,\,\xi \in \bR^{n - 1}, \\[0.5em]
		       i \xi^{\sf{T}} [\hat{v}](\lambda,\,\xi) + [\partial_y \hat{w}](\lambda,\,\xi) & = & 0, & \quad \mbox{Re}\,\lambda \geq 0,\,\xi \in \bR^{n - 1}
	\end{array}
\end{equation*}
and we obtain
\begin{equation*}
	\begin{array}{rcll}
		\hat{v}(\lambda,\,\xi,\,y) = - {\displaystyle{\int^\infty_0}} K^\pm_v(\lambda,\,\xi,\,y,\,\eta)\,i \xi \hat{p}(\lambda,\,\xi,\,\eta)\,\mbox{d}\eta,        & \quad \mbox{Re}\,\lambda \geq 0,\,\xi \in \bR^{n - 1},\,y > 0, \\[1.5em]
		\hat{w}(\lambda,\,\xi,\,y) = - {\displaystyle{\int^\infty_0}} K^\pm_w(\lambda,\,\xi,\,y,\,\eta)\,\partial_y \hat{p}(\lambda,\,\xi,\,\eta)\,\mbox{d}\eta, & \quad \mbox{Re}\,\lambda \geq 0,\,\xi \in \bR^{n - 1},\,y > 0,
	\end{array}
\end{equation*}
where
\begin{equation*}
	\begin{array}{rcl}
		K^\pm_v(\lambda,\,\xi,\,y,\,\eta)
			& := & \left( 1 \pm {\displaystyle{\frac{(\omega + |\zeta|) |\zeta|}{{(\omega + 2 |\zeta|)}^2 \pm |\zeta|^2}}} \right) G_+(\lambda,\,\xi,\,y,\,\eta) \\[1.5em]
			&    & \quad \mp \ {\displaystyle{\frac{(\omega + |\zeta|) |\zeta|}{{(\omega + 2 |\zeta|)}^2 \pm |\zeta|^2}}} G_-(\lambda,\,\xi,\,y,\,\eta)
	\end{array}
\end{equation*}
and
\begin{equation*}
	\begin{array}{rcl}
		K^\pm_w(\lambda,\,\xi,\,y,\,\eta)
			& := & \left( 1 - {\displaystyle{\frac{(\omega + 2 |\zeta|) |\zeta| \mp |\zeta|^2}{{(\omega + 2 |\zeta|)}^2 \pm |\zeta|^2}}} \right) G_+(\lambda,\,\xi,\,y,\,\eta) \\[1.5em]
			&    & \quad + \ {\displaystyle{\frac{(\omega + 2 |\zeta|) |\zeta| \mp |\zeta|^2}{{(\omega + 2 |\zeta|)}^2 \pm |\zeta|^2}}} G_-(\lambda,\,\xi,\,y,\,\eta)
	\end{array}
\end{equation*}
for $\mbox{Re}\,\lambda \geq 0,\,\xi \in \bR^{n - 1},\,y > 0,\,\eta > 0$.
Hence, if $p$ is obtained as a solution to the Neumann problem ${(\mbox{E}_N)}^{\infty, \bR^n_+}_{h}$, we have
\begin{equation*}
	- \partial_y \hat{p}(\lambda,\,\xi,\,y) = |\xi| e^{- |\xi| y}\,\widehat{{(- \Delta_\Gamma)}^{-1/2} h}(\lambda,\,\xi), \quad \mbox{Re}\,\lambda \geq 0,\,\xi \in \bR^{n - 1},\,y > 0,
\end{equation*}
which implies
\begin{equation*}
	\label{eqn:trace:1-0}\tag{T${}_{1, 0}$}
	\begin{array}{rcl}
		[\hat{w}](\lambda,\,\xi) & = & {\displaystyle{\int^\infty_0}} K^{\pm}_w(\lambda,\,\xi,\,0,\,\eta)\,|\xi| e^{- |\xi| \eta}\,\widehat{{(- \Delta_\Gamma)}^{-1/2} h}(\lambda,\,\xi)\,\mbox{d}\eta \\[1.5em]
                             & = & {\displaystyle{\frac{1}{\omega^2 \pm |\zeta|^2}}} \hat{h} (\lambda,\,\xi), \quad \mbox{Re}\,\lambda \geq 0,\,\xi \in \bR^{n - 1},
	\end{array}
\end{equation*}
i.\,e. $T^{\pm 1}[w] = h$ with $T^{\pm 1}$ as defined in \Secref{halfspace}.

On the other hand, if $p$ is obtained as a solution to the Dirichlet problem ${(\mbox{E}_D)}^{\infty, \bR^n_+}_{h}$, we have
\begin{equation*}
	- \partial_y \hat{p}(\lambda,\,\xi,\,y) = |\xi| e^{- |\xi| y}\,\hat{h}(\lambda,\,\xi), \quad \mbox{Re}\,\lambda \geq 0,\,\xi \in \bR^{n - 1},\,y > 0,
\end{equation*}
which implies
\begin{equation*}
	\label{eqn:trace:1-1}\tag{T${}_{1, 1}$}
	\begin{array}{l}
		- 2 \mu [\partial_y w] + [p] \\[0.5em]
		\qquad = - \ 2 \mu {\displaystyle{\int^\infty_0}} (\partial_y K^{\pm}_w)(\lambda,\,\xi,\,0,\,\eta)\,|\xi| e^{- |\xi| \eta}\,\hat{h}(\lambda,\,\xi)\,\mbox{d}\eta + \hat{h}(\lambda,\,\xi) \\[1.5em]
    \qquad = {\displaystyle{\frac{(\omega^2 \mp |\zeta|^2) + 2 \frac{\omega}{\omega + |\zeta|} (|\zeta|^2 \pm |\zeta|^2)}{\omega^2 \pm |\zeta|^2}}} \hat{h} (\lambda,\,\xi), \quad \mbox{Re}\,\lambda \geq 0,\,\xi \in \bR^{n - 1},
	\end{array}
\end{equation*}
i.\,e. $S^{\pm 1} (- 2 \mu [\partial_y w] + [p]) = h$ with $S^{\pm 1}$ as defined in \Secref{halfspace}.

\section*{Acknowledgements}
M.\,K.~gratefully acknowledges financial support by the Deutsche For\-schungs\-gemeinschaft within the International Research Training Group ``Mathematical Fluid Dynamics'' (IRTG 1529).

\bibliographystyle{plain}
\bibliography{references}
\end{document}